\documentclass{amsart}
\usepackage[utf8]{inputenc}
\usepackage{amsmath}
\usepackage{tikz}
\usetikzlibrary{patterns}
\usepackage{amssymb}
\usepackage{stmaryrd}
\usepackage{mathtools}

\usepackage{hyperref}
\hypersetup{
    colorlinks,
    linkcolor={red!50!black},
    citecolor={blue!50!black},
    urlcolor={blue!80!black}
}

\usepackage{subcaption}
\usepackage{bm}
\newcommand{\fclass}[1]{{\llbracket #1 \rrbracket}}

\newcommand{\CLip}{C^\mathrm{Lip}}
\newcommand{\Ccell}{C^\mathrm{cell}}

\newcommand{\cF}{\mathcal{F}}

\newcommand{\cR}{\mathcal{R}}
\newcommand{\cT}{\mathcal{T}}
\newcommand{\cU}{\mathcal{U}}

\newcommand{\ba}{\mathbf{a}}
\newcommand{\bb}{\mathbf{b}}
\newcommand{\bc}{\mathbf{c}}
\newcommand{\bd}{\mathbf{d}}
\newcommand{\bD}{\mathbf{D}}
\newcommand{\be}{\mathbf{e}}
\newcommand{\bE}{\mathbf{E}}

\newcommand{\br}{\mathbf{r}}

\newcommand{\bR}{\mathbf{R}}

\newcommand{\bz}{\mathbf{z}}
\newcommand{\frakg}{\mathfrak{g}}
\newcommand{\frakh}{\mathfrak{h}}
\newcommand{\frakn}{\mathfrak{n}}
\newcommand{\frakl}{\mathfrak{l}}
\newcommand{\sansA}{\mathsf{A}}
\newcommand{\sansB}{\mathsf{B}}

\newcommand{\sansC}{\mathsf{C}}
\newcommand{\sansD}{\mathsf{D}}
\newcommand{\sansP}{\mathsf{P}}
\newcommand{\one}{\mathbf{1}}

\newcommand{\R}{\mathbb{R}}

\newcommand{\Z}{\mathbb{Z}}

\newcommand{\N}{\mathbb{N}}
\newcommand{\Q}{\mathbb{Q}}

\newcommand{\Prob}{\mathbb{P}}
\newcommand{\EE}{\mathbb{E}}

\newcommand{\from}{\colon}
\newcommand{\cSc}{\ensuremath{\mathcal{S}}}

\newcommand{\FD}{X_0}
\newcommand{\chxop}{{\mathbf{X}_\infty^{\mathsf{op}}}}
\newcommand{\chx}{{\mathbf{X}_{\infty}}}

\newcommand{\blog}{\operatorname{\overline{log}}}

\newcommand{\dimAN}{\dim_{\mathrm{AN}}}
\newcommand{\dGamma}{d_{[\Gamma]}}
\newcommand{\Td}{\mathrm{T}d}

\DeclareMathOperator{\area}{area}
\DeclareMathOperator{\Ad}{Ad}
\DeclareMathOperator{\Env}{Env}
\DeclareMathOperator{\rank}{rank}
\DeclareMathOperator{\Rrank}{\mathbb{R}--rank}

\DeclareMathOperator{\nbhd}{nbhd}

\DeclareMathOperator{\FV}{FV}

\DeclareMathOperator{\Lip}{Lip}

\DeclareMathOperator{\SL}{SL}

\DeclareMathOperator{\Stab}{Stab}

\DeclareMathOperator{\diam}{diam}

\DeclareMathOperator{\id}{id}
\DeclareMathOperator{\inter}{int}
\DeclareMathOperator{\mass}{mass}
\DeclareMathOperator{\supp}{supp}

\DeclareMathOperator{\vol}{vol}
\newtheorem{thm}{Theorem}[section]

\newtheorem{lemma}[thm]{Lemma}
\newtheorem{prop}[thm]{Proposition}
\newtheorem{cor}[thm]{Corollary}
\theoremstyle{remark}
\newtheorem{defn}[thm]{Definition}
\newtheorem*{remark}{Remark}
\newtheorem*{ack}{Acknowledgments}

\begin{document}
\title{Filling functions of arithmetic groups}
\author{Enrico Leuzinger\and Robert Young}
\begin{abstract}
  The Dehn function and its higher-dimensional generalizations measure the difficulty of filling a sphere in a space by a ball.  In nonpositively curved spaces, one can construct fillings using geodesics, but fillings become more complicated in subsets of nonpositively curved spaces, such as lattices in symmetric spaces.  In this paper, we prove sharp filling inequalities for (arithmetic) lattices in higher rank semisimple Lie groups.  When $n$ is less than the rank of the associated symmetric space, we show that the $n$--dimensional filling volume function of the lattice grows at the same rate as that of the associated symmetric space, and when $n$ is equal to the rank, we show that the $n$--dimensional filling volume function grows exponentially.  This broadly generalizes a theorem of Lubotzky--Mozes--Raghunathan on length distortion in lattices and confirms conjectures of Thurston, Gromov, and Bux--Wortman.
\end{abstract}

\maketitle

\tableofcontents


\section{Introduction and main results} 
The Dehn function $\delta_G$ of a finitely presented group $G$ measures the complexity of the word problem in $G$.  This can be interpreted combinatorially or geometrically.  Combinatorially, when $G$ is a group equipped with a finite presentation, we define $\delta_G(\ell)$ to be the maximum number of applications of relators necessary to reduce a word of length $\ell$ that represents the identity to the trivial word.  Geometrically, when $X$ is a simply-connected manifold or simplicial complex, we define $\delta_X(\ell)$ to be the maximum area necessary to fill a closed curve of length $\ell$ by a disc.  If $G$ acts geometrically (cocompactly, properly discontinuously, and by isometries) on $X$, then the combinatorial Dehn function $\delta_G$ of $G$ and the geometric Dehn function $\delta_X$ of $X$ have the same asymptotic growth rate.  (See Section~\ref{sec:fillingDefs} for precise definitions.) 

Filling volume functions generalize the Dehn function to higher dimensions.  If $X$ is an $(n-1)$--connected metric space, the $n$--dimensional filling volume function $\FV_X^n$ measures the difficulty of filling $(n-1)$--cycles in $X$ by $n$--chains.  This is harder to interpret in terms of group theory, but it yields a quasi-isometry invariant in the sense that if $X$ and $Y$ are quasi-isometric, highly connected, and have bounded geometry (for instance, if they support a cocompact group action), then $\FV^n_X$ and $\FV^n_Y$ have the same asymptotic growth rate.  Consequently, it gives rise to a group invariant; when a group $G$ acts geometrically on $X$, we let $\FV_G^n=\FV_X^n$.  This depends on the choice of $X$, but its asymptotic growth rate is well-defined.

In this paper, we will compute sharp bounds on the Dehn function and higher-dimensional filling volume functions of (irreducible) lattices in higher rank semisimple Lie groups. Note that according to the arithmeticity theorem of Margulis such lattices are arithmetic.

In the case that the lattice $\Gamma\subset G$ is uniform in $G$, the filling functions are known.  In fact a uniform lattice acts geometrically on the corresponding symmetric space $X=G/K$ of noncompact type, and the filling functions of such spaces were computed in \cite{LeuISO}; they are euclidean up to the rank: $\FV_X^n(V)\approx V^{\frac{n}{n-1}}$ for $n\leq \textup{rank} X$. More generally, for nonpositively curved spaces one has $\delta_X(L)\lesssim L^2$ and $\FV_X^n(V)\lesssim V^{\frac{n}{n-1}}$ (for all $n$)\cite{GroFRM, WengerShort}.

If the lattice $\Gamma$ is nonuniform, such bounds are more difficult to prove.  One reason is that when $\Gamma$ is a nonuniform lattice, the quotient $Y:=\Gamma\backslash X$ is noncompact and can be partitioned into a compact \emph{thick part} and a \emph{thin part} (a set of ``cusps'').  Reduction theory for arithmetic groups shows that the thick part of $Y$ lifts to a contractible invariant subset $X_0\subset X$ which is the complement of a union of horoballs.  This set is quasi-isometric to $\Gamma$, so the filling volume functions of $\Gamma$ measure the difficulty of filling cycles in $X_0$ by chains in $X_0$, that is, chains that avoid these horoballs. 

Thurston conjectured that the Dehn function of $\Gamma=\SL_{k+1}(\Z)$ for $k\ge 3$ is quadratic \cite{GerstenSurv}.  This has been proved for $k\ge 4$ \cite{YoungQuad}.  Gromov extended this conjecture to higher-dimensional filling problems and arbitrary arithmetic lattices \cite{GroAII}. Our first main results confirm these conjectures.
 
\begin{thm}\label{thm:mainThmDehn}
  Let $\Gamma$ be an irreducible nonuniform lattice in a connected, center-free semisimple Lie group $G$ without compact factors.  Let $k=\Rrank(G)$ and suppose that $k\geq 3$.  Then the Dehn function of $\Gamma$ is quadratic: $\delta_{\Gamma}(L)\approx L^2$ for all $L\ge 1$.
\end{thm}
\begin{thm}\label{thm:mainThmUpper}
  With $\Gamma$ and $k$ as in Theorem~\ref{thm:mainThmDehn}, for all $2\le n<k$ and all $V\ge 1$, we have
  \begin{equation}\label{eq:mainThmUpperBound}
    \FV^n_{\Gamma}(V)\approx V^{\frac{n}{n-1}}.
  \end{equation}
\end{thm}

Note that Theorem~\ref{thm:mainThmUpper} holds only in dimensions below the rank of $\Gamma$.  It has been conjectured that the rank is a critical dimension for the isoperimetric behavior and that filling functions in the dimension of the rank grow exponentially \cite{LeuzPitRk2}; this has been shown for $\SL_{k+1}(\Z)$ \cite{ECHLPT}, for nonuniform lattices in semisimple groups of $\mathbb R$-rank $2$ \cite{LeuzPitRk2}, and for lattices of relative $\Q$--type $A_n, B_n, C_n, D_n, E_6$, or $E_7$ \cite{WortExpLower}.  Conversely, Gromov showed that filling functions of lattices in linear groups are at most exponential \cite{GroAII}.  Our next result confirms the conjecture in general.

\begin{thm}\label{thm:mainThmLower}
  With $\Gamma$ and $k$ as in Theorem~\ref{thm:mainThmDehn}, there is a $c>0$ such that for $V\ge 1$, 
  \begin{equation}\label{eq:mainThmLowerBound}
    \FV^{k}_{\Gamma}(V)\gtrsim \exp(cV^{\frac{1}{k-1}}).
  \end{equation}
\end{thm}

A broader conjecture along these lines was proposed by Bux and Wortman, based on the distortion of filling volumes.  We say that $X_0\subset X$ is \emph{undistorted up to dimension $n$} if there is a $C>0$ such that for any $m<n$ and any $m$--cycle $\alpha$ in $X_0$, we have 
$$\FV_{X_0}(\alpha)\le C \FV_{X}(\alpha)+C\mass \alpha+C.$$ 
Uniform lattices are undistorted in all dimensions.  
Bux and Wortman conjectured that $S$-arithmetic groups (defined over number fields or function fields) acting on products of symmetric spaces and Euclidean  buildings  are undistorted in dimensions below the  geometric (or total) rank \cite{BuxWortFiniteness}. 
This is a strong generalization of a theorem of Lubotzky, Mozes, and Raghunathan on the distance distortion of lattices \cite{LMR}.  
Finiteness properties and filling invariants of $S$-arithmetic groups have been studied in papers including \cite{BuxWortFiniteness, BuxKohlWitzel, YoungHigherSol, BestvinaEskinWortman}.
   The following theorem confirms the Bux-Wortman conjecture in the case of 
nonuniform arithmetic groups defined over number fields.

\begin{thm}\label{thm:disdim}
  If $X_0$ is as above, then $X_0$ is undistorted up to dimension $k-1$, but not up to dimension $k$.
\end{thm}

Note that undistortedness up to dimension $1$ is equivalent to $X_0$ being quasi-isometrically embedded in $X$, so Theorem~\ref{thm:disdim} gives a new proof of the theorem of Lubotzky--Mozes--Raghunathan.

Most previous bounds on filling invariants of arithmetic lattices involve explicit constructions of cycles and chains in some thick part of $X$.  The first such bound used an arithmetic construction to produce subgroups $\Z^{n-2}\subset \SL_n(\Z)$ that act on flats contained in the thick part \cite{ECHLPT}.  Pieces of these flats can be glued together to produce $(n-2)$--cycles with exponentially large filling volume.  Similar constructions were used in \cite{LeuzPitRk2} and \cite{WortExpLower} to find exponential lower bounds in other groups.  Upper bounds on filling invariants of arithmetic lattices and solvable groups have been found in \cite{DrutuFilling, YoungQuad, CohenSP, LeYoRank1, BestvinaEskinWortman}, and \cite{CorTessSolv}. These bounds typically combine explicit constructions of chains that fill cycles of a particular form with ways to decompose arbitrary cycles into pieces of that form. 

The bounds in this paper are based on the probabilistic method rather than explicit constructions.  Instead of constructing a single filling, we show that when $n<k$ and $\alpha$ is an $(n-1)$--cycle in $X_0$ of mass $V$, there is a large family of $n$--chains in $X$ with boundary $\alpha$ and mass at most $V^{\frac{n}{n-1}}$.  We prove our bounds by considering a random chain $\beta$ drawn from this family.

Each chain in this family is made of pieces of flats, so the geometry of a random chain $\beta$ depends on the behavior of random flats.  Kleinbock and Margulis \cite{KMLog} showed that random flats in $Y$ typically spend most of their time in the thick part of $Y$ (see Figure~\ref{fig:thickHoroballs}); indeed, the thin part of $Y$ has exponentially small volume, so its intersection with a ``typical'' flat $E$ is exponentially small.  Similar equidistribution results hold even when $E$ is drawn from a narrower distribution, for instance, a random flat $E_y$ that passes through a given point $y\in Y$.  In this case, if $y$ lies deep in the thin part, then $E_y$ must intersect the thin part, but we will see that except for a ball around $y$, most of $E_y$ typically lies in the thick part of $Y$ (Figure~\ref{fig:thinHoroballs}).  Consequently, $\beta$ typically does not lie in $X_0$, but with high probability, it lies close to $X_0$.  With some additional work, we can retract it to $X_0$ and obtain the desired bounds.

\begin{figure}
  \begin{minipage}[t]{.475\linewidth}
    \centering
    \includegraphics[width=.9\textwidth]{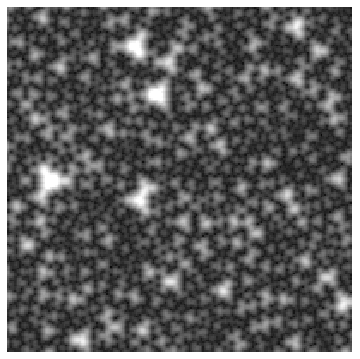} 
    \subcaption{A flat centered in the thick part.  Note that the flat lies close to the thick part except for a few small pieces.}\label{fig:thickHoroballs}
  \end{minipage}\hfill%
  \begin{minipage}[t]{.475\linewidth}
    \centering
    \includegraphics[width=.9\textwidth]{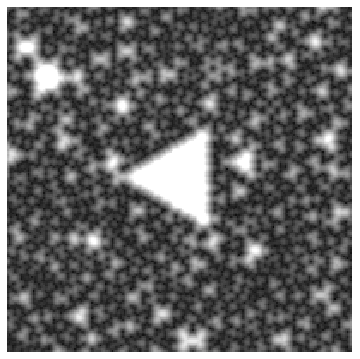} 
    \subcaption{A flat centered in the cusp.  The triangular symmetry of the root system can be seen in the shape of the intersection with the cusp. \label{fig:thinHoroballs}}
  \end{minipage}
  \caption{Generic flats in $\SL_3(\Z)$.  Dark parts of the flat lie in the thick part of $\SL_3(\Z)\backslash \SL_3(\R)$, and light parts lie in the cusp.}\label{fig:1}
\end{figure}

In \cite{athreyaULTRA} a $p$-adic version of the logarithmic law of Kleinbock and Margulis is established.  We speculate that our methods can also be used to prove analogues of Theorem~\ref{thm:mainThmDehn} and Theorem~\ref{thm:mainThmUpper} for $S$--arithmetic groups (at least in characteristic zero), but this paper will only address the arithmetic case.

\begin{ack}
  This material is based upon work supported by the National Science Foundation under Grant No.\ DMS 1612061.  R.Y.\ was supported by a Sloan Fellowship and the Fall 2016 program on Geometric Group Theory at the Mathematical Sciences Research Institute.  This project began during a visit to ETH Zürich, and we would like to thank ETH Zürich and Urs Lang for their hospitality during our visit.  
\end{ack}

\subsection{Sketch of proof}\label{sec:sketch}

The constructions in this paper follow the same broad outline as the constructions in \cite{LeYoRank1}.  As in that paper, we build fillings of cycles by gluing together large simplices.  These simplices are constructed as part of a continuous map $\Omega_\bR\from S\to X$, where $S=\Delta_{\FD}^{(k-1)}$ is the $(k-1)$--skeleton of the infinite-dimensional simplex with vertex set equal to the thick part $\FD$.  Any set $V\subset\FD$ of at most $k$ points determines a simplex $\langle V\rangle$ of $S$ and a simplex $\Omega_\bR|_{\langle V\rangle}$ in $X$.  If $V,V'$ are two such sets, then $\Omega_\bR|_{\langle V\rangle}$ and $\Omega_\bR|_{\langle V'\rangle}$ agree on the intersection $\langle V\cap V'\rangle$.  This makes it possible to build complicated fillings out of these simplices.

There are two main differences between the constructions in this paper and those in \cite{LeYoRank1}.  First, instead of constructing a single map $\Omega$, we construct a family of maps $\Omega_\bD$, parametrized by a certain tuple $\bD$ of chambers in the geodesic boundary $X_\infty$, then construct a random variable $\bR$ which takes values in the set of such tuples.  Then $\Omega_\bR$ is a random map, and we can analyze it using dynamical results of Kleinbock and Margulis.  These results show that $\Omega_\bR$ is typically a logarithmic distance from $\FD$, so we can use $\Omega_\bR$ to produce fillings that lie in a logarithmic neighborhood of $X_0$.

Second, since these fillings are logarithmically far from $\FD$, we need a new argument to retract them to $\FD$.  This is a two-step process.  First, we apply a retraction $X\to \FD$ with an exponentially growing Lipschitz constant to produce fillings in $\FD$; these fillings have polynomial volume, so $X_0$ has polynomial filling functions below the rank.  Second, we use the polynomial bound on the filling functions to construct a new retraction from $X\to X_0$.  This new retraction satisfies better bounds than the exponential retraction, and when we apply it to the fillings produced from $\Omega_\bR$, we get sharp bounds on $\FV^n_{\Gamma}$.  We call this the bootstrapping argument.

As an illustration of this technique, we consider the dimension--1 case, for which Theorem~\ref{thm:disdim} reduces to the theorem of Lubotzky--Mozes--Raghunathan that the inclusion of $\Gamma$ into $G$ is a quasi-isometric embedding, or equivalently, that $\FD$ is undistorted in $X$.  Let $x$ and $y$ be two elements of $\FD$ and let $2r=d_X(x,y)$.  Let $\gamma\from \R \to X$ be the infinite geodesic connecting $x$ and $y$, parametrized so that $\gamma(-r)=x$ and $\gamma(r)=y$.  Let $m=\gamma(0)$ be the midpoint of $x$ and $y$.  There is a flat $F$ containing $x$, $y$, and $\gamma$; let $\bc_x$ and $\bc_y$ be opposite chambers of $F_\infty$ such that $\lim_{t\to -\infty}\gamma (t)\in \bc_x$ and $\lim_{t\to \infty}\gamma(t)\in \bc_y$.

We would like to connect $x$ and $y$ by a path in a flat that lies close to $\FD$.  The intersection $F\cap \FD$ may be disconnected, so we cannot necessarily connect $x$ and $y$ by a path in $F$, but we will connect $x$ and $y$ in a flat $H$ that lies close to $F$.  Let $\br_x$ and $\br_y$ be independent random chambers of $X_\infty$ that are close to $\bc_x$ and $\bc_y$.  Then the chambers $\br_x$ and $\br_y$ are opposite, so there is a unique flat $E_{\br_x,\br_y}$ such that  $\br_x$ and $\br_y$ are chambers of $(E_{\br_x,\br_y})_\infty$, and it means that $E_{\br_x,\br_y}$ is close to $m$, $x$, and $y$.  Let $x', y'\in E_{\br_x,\br_y}$ be points that are close to $x$ and $y$.  

We cannot expect all of $E_{\br_x,\br_y}$ to be close to $\FD$.  The flat $E_{\br_x,\br_y}$ passes close to $m$, and $m$ might be far from $\FD$.  We can, however, prove equidistribution results for all of $E_{\br_x,\br_y}$ except for a ball around $m$.  

To state these results, we introduce a version of the exponential map.  For any $v\in X$, let $e_v\from X_\infty\times [0,\infty)\to X$ be the map such that for any $\sigma\in X_\infty$, the map $t\mapsto e_v(\sigma,t)$ is the unit-speed geodesic ray from $v$ toward $\sigma$.  Since $E_{\br_x,\br_y}$ is close to $m$, the image $e_m((E_{\br_x,\br_y})_\infty\times [0,\infty))$ is close to $E_{\br_x,\br_y}$.  We can use results of Kleinbock and Margulis to show that there are $b>0$ and $R_0>0$ with $R_0\approx r$ such that for any point $\sigma\in (E_{\br_x,\br_y})_\infty$ and any $R>R_0$, 
\begin{equation}\label{eq:expMomentsSketch}
  \EE[\exp (b\dGamma(e_m(\sigma,R)))]\lesssim 1,
\end{equation}
where $\dGamma(v)=d_X([\Gamma],v)$ denotes the distance to the $\Gamma$--orbit $[\Gamma]=\Gamma K\subset X=G/K$.  That is, the probability that $e_m(\sigma,R)$ is distance $l$ from $\FD$ decays exponentially with $l$.  (Since $E_{\br_x,\br_y}$ is random, it is a little ambiguous to say that $\sigma$ is a point in $(E_{\br_x,\br_y})_\infty$; see Sec.~\ref{sec:equidistribution} for a more rigorous statement.)

A similar application of Kleinbock--Margulis shows that the Weyl chambers $e_{x}(\br_x\times [0,\infty))$ and $e_{y}(\br_y\times [0,\infty))$ also lie close to $\FD$.  These Weyl chambers lie close to Weyl chambers in $E_{\br_x,\br_y}$ based at $x'$ and $y'$, so $E_{\br_x,\br_y}$ lies close to $\FD$ on a set shaped like the one in Figure~\ref{fig:equiflat}.  When $k\ge 2$, this set is connected, so we can connect $x'$ and $y'$ by a path $\Omega$ that lies in this region.  Each point in $\Omega$ lies close to $\FD$ with high probability.  

\begin{figure}
  \begin{tikzpicture}[scale=.75]
    \fill[color=gray!30] (-3,-3) rectangle (3,3);
    \fill[fill=white] (0,0) circle (2);
    \filldraw[draw=black, fill=gray!30] (3,1) -- (1,0) -- (3, -1);
    \filldraw[draw=black, fill=gray!30]  (-3,1) -- (-1,0) -- (-3, -1); 
    \draw (0,0) circle (2);
    \node at (-1,0) [below=.12cm] {$x'$}; 
    \node at (1,0) [below=.12cm] {$y'$};
    \draw[draw=black, very thick]  (-1,0) -- (-2.5,0) arc (180:0:2.5) node [below] {$\Omega$}-- (1,0); 
  \end{tikzpicture} 
  \caption{\label{fig:equiflat} Regions of the random flat $E_{\br_x,\br_y}$ that lie close to $\FD$.  When $k\ge 2$, the union of these regions is connected, so we can connect $x'$ and $y'$ by a curve $\Omega$.}
\end{figure}
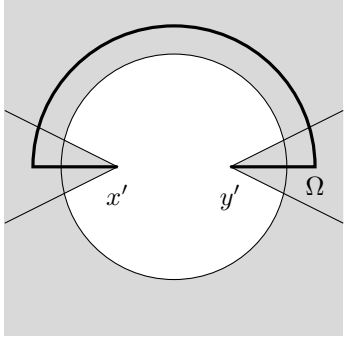

Unfortunately, though the expected distance from $\FD$ to each point in this path is bounded, $\Omega$ has length on the order of $r$, so the expectation of the \emph{maximum} distance from $\FD$ to $\Omega$ is on the order of $\log r$.  That is, letting $X(l)=\dGamma^{-1}([0,l])$, there is an $\eta>0$ such that $\Omega\subset X(\eta \log r)$ with positive probability.  

In order to bound the distortion of $\FD$ in $X$, we need to retract $\Omega$ to $\FD$.  By results of \cite{LeuExh}, there is a $c>0$ such that the closest-point projection $\rho \from X \to X_0$ satisfies $\log \Lip(\rho|_{X(l)})\le c(l+1)$, i.e., 
$$\Lip(\rho|_{X(\eta \log r)})\lesssim r^{c\eta}.$$
The composition $\rho\circ \Omega$ is a curve in $\FD$, and with positive probability, $\ell(\rho\circ \Omega)\lesssim r^{c\eta+1}$.  Therefore, $\FD$ is at most polynomially distorted in $X$.  

To get the sharp bound, we apply the polynomial distortion bound to $\Omega$.  We call this the bootstrapping argument, since it uses the polynomial bound to sharpen itself.  Suppose that $N$ is an integer, that $N\approx r$, and that $\Omega\from [0,N]\to X$ has unit speed.  Let $v_i=\rho(\Omega(i))\in \FD$.  Then 
$$d_X(v_i,v_{i+1})\le d(\Omega(i),\FD)+1+d(\Omega(i+1),\FD)\le 2d(\Omega(i),\FD)+2.$$
Since $\FD$ has polynomial distortion, 
$$d_{\FD}(v_i,v_{i+1})\lesssim (d(\Omega(i),\FD)+1)^{c\eta+1},$$
and there is a path $\omega\from [0,N]\to \FD$, such that $\omega(i)=v_i$ for all $i$ and
$$\ell(\omega)\lesssim \sum_{i=0}^{N-1}(d(v_i,\FD)+1)^{c\eta+1}.$$
The probability that $d(v_i,\FD)>l$ decays exponentially with $l$, so $\EE[d(v_i,\FD)^{c\eta+1}]$ is bounded, and $\ell(\omega)\lesssim N\approx r$.  This is a curve in $\FD$ whose endpoints are a bounded distance from $x$ and $y$, so $d_{\FD}(x,y)\lesssim r+1$, and $\FD$ is undistorted in $X$.

The higher-dimensional case is more complicated, but follows a similar strategy.  Instead of constructing a path in a single random flat, we instead construct a surface in a collection of random flats.  Let $\bD=(\bd_\delta)_{\delta\in \cF(S)}\in (\chx)^{\cF(S)}$ be a tuple of chambers of $X_\infty$, indexed by the simplices of $S$.  We define $\Omega_\bD$ by induction on dimension.  We send each vertex $x\in \cF^0(S)=\FD$ to the corresponding point in $\FD$.  For each edge $e=\langle v,w\rangle$, we will construct $\Omega_\bD(e)$ so that it lies close to the two flats $E_{\bd_v,\bd_e}$ and $E_{\bd_w,\bd_e}$.  When we replace $\bD$ by a random variable $\bR=(\br_\delta)_{\delta\in \cF(S)}$, these become random flats, and, as above, there is a compact set $B$ such that any point in $E_{\br_v,\br_e}\cup E_{\br_w,\br_e}\setminus B$ is close to $\FD$ with high probability.  We choose $\Omega_\bD(e)$ as a path in a neighborhood of $E_{\bd_v,\bd_e}\cup E_{\bd_w,\bd_e}$ that avoids $B$.

Then for any triangle $\delta\in \cF^2(S)$, the image $\Omega_\bD(\partial \delta)$ lies in a neighborhood of the union of six flats.  We denote the boundary at infinity of these flats by $M_{\partial \delta}(\bD)$, and construct $\Omega_\bD(\delta)$ as a surface lying in a neighborhood of the union
$$M_{\delta}(\bD):=\bigcup_{\bb\in \cF^{k-1}(M_{\partial \delta}(\bD))} E_{\bb,\bd_\delta}.$$
This is also a union of boundedly many flats that satisfy equidistribution bounds when $\bD$ is replaced by $\bR$, so we can repeat the process to extend $\Omega_\bD$ to all of $S$.  

This results in a $b>0$ and a random map $\Omega_\bR\from S\to X$ such that for every $s\in S$, 
$$\EE[\exp (b\dGamma(\Omega_\bR(s)))]\lesssim 1.$$
We use this random map to show that there is an $\eta>0$ such that if $n\le k-1$ and $\alpha\in \CLip_{n-1}(\FD)$ is a cycle in $\FD$ of mass at most $V$, then $\alpha$ bounds a chain $\beta_0\in \CLip_{n}(X(\eta \log V))$ with mass at most $V^{\frac{n}{n-1}}$.  This is logarithmically far from $\FD$, but the image of $\beta_0$ under closest-point projection $\rho$ is a polynomial filling for $\alpha$, so $\FD$ satisfies a polynomial filling inequality.  We use this polynomial filling inequality in a bootstrapping argument to prove the sharp filling inequality.  

\subsection{Outline of paper}

In Section~\ref{sec:prelims}, we review some notation and definitions that will be used in the rest of the paper.  This includes smooth random variables and smooth random chambers, the thick part $X_0$ of a symmetric space $X$, and a definition of the filling volume function $\FV$ based on Lipschitz chains.  

In Section~\ref{sec:equidistribution}, we prove bounds on the distribution of random flats and random Weyl chambers in $X$.  These are based on work of Kleinbock and Margulis on logarithm laws, and lead to the inequality \eqref{eq:expMomentsSketch} used in the sketch above.  In Section~\ref{sec:lower bound}, we use these results to prove that $\FV^{k}_{\Gamma}$ grows faster than any polynomial when $k=\rank X$.  We prove this bound by constructing a $(k-1)$--sphere that lies in a random flat.  The center of this sphere lies deep in the thin part, so any filling that avoids the thin part has exponentially large volume.  Unfortunately, the sphere itself may have pieces that are logarithmically far away from the thick part $\FD$, and retracting the sphere to the thick part increases its volume by a polynomial factor.  Later, we will fix this by applying the bootstrapping argument alluded to above.

In Sections~\ref{sec:parametrizedCones}--\ref{sec:randomMapsX}, we construct the family of maps $\Omega_\bD$ and the random map $\Omega_\bR$.  We construct $\Omega_\bD$ by defining functions $f\from S\to X$ and $r\from S\to \R$ and a family of functions $\sansP_\bD\from S\to X_\infty$.  In Section~\ref{sec:parametrizedCones}, we construct $\sansP_\bD$ and show that it varies smoothly with $\bD$.  Next, in Section~\ref{sec:random cones}, we construct the random variable $\bR$ and prove properties of $\sansP_\bR$.  Finally, in Section~\ref{sec:randomMapsX}, we construct $f$ and $r$ and describe $\Omega_\bR$.  

This allows us to prove Theorems~\ref{thm:mainThmDehn} and \ref{thm:mainThmLower} in Section~\ref{sec:DiscsSpheres}.  The bounds in these theorems rely on two-dimensional fillings or round spheres in $X$, so they follow from a simpler version of the bootstrapping argument.  

In higher dimensions, complications may arise, and we need the full version of the bootstrapping argument.  For example, a closed curve in $X$ has diameter at most equal to its length, while a $2$--sphere may consist of two spheres connected by a long, skinny tube, and may have diameter much larger than its area.  We will need some additional tools from geometric measure theory to handle spheres like this.  In Section~\ref{sec:GMT tools}, we introduce these tools, and in Section~\ref{sec:higherDims}, we use them to complete the proofs of Theorem~\ref{thm:mainThmUpper} and Theorem~\ref{thm:disdim}.  

\section{Preliminaries and notation}\label{sec:prelims}

\subsection{Notation and standing assumptions}\label{sec:standing}


If $Z$ is a simplicial or CW complex and $d\geq 0$, we denote by $\mathcal F(Z)$ the set of faces (or cells) of $Z$ and by ${\mathcal F}^d(Z)$ the set of $d$-dimensional faces (or cells) of $Z$.

Throughout this paper, $G$ will be a connected, semisimple, center-free Lie group without compact factors and $K$ will be a maximal compact subgroup.  We denote the Lie algebra of $G$ by $\frakg$ and equip it with an $\Ad(K)$--invariant norm.  Let $X=G/K$ be the corresponding symmetric space of noncompact type and let $d$ be its distance function.  For all $g\in G$, we denote the point $gK\in X$ by $[g]$.  We let $k=\rank X= \Rrank G$.  Let $\Gamma$ be an irreducible nonuniform lattice in $G$ and let $\dGamma \from X\to \R$ be the function $\dGamma(x)=d_X(x,[\Gamma])$ for all $x\in X$; note that if $x=[h]$ for some $h\in G$, then 
$$|d_{[\Gamma]}(h)-d_{\Gamma}(h)|=|d_X([h],[\Gamma])-d_G(h,\Gamma)|\le 2\diam K\approx 1.$$

Let $X_\infty$ denote the geodesic boundary of $X$ at infinity, equipped with the Tits metric $\Td$ associated to the angular metric $\angle$. In particular, $(X_\infty, \Td)$ is the geometric realization of a spherical Tits building.  Let $\chx=\mathcal F^{k-1}(X_{\infty})$ be its set of (maximal) chambers.  If $\bb,\bc\in \chx$ are two opposite chambers, the convex hull of $\bb$ and $\bc$ is an apartment of $X_{\infty}$ \cite[4.70]{AbramBrown} and we let $E_{\bb,\bc}\subset X$ be the corresponding flat, which we call the flat spanned by $\bb$ and $\bc$.  

We fix a maximal $\R$--split torus $A\subset G$ and its corresponding flat $E=[A]\subset X$. Any two such tori are conjugate and are isomorphic to $\R^k$. Let $E_\infty\subset X_\infty$ denote the boundary at infinity of $E$.    We fix a chamber $\bz\in \cF^{k-1}(E_\infty)$ and let $\bz^*$ be its opposite chamber in $E_\infty$.  Let $P=\Stab(\bz)$ be the stabilizer of $\bz$; this is a minimal parabolic subgroup of $G$.  By the Levi decomposition, we can write $P=NAM$ where $N$ is normal and nilpotent and $M$ is the centralizer of $A$ in $K$.  In particular, $M$ is compact.  Note that $P$ acts transitively on $X$; indeed, $NA$ acts simply transitively on $X$.

The notations $f\lesssim g$ and $g\gtrsim f$ indicate that $f\le C g$ for some universal constant $C>0$, and we write $f\approx g$ if and only if $f\lesssim g$ and $g\lesssim f$.  If the implicit constant $C$ depends on some quantities $a,b$, we show this by a subscript, i.e., $f\lesssim_{a,b} g$.  In what follows, many of our implicit constants will depend on $G$, $k$, and $\Gamma$, so for brevity, we omit these subscripts.

For all $t>0$, let $\blog t:=\max\{1,\log t\}$.

\subsection{Probability: Sobolev norms and smooth random variables}\label{sec:probability prelims}

Kleinbock and Margulis \cite{KMLog} proved quantitative results on the distribution of the geodesic flow in quotients $\Gamma\backslash G$ of semisimple Lie groups, showing how the distribution of the geodesic flow at time $t$ depends on $t$ and on the smoothness of the initial distribution.  In this section, we will introduce some concepts that we will need to state these results.

Let $H$ be a Lie group and let $\frakh$ be the Lie algebra of $H$.  Let $\mu$ be a left-invariant Haar measure on $H$.  If $x\in H$ is a continuous random variable, we let $\phi_x\from H\to \R$ be its probability density function, so that for any open set $U\subset H$, we have
$$\Prob[x\in U]=\int_{U} \phi_x(u) \; d\mu(u).$$
If $\phi_x$ is smooth, we say that $x$ is a \emph{smooth random variable}.

If $H,L$ are Lie groups, $f\from H\to L$ is a submersion, and $x$ is a smooth random variable with compact support, then the implicit function theorem implies that $f(x)$ is also a smooth random variable.  If $\phi=\phi_x\in C^\infty(H)$ is the density function of $x$, we define the \emph{push-forward} $f_*(\phi)$ as the density function of $f(x)$, i.e., $f_*(\phi_x)=\phi_{f(x)}$.  In the case that $f$ is a diffeomorphism, this is given by 
\begin{equation}\label{eq:define pushforward}
  f_*(\phi)(f(x))=\left|J(f)(x)^{-1}\right|\cdot \phi(x),
\end{equation}
where $J(f)$ is the Jacobian of $f$.  We take \eqref{eq:define pushforward} to be the definition of $f_*(\phi)$ for arbitrary $\phi\in C^{\infty}(H)$.  

We bound the smoothness of a function on $H$ by introducing a left-invariant \emph{Sobolev norm}.  Let
$$\cT(\frakh)=\R\oplus \frakh\oplus (\frakh\otimes\frakh)\oplus\dots$$
be the tensor algebra of $\frakh$ and let
$$\cT_j(\frakh)=\R\oplus \frakh\oplus \dots\oplus \frakh^{\otimes j}$$
be the subspace spanned by tensors of rank at most $j$.  

We view the elements of $\cT(\frakh)$ as left-invariant differential operators on $H$ by defining $(V_1\otimes\dots \otimes V_j)f=V_1\dots V_jf$.  If $l>0$ and
$f\in C^l(H)$, we define
\begin{equation}\label{eq:defSobOperator}
  \|f\|_{l,2}=\mathop{\sup_{\Upsilon \in \cT_l(\frakh)}}_{\|\Upsilon\|\le 1} \|\Upsilon  f\|_2,
\end{equation}
where $\|\cdot\|_2$ is the $L_2$ norm with respect to $\mu$.  Alternatively, if $Y_1,\dots, Y_d\in \frakh$ is an orthonormal basis, then tensors of the form $Y_{i_1}\otimes \dots \otimes Y_{i_j}$ for $0\le j\le l$ form an orthonormal basis of $\cT_l(\frakh)$, so
\begin{equation}\label{eq:defSobOrtho}
  \|f\|_{l,2}=\sqrt{\sum_{0\le j\le l, 1\le i_n\le d} \|Y_{i_1}\dots Y_{i_j} f\|_2^2}.
\end{equation}
For all $h\in H$, let $hf(x)=f(h^{-1}x)$.  For all $\Upsilon\in \cT(\frakh)$, we have $\Upsilon hf=h(\Upsilon f)$, so $\|hf\|_{l,2}=\|f\|_{l,2}$.  

For any subset $Z\subset H$, let $C^\infty(Z)$ be the set of smooth functions with support in $Z$.  Let $B^\infty(Z)$ be the set of elements of $C^\infty(Z)$ whose derivatives are all $L^2$ functions.  If $U$ is contained in a compact set, then $C^\infty(U)=B^\infty(U)$.  If $\Gamma$ is a lattice in $H$, then elements of $\cT(\frakh)$ also act as differential operators on functions on $\Gamma\backslash H$, so we may also define $\|\cdot\|_{l,2}$ on $B^\infty(\Gamma\backslash H)$.  

\begin{remark}
  The definition \eqref{eq:defSobOperator} takes a supremum over left-invariant differential operators, but one can define an equivalent norm by taking a supremum over all differential operators.  Any map $\Upsilon\from H\to \cT_l(\frakh)$ acts on $B^\infty(H)$ as a differential operator of order at most $l$.  We define $\|\Upsilon\|_\infty=\sup_{h\in H} \|\Upsilon(h)\|_2$.  For all $\Upsilon$ and all $f\in B^\infty(H)$, we have 
  $$\|\Upsilon f\|_2\lesssim_{l,\dim H} \|\Upsilon\|_\infty \cdot \|f\|_{l,2}.$$
\end{remark}

Let $H,L$ be Lie groups.  When $f\from H\to L$ is a diffeomorphism or submersion, we can bound the Sobolev norm of $f_*(\phi)$ in terms of $\|\phi\|_{l,2}$ and the derivatives of $f$.
\begin{lemma}\label{lem:SobolevProps}
  Let $x\in H$, $y\in L$ be smooth random variables with density functions $\alpha\in B^\infty(H)$ and $\beta\in B^\infty(L)$ respectively.  The following properties hold:
  \begin{enumerate}
  \item If $f\from H\to L$ is a diffeomorphism and if $C\subset H$ is
    a compact set such that $\supp \alpha\subset C$, then
    $$\|f_*(\alpha)\|_{l,2}\lesssim_{f,C,l} \|\alpha\|_{l,2}.$$
  \item If $f\from H\to L$ is a Lie group isomorphism, then
    $$\|f_*(\alpha)\|_{l,2}\le \Lip (f^{-1})^{l+\frac{\dim H}{2}} \|\alpha\|_{l,2}.$$
  \item The product $(x,y)\in H\times L$ is a smooth random variable and
    $$\|\phi_{(x,y)}\|_{l,2}\lesssim_{\dim H,l} \|\alpha\|_{l,2}\|\beta\|_{l,2}.$$
  \item 
    If $g\from H\to L$ is a smooth map and $z\in G$ is a point such that the derivative $Dg_z\from T_zH\to T_{g(z)}L$ is surjective, then there is an $\epsilon>0$ such that if $\alpha\in C^\infty(B_z(\epsilon))$, then $g_*(\alpha)\in B^\infty(L)$ and
    $$\|g_*(\alpha)\|_{l,2} \lesssim_{g,z} \|\alpha\|_{l,2}.$$
  \end{enumerate}
\end{lemma}
\begin{proof}

  To show the first property, note that for all $\Upsilon\in T_l(\frakh)$, the operator $\alpha\mapsto \Upsilon f_*(\alpha)$ is a differential operator of order at most $l$.  If $\|\Upsilon\|=1$, then the coefficients of this operator are bounded in terms of the derivatives of $f|_C$ and $f^{-1}|_{f(C)}$, so 
  $$\|\Upsilon f_*(\alpha)\|_2\lesssim_{f,C,l} \|\alpha\|_{l,2}$$
  and thus $\|f_*(\alpha)\|_{l,2}\lesssim_{f,C,l}\|\alpha\|_{l,2}$.

  For the second property, let $c=|J(f^{-1})(e)|$ be the norm of the Jacobian of $f^{-1}$.  For any $\gamma\in B^\infty(H)$, we have $f_*(\gamma)=c  \gamma\circ f^{-1}\in B^{\infty}(L)$.  By the chain rule, for any $Y\in \frakl$,
  $$Y(f_*(\gamma))=c Y(\gamma\circ f^{-1})=c (Df^{-1}(Y)\gamma)\circ f^{-1} =f_*(Df^{-1}(Y)\gamma),$$
  so for $Y_1,\dots, Y_j\in \frakl$, we have
  $$Y_1\dots Y_j f_*(\alpha)=c\cdot (Df^{-1}(Y_1)\dots Df^{-1}(Y_j) \alpha)\circ f^{-1}$$
  and
  \begin{align*}
    \|Y_1\dots Y_j f_*(\alpha)\|_2^2
    &\le \int_L c^2 (Df^{-1}(Y_1)\cdots Df^{-1}(Y_j) \alpha)(f^{-1}(v))^2\; dv\\
    &= \int_H c^2(Df^{-1}(Y_1)\cdots Df^{-1}(Y_j) \alpha)(w)^2 c^{-1}\; dw\\
    &= c  \|Df^{-1}(Y_1)\cdots Df^{-1}(Y_j) \alpha\|_2^2\\
    &\le c(\Lip f^{-1})^j \|Y_1\|_2\cdots \|Y_j\|_2 \|\alpha\|_{l,2}.
  \end{align*}
  We therefore have 
  $$\|f_*(\alpha)\|_{l,2}\le \sqrt{c} (\Lip f^{-1})^l \|\alpha\|_{l,2}\le (\Lip f^{-1})^{l+\frac{\dim H}{2}} \|\alpha\|_{l,2}$$
  as desired.

  Third, if $y\in L$ is a smooth random variable in $H$ with density function $\beta$, then $(x,y)\in H\times L$ is a smooth random variable with density function $\phi(x,y)=\alpha(x)\beta(y)$.  Let $Y_1,\dots, Y_{d_1}$ be an orthonormal basis for $\frakh$ and $Y_{d_1+1},\dots, Y_{d_1+d_2}$ be an orthonormal basis for $\frakl$.  Suppose that $0\le j\le l$ and that $1\le i_n\le d$ for all $n=1,\dots, j$; then 
  $$Y_{i_1}\dots Y_{i_j} \phi(x,y)= \bigl(\Upsilon_1 \alpha(x)\bigr) \bigl(\Upsilon_2 \beta(y) \bigr)$$
  for two operators $\Upsilon_1\in \cT(\frakh)$, $\Upsilon_2\in \cT(\frakl)$ with $\|\Upsilon_1\|=\|\Upsilon_2\|=1$.  (Here, $\Upsilon_1=Y_{i'_1}\dots Y_{i'_{j'}}$ and $\Upsilon_2=Y_{i''_1}\dots Y_{i''_{j''}}$, where $i'_1,\dots, i'_{j'}$ is the subsequence of $i_1,\dots, i_j$ consisting of values that are at most $d_1$ and $i''_1,\dots, i''_{j''}$ is the complement.)  It follows that 
  $$\|Y_{i_1}\dots Y_{i_j} \phi\|_2=\|\Upsilon_1\alpha\|_2 \|\Upsilon_2\beta\|_2\le \|\alpha\|_{l,2}\|\beta\|_{l,2}.$$
  By \eqref{eq:defSobOrtho}, this implies $\|\phi\|_{l,2}\lesssim_{l,d}\|\alpha\|_{l,2}\|\beta\|_{l,2}$.  

  To prove the last property, note that $H$ and $L$ are locally diffeomorphic to Euclidean spaces, so if the property holds for $H=\R^m$ and $L=\R^n$, then it holds for arbitrary Lie groups by part 1 of the lemma.  Let $g\from \R^m\to \R^n$ be a smooth map and let $z\in \R^m$ be a point such that the derivative $Dg_z\from \R^m\to \R^n$ is surjective.  By the Implicit Function Theorem, there is an $\epsilon>0$ and there are diffeomorphisms $a\from \R^m\to \R^m$ and $b\from \R^n\to \R^n$ such that $g|_{B_z(\epsilon)}=b\circ p\circ a$, where $p\from \R^m\to \R^n$ is the projection to the first $n$ coordinates.  By composing with a translation, we may suppose that $a(z)=0$, and we choose $\epsilon$ small enough that $g(B_z(\epsilon))\subset B_0(\frac{1}{2})$.

  Let $I=[-\frac{1}{2},\frac{1}{2}]$ and let $U=I^m$ be a unit cube in $\R^m$.  Suppose that $\phi\in C^\infty(U)$ and $\Upsilon\in \cT(\R^n)$.  We view $\R^n$ as the subspace of $\R^m$ spanned by the first $n$ coordinates.  Then for all $x\in \R^n$, we have
  $$p_*(\phi)(x)=\int_{I^{m-n}} \phi(x+y)\; dy$$
  and
  $$\Upsilon p_*(\phi)(x)=\int_{I^{m-n}} \Upsilon \phi(x+y)\; dy.$$
  By Jensen's inequality,
  \begin{align*}
    \|\Upsilon p_*(\phi)\|_2 
    &= \sqrt{\int_{I^n} \left(\int_{I^{m-n}} \Upsilon \phi(x+y)\; dy \right)^2\;dx}\\
    &\le \sqrt{\int_{I^n} \int_{I^{m-n}} (\Upsilon \phi(x+y))^2\; dy \;dx}\\
    &= \|\Upsilon \phi\|_2
  \end{align*}
  so $\|p_*(\phi)\|_{l,2}\le \|\phi\|_{l,2}$.  

  By our choice of $\epsilon$ and by part 1 of the lemma, we have $a_*(\alpha)\in C^\infty(U)$, so 
  $$\|g_*(\alpha)\|_{l,2}=\|(b_*\circ p_*\circ a_*)(\alpha)\|_{l,2}\lesssim_{a,b} \|\alpha\|_{l,2},$$
  as desired.
\end{proof}

\subsection{Arithmetic groups and lattices}

\subsubsection{Thick parts of groups and symmetric spaces}\label{sec:lattices and horoballs}
As in Section~\ref{sec:standing}, let $G$ be a semisimple group with $\Rrank(G)\geq 2$ and let $\Gamma$ be a nonuniform irreducible (arithmetic) lattice in $G$.  Let $X=G/K$ be the corresponding symmetric space, and for $g\in G$, let $[g]=gK$ be the projection of $g$ to $X$.   

The quotient $Y:=\Gamma\backslash G$ is noncompact, but it can be divided into a thick part (a neighborhood of a basepoint) and a thin part (the union of a collection of cusps).  Let $d_\Gamma \from G\to \R$ be the function $d_\Gamma(g):=d(g,\Gamma)$.  This descends to a function on $Y$, and we let $H(s):=\{y\in Y\mid d_\Gamma(y)> s\}$.  When $s$ is large, the quotient $\Gamma\backslash H(s)$ is contained in the thin part of $Y$.  Kleinbock and Margulis showed that the volume of the thin part of $Y$ decays exponentially with distance.  
\begin{lemma}[{\cite[5.1]{KMLog}}]\label{lem:small horoballs}
  There are $\sansA>0$ and $C_1,C_2>0$ such that for all $s>0$, 
  $$C_1 e^{-\sansA s}\le \mu(H(s))\le C_2 e^{-\sansA s}.$$
\end{lemma}
In the terminology of \cite{KMLog}, $d_\Gamma$ is an $\sansA$--DL function, which implies that it is a DL--function.  Consequently, it satisfies the following smoothing result.  Let $\mu$ be the Haar measure on $G$ (resp.\ $Y$) normalized such that  $\mu(Y)=1$.
\begin{lemma}[{\cite[4.2]{KMLog}}]\label{lem:KM smoothing}  
  For all $s\in \R$, there are two non-negative functions $m',m''\in B^\infty(Y)$ such that $m'\le \one_{H(s)}\le m''$,
  $$\mu(m')\approx \mu(m'')\approx \mu(H(s)),$$
  and 
  $$\|m'\|_{l,2}\approx \|m''\|_{l,2} \approx \mu(H(s)).$$
\end{lemma}
Lemma~4.2 of \cite{KMLog} proves this result with a slightly weaker norm (i.e., bounds on $\|\Upsilon m'\|_2$ and $\|\Upsilon m''\|_2$ for a particular differential operator $\Upsilon$ rather than for all differential operators), but the construction, which is based on a convolution with a smooth bump function, satisfies the stronger bound.

The next lemma follows from Lemma 3 in \cite{Leu-Kazh}.
\begin{lemma}\label{lem:injrad}
  The injectivity radius of $Y$ at $y$ decays exponentially with $d_\Gamma(y)$, and for all $r>0$, there is a $c>0$ such that for all $g \in G$,
  $$\#\{\gamma\in \Gamma\mid d(g,\gamma g) < r\} \le c e^{c d_\Gamma(g)}.$$
\end{lemma}

We can similarly divide the locally symmetric space $\Gamma\backslash X$ into a thick and thin part.  In this case, we can choose the thick part so that its lift $X_0\subset X$ is a submanifold with corners on which $\Gamma$ acts cocompactly.  In fact, the following theorem follows from Theorem~5.2 of \cite{LeuPoly}.
\begin{thm}\label{thm:fundDomain}
  If $X$, $\Gamma$ are as in the standing assumptions, then there exist a $\Gamma$--invariant submanifold with corners $[\Gamma]\subset \FD\subset X$ such that $\Gamma\backslash \FD$ is compact and an exponentially Lipschitz $\Gamma$--equivariant deformation retraction $\rho\from X\to \FD$.  That is, there is a $c>0$ such that for all $r>0$, $\rho$ satisfies 
  $$\Lip(\rho|_{X(r)})\lesssim  e^{cr}.$$
  In fact, $\rho$ is the closest-point projection to $\FD$.
\end{thm}
The submanifold $X_0$ is the complement of a union of horoballs and is quasi-isometric to $G$.  

It will be convenient to define neighborhoods of $X_0$.  Let $d_{[\Gamma]}(x):=d(x,[\Gamma])$ for all $x\in X$, and for $r\geq 0$, let $X(r):=d_{[\Gamma]}^{-1}([0,r])$.  Since $\Gamma \backslash \FD$ is compact, there is an $r_0$ such that $X_0\subset X(r_0)$.

\subsubsection{Shadows} \label{sec:shadows}

In this section we will recall the definition of a shadow of a point as in \cite{LeYoRank1}.  As in Section~\ref{sec:standing}, $X_{\infty}$ is the spherical Tits building associated to $X$, $\chx$ is its set of (maximal) chambers, $\bz$ and $\bz^*$ are a pair of opposite chambers, and $P=NAM$ is the Levi decomposition of $P=\Stab(\bz)$.  Let $\chxop=\chxop(\bz)$ be the set of chambers opposite to $\bz$.

The shadow of a point $x\in X$ will consist of the chambers $\bc\in \chxop$ such that the flat $E_{\bc,\bz}$ joining $\bc$ and $\bz$ is close to $x$.  We describe this in terms of a bijection $\iota\from \chxop\to N$.   The group $NA$ acts transitively on $\chxop$ and the stabilizer of $\bz^*$ is $A$.  Let $\iota(n\bz^*)=n$ for all $n\in N$.  A shadow of $x$ will correspond to a neighborhood in $N$.  

Let $p_N\from NA\to N$, $p_A\from NA\to A$ be the projections $p_N(na)=n$ and $p_A(na)=a$ for all $n\in N$, $a\in A$.  The group $NA$ acts on $\chxop$, and we let $\rho_g\from N\to N$ be the conjugate action $\rho_g(\iota(\bc))=\iota(g\bc)$.  By the normality of $N$, we have
\begin{multline*}
  \rho_g(n) = \iota(g \iota^{-1}(n)) = \iota(gn\bz^*) = \iota(p_N(g)p_A(g)n\bz^*)\\ = \iota(p_N(g)n^{p_A(g)}p_A(g)\bz^*)
  = \iota(p_N(g)n^{p_A(g)}\bz^*) = p_N(g)n^{p_A(g)}.
\end{multline*}

\begin{defn}\label{def:shadows}
  Let $\frakg$ and $\frakn$ be the Lie algebras of $G$ and $N$ and let $\|\cdot\|$ be the $\Ad(K)$--invariant norm on $\frakg$.  For any $n\in N$, we define $d_N(n)=\|\log n\|$.  

  Suppose that $x\in X$ and $\bc\in \chxop$.  There is a unique $g\in NA$ such that $x=[g]$, and we define $d_x\from \chxop \to \R$ so that 
  $$d_x(\bc)=d_N(\iota(g^{-1}\bc)).$$
  Note that if $a_1,a_2,\dots \in A$ and $[a_i]$ converges to $\bz$, then $d_{[a_i]}(\bc)\to 0$ (see also Lemma \ref{lem:dil} below).
  For $r>0$, the {\emph r-shadow of $x$} is defined as 
  $$\cSc_{x}(r) := \{\bc\in \chxop \mid d_x(\bc)<r\}.$$ We also set $\cSc_{x}:=\cSc_{x}(1)$.
\end{defn}
It follows from the definition that shadows are equivariant under the action of $NA$; i.e., if $x\in X$, $g\in NA$, $r>0$, and $\bc\in \chxop$, then $d_{gx}(g\bc)=d_{x}(\bc)$ and $\cSc_{gx}(r)=g\cSc_{x}(r)$.  

We proved versions of the following lemmas in \cite{LeYoRank1}
\begin{lemma}[{\cite[3.2]{LeYoRank1}}]\label{lem:compare}
  There is a constant $C>0$ depending on $X$ such that for all points $x\in X$ and chambers $\bc\in \chxop$ opposite to $\bz$,
  $$d(x,E_{\bc,\bz})\le d_x(\bc)\le \exp(C d(x,E_{\bc,\bz})).$$
	Here $E_{\bc,\bz}$ denotes the unique flat in $X$ joining the chambers $\bc$ and $\bz$.
\end{lemma}

\begin{lemma}[{\cite[3.3]{LeYoRank1}}]\label{lem:dil}
  Let $x\in X$ and let $\gamma\from [0,\infty)\to X$ be a unit-speed geodesic ray starting at $x$ and asymptotic to a point $\sigma\in \inter \bz$ in the interior of $\bz$.  There are constants $D, \kappa>0$ depending on $\sigma$ such that for all $t\ge 0$ and all $\bc\in \chxop$,
  $$d_{\gamma(t)}(\bc)\le D e^{-\kappa t}d_x(\bc)$$
  and thus for any $r>0$, 
  $$\cSc_x(r)\subset \cSc_{\gamma(t)}(D e^{-\kappa t}r).$$
\end{lemma}

\begin{lemma}[{see \cite[3.4]{LeYoRank1}}]\label{lem:largeShadows}
  There is a constant $E>0$ with the following property.  Let $x\in X$ and let $\gamma\from [0,\infty)\to X$ be the unit-speed geodesic ray starting at $x$ and asymptotic to the barycenter of $\bz$.  For any $y\in X$ and any $t\ge E(d(x,y)+1)$, we have $\cSc_{y}\subset \cSc_{\gamma(t)}.$
\end{lemma}

Shadows induce a partial ordering on $X$, where $x\prec y$ if $\cSc_x\subset \cSc_y$.  The set of points $p$ such that $p\succ x$ is coarsely convex in the following sense.
\begin{lemma}\label{lem:convex hull shadows}
  There is a $r>0$ with the following property.  Suppose that $x\in X$.  Let
  $$P_x=\{p\in X\mid \cSc_x\subset \cSc_p\}$$
  and let $Q_x$ be the convex hull of $P_x$.  Then $\cSc_x\subset \cSc_q(r)$ for all $q\in Q_x$.
\end{lemma}
\begin{proof}
  For $y\in X$, let 
  $$\cR_y=\{\bc\in \chxop\mid d(y,E_{\bc,\bz})<1\}.$$
  By Lemma~\ref{lem:compare}, there is a $r>0$ such that $\cSc_y\subset \cR_y\subset \cSc_y(r)$ for all $y\in X$, and we define
  $$U_x=\{y\in X\mid \cSc_x\subset \cR_y\}.$$
  This contains $P_x$, and we claim that it is convex.

  We have $y\in U_x$ if and only if $d(y,E_{\bc,\bz})<1$ for all $\bc\in \cSc_x$, i.e.,
  $$U_x=\bigcap_{\bc\in \cSc_x} N_1(E_{\bc,\bz}),$$
  where, for $L\subset X$, $N_1(L)=\{y\in X\mid d(y,L)<1\}$.  Each set $N_1(E_{\bc,\bz})$ is convex (it is a neighborhood of a flat), so their intersection is also convex.  It follows that $Q_x\subset U_x$.  Consequently, for all $q\in Q_x$, we have $q\in U_x$ and thus $\cSc_x\subset \cR_q\subset \cSc_q(r)$, as desired.
\end{proof}

\subsubsection{Smooth random chambers}

Since there is a bijection between $\chxop$ and $N$, we can define random chambers in terms of distributions on $N$.  A \emph{smooth random chamber} $\br\in \chxop$ is a random variable such that the density function $\phi_{\iota(\br)}$ of $\iota(\br)$ satisfies $\phi_{\iota(\br)}\in C^\infty(N)$.  For brevity, we will shorten $\phi_{\iota(\br)}$ to $\phi_\br$.  For any $g\in NA$, we have
$$\phi_{g\br}=\phi_{\iota(g\br)}=\phi_{\rho_g(\iota(\br))}=(\rho_g)_*\phi_\br.$$

In Section~\ref{sec:probability prelims}, we defined Sobolev norms on $C^\infty(N)$.  These are invariant under the action of $N$ but not invariant under the action of $NA$, so we will define a collection of norms that measure the smoothness of a function $\phi$ as viewed from a (base)point $g\in NA$ or $x\in X$.  In Section~\ref{sec:equidistribution}, we will use these norms to bound the distribution of random flats or chambers that pass through $[g]$.  Let $l>0$ be an integer depending on $G$ and $\bz$ to be determined later (see Theorem~\ref{thm:KM matrix coeffs}).  If $\phi\in C^\infty(N)$ and $g\in NA$, let
$$\|\phi\|_g:=\|(\rho_{g^{-1}})_*\phi\|_{l,2},$$  
so that for any smooth random chamber $\br$, $$\|\phi_{\br}\|_g=\|\phi_{g^{-1}\br}\|_{l,2}.$$
For any $h\in NA$, we have
$$\|\phi_{h\br}\|_{g}=\|\phi_{g^{-1}h}\br\|_{l,2} = \|\phi_{\br}\|_{h^{-1}g}.$$
Note that for all $g\in NA$ and $n\in N$, the $N$--invariance of $\|\cdot\|_{l,2}$ implies that
$$\|\phi\|_{gn}=\|(\rho_{n^{-1}}\circ \rho_{g^{-1}})_*\phi\|_{l,2}=\|\phi\|_{g}.$$
In particular, for all $g\in NA$, $\|\phi\|_g=\|\phi\|_{p_A(g)}$.

For all $x\in X$, there is a unique $g\in NA$ such that $x=[g]$, and we define $\|\phi\|_x:=\|\phi\|_g$.  For any $h\in NA$ and any $x\in X$, we have $\|(\rho_{h})_*\phi\|_{hx}=\|\phi\|_x$.
\begin{lemma}\label{lem:expDecay}
  There is a constant $c>0$ depending on $G$ such that for all $\phi\in C^\infty(N)$ and all $x,y\in X$,
  \begin{equation}\label{eq:expDecay}
    \|\phi\|_{x}\le \exp( c d(x,y)) \|\phi\|_{y}.
  \end{equation}
\end{lemma}
\begin{proof}
  Let $g,h\in NA$ be such that $x=[g]$, $y=[h]$.  Let $n=p_N(h^{-1}g)$, $a=p_A(h^{-1}g)$, so that $g=hna$; since $p_A$ is distance-decreasing, we have $\|a\|\le d(x,y)$.  

  By the $N$--invariance of $\|\cdot\|_{l,2}$ and Lemma~\ref{lem:SobolevProps},
  $$\|\phi\|_{g}=\|(\rho_{a^{-1}})_* (\rho_{h^{-1}})_*\phi\|_{l,2}\le \Lip(\rho_{a})^{l+\frac{\dim N}{2}}\|(\rho_{h^{-1}})_*\phi\|_{l,2}=\Lip(\rho_{a})^{l+\frac{\dim N}{2}}\|\phi\|_h.$$
  Since $\rho_{a^{-1}}$ acts on $N$ by conjugation, there is a $c_0$ such that $\Lip(\rho_{a})\le e^{c_0\|a\|},$ and the lemma holds.
\end{proof}

\subsection{Filling invariants}\label{sec:fillingDefs}

In this section, we will review the definitions of the filling invariants and some related objects used to state the main theorems.  For a full discussion of higher-dimensional analogues of the Dehn function, see \cite[2.1]{ABDDYraag}.

Let $Y$ be a simply-connected simplicial complex or Riemannian manifold.  If $\alpha\from S^1\to Y$ is a Lipschitz map, we define the \emph{filling area} of $\alpha$ as
\begin{equation}\label{eq:def filling area}
  \delta(\alpha)=\mathop{\inf_{\beta\from D^2\to Y}}_{\beta|_{S^1}=\alpha}\vol^2 \beta,
\end{equation}
where the infimum is taken over Lipschitz maps $\beta\from D^2\to Y$ that agree with $\alpha$ on the boundary.  By Rademacher's Theorem, $\beta$ is differentiable almost everywhere, so we may define $J_\beta$ to be its Jacobian and let $\vol^2 \beta=\int_{D^2} |J_\beta(x)|\;dx$.  

We define the \emph{Dehn function} of $Y$ to be the function $\delta_Y\from [0,\infty)\to \R\cup \{\infty\}$,
\begin{equation}\label{eq:def Dehn fn}
  \delta_Y(t)=\mathop{\sup_{\alpha\from S^1\to Y}}_{\ell(\alpha)\le t}\delta(\alpha).
\end{equation}

If $Y$ and $Y'$ are quasi-isometric simply-connected simplicial complexes or Riemannian manifolds with bounded geometry (i.e., bounded degree or bounded curvature), then there is a $C>0$ such that for all $t>0$,
\begin{equation}\label{eq:eqRel}
  \delta_Y(t)\le C \delta_{Y'}(Ct+C)+Ct+C
\end{equation}
and vice versa \cite{BridsonWordProblem}.  Consequently, one can define the Dehn function $\delta_G$ of a group by letting $\delta_G=\delta_Y$ for any complex or manifold $Y$ on which $G$ acts geometrically; this depends on the choice of $Y$, but any two choices satisfy the inequality \eqref{eq:eqRel}.  

We define higher-order Dehn functions by considering fillings of spheres by discs.  When $d\ge 1$, we define the \emph{homotopical filling volume} of a Lipschitz map $\alpha\from S^{d-1}\to Y$ as the infimal volume of a Lipschitz extension $\beta\from D^d\to Y$ that agrees with $\alpha$ on its boundary and define the $(d-1)$th--order Dehn function $\delta^{d-1}_Y(V)$ as the supremum of the homotopical filling volumes of spheres of volume at most $V$.  

In practice, it is often easier to work with a homological version of $\delta^{n-1}$, the \emph{$n$--dimensional filling volume function}, which measures the difficulty of filling Lipschitz $(n-1)$--cycles by Lipschitz $n$--chains.  Let $Y$ be a simplicial complex or Riemannian manifold.  A \emph{singular Lipschitz $d$--chain} in $Y$, or simply a Lipschitz $d$--chain, is a formal sum (with integer coefficients) of Lipschitz maps $\Delta^d\to Y$, where $\Delta^d$ is the unit $d$--simplex.  The Lipschitz chains form a subcomplex of the singular chain complex, which we denote $\CLip_d(Y)$.

Let $\alpha\in \CLip_{d-1}(Y)$.  There are nonzero coefficients $c_1,\dots, c_n\in \Z$ and distinct Lipschitz maps $\alpha_1,\dots,\alpha_n\from \Delta^{d-1}\to Y$ such that $\alpha=\sum_{i=1}^n c_i \alpha_i$.  Each of these maps is Lipschitz, so, as above, we define $$\vol^{d-1}(\alpha_i)=\int_{\Delta^{d-1}} |J_{\alpha_i}(x)|\;dx.$$
Let $$\supp \alpha=\bigcup_{i=1}^n \alpha_i(\Delta^{d-1}),$$
and 
$$\mass \alpha=\sum_{i=1}^n |c_i| \vol \alpha_i.$$
If $f\from Y\to \R$ is a continuous function, we define
\begin{equation}\label{eq:integral over chains}
  \int_\alpha f(y) \;dy=\sum_{i=1}^n |c_i| \int_{\Delta^{d-1}} f(\alpha(x)) |J_{\alpha_i}(x)|\;dx.
\end{equation}
Note that $\int_{\alpha} 1\;dy=\mass \alpha$ and in general, $\int_{\alpha} f \;dy\le \|f\|_\infty \mass \alpha$.  If $Z\subset Y$, let 
$$\mass_Z \alpha=\int_\alpha \one_Z\;dx.$$

When $\alpha\in \CLip_{d-1}(Y)$ is a cycle, we define the filling volume of $\alpha$ by
$$\FV^{d}(\alpha)=\mathop{\inf_{\beta\in \CLip_{d}(Y)}}_{\partial \beta=\alpha}\mass \beta.$$
If $Y$ is $(d-1)$--connected and $V\ge 0$, we define
$$\FV^{d}_Y(V)=\mathop{\sup_{\alpha\in \CLip_{d-1}(Y)}}_{\partial \alpha=0, \mass \alpha \le V}\FV^d(\alpha).$$
Like the Dehn function, this is a quasi-isometry invariant; if $Y$ and $Y'$ are quasi-isometric $(d-1)$--connected simplicial complexes or Riemannian manifolds with bounded geometry, then there is a $C>0$ such that for all $t>0$,
\begin{equation}\label{eq:eqRel1}
  \FV^d_Y(t)\le C \FV^d_{Y'}(Ct+C)+Ct+C
\end{equation}
and vice versa.  Consequently, if $G$ is a group that acts geometrically on $Y$, we define $\FV_G=\FV_Y$; as before, this is well-defined up to \eqref{eq:eqRel1}.  (See for instance \cite[10.3]{ECHLPT} or \cite{AWP}.)  

When $Y$ is CAT(0), for instance, when $Y$ is a nonpositively curved symmetric space, Gromov showed that it satisfies isoperimetric bounds based on the Euclidean isoperimetric inequality.
\begin{thm}[{\cite{GroFRM, WengerShort}}]\label{thm:CAT0 fillings}
  Let $Y$ be a CAT(0) space and $n\ge 2$.  Then for all $V>0$,
  $$\FV_Y^n(V)\lesssim V^{\frac{n}{n-1}}.$$
\end{thm}
Although the symmetric space $X$ is CAT(0), the thick part $X_0$ typically is not; one proof of this is that, by Theorem~\ref{thm:mainThmLower}, the $k$--dimensional filling volume function is exponential.

For a survey of the relationship between homological and homotopical filling invariants, see \cite{ABDYpnas}.  The main facts that we will use here are that when $Y$ is $(d-1)$--connected and $d\ge 4$, we have $\delta^{d-1}_Y(V)=\FV^{d}_Y(V)$, and when $d=3$, we have $\delta^{2}_Y(V)\le\FV^{3}_Y(V)$ \cite[Rem.2.6(4)]{BBFSsnowflake} \cite[App.2.(A')]{GroFRM}.  In this paper, we will prove upper bounds on $\delta_{X_0}$ and on $\FV^d_{X_0}$ for $d\ge 2$.  By the above, these bounds imply upper bounds on $\delta^{d-1}_{X_0}$ for all $d$.  

\section{Logarithm laws and random chambers}\label{sec:equidistribution}

The first step in the proof of the main theorems is to describe the behavior of random flats and random chambers in $X$; that is, translates $[gA]$ and $g\bz^*$, where $g$ is a random variable in either $G$ or $NA$.  

Our main tool is the following result of Kleinbock and Margulis that bounds the matrix coefficients of the action of $G$ on $B^\infty(Y)$, where $Y=\Gamma\backslash G$ is a quotient of a semisimple group by a non-uniform lattice. The group $G$ acts on the function space $B^\infty(Y)$ defined in Section~\ref{sec:probability prelims} by the right regular representation; that is, for $\phi\in B^\infty(Y)$ and $g\in G$, we let $\phi g\in B^\infty(Y)$ be the function $(\phi g)(y)=\phi(yg^{-1})$, so that if $\phi=\phi_x$ is the density function of $x$, then $\phi g=\phi_{xg}$ is the density function of $xg$.  Let $\mu$ be the Haar measure on $Y$ such that $\mu(Y)=1$.  Define $\mu(\phi)=\int_Y\phi(y) \;d\mu(y)$ and $(\phi, \psi)=\int_Y \phi(y)\psi(y)\;d\mu(y)$. Recall from Section~\ref{sec:probability prelims} that $B^\infty(Y)$  is equipped with a norm $\| \cdot \|_{l,2}$. 

 For a left-invariant metric $d$ on $G$ and $g\in G$ let $\|g\|=d(e,g)$ be the distance to the identity.
\begin{thm}{{\cite[3.5]{KMLog}}}\label{thm:KM matrix coeffs}
  Let $G$ be a connected semisimple, center-free Lie group without compact factors, and let $\Gamma$ be an irreducible non-uniform lattice in $G$,   There exist constants $\sansB, \sansC>0$ and $l\in \N$ such that for any two functions $\phi, \psi\in B^\infty(Y)$ and any $g\in G$, we have
  $$|(\phi g, \psi)-\mu(\phi)\mu(\psi)|\le \sansB e^{-\sansC \|a\|}\|\phi\|_{l,2} \|\psi\|_{l,2}.$$
\end{thm}

We will use Theorem~\ref{thm:KM matrix coeffs} to describe the distribution of points in random flats and Weyl chambers in $X$.  First, we consider random flats.  If $h\in G$ and if $g\in B_e(1)$ is a smooth random variable, then $[hgA]$ is a random flat centered near $[h]$.  The following lemma shows that for all $a\in A$, the distribution of $d_\Gamma(hga)$ can be bounded in terms of $a$, $d_\Gamma(h)$, and the distribution of $g$.  
\begin{lemma}\label{lem:equiFlats}
  Let $\sansA$, $\sansC$, and $l$ be as in Lemma~\ref{lem:small horoballs} and Theorem~\ref{thm:KM matrix coeffs}.  There is a $\sansD>1$ with the following property.  Let $h\in G$ be a point, $g\in B_e(1)\subset G$ be a smooth random variable, $a\in A$, and $y=hga$.  Then for any $s>0$,
  $$\Prob[d_\Gamma(y)>s]\lesssim e^{-\sansA s} \bigl(1+e^{-\sansC \|a\|} e^{\sansD d_\Gamma(h)} \|\phi_g\|_{l,2}\bigr).$$
\end{lemma}
\begin{proof}
  Let $p\from G\to Y$ be the quotient map and let $\phi=\phi_{p(hg)} \in B^\infty(Y)$ be the density function of $p(hg)$.  Then $\phi a$ is the distribution function of $p(hga)$, and 
  $$\Prob[d_\Gamma(y)>s]=\Prob[p(y)\in H(s)]= \int_{H(s)} (\phi a)(u) \; d\mu(u)=(\phi a, \one_{H(s)}).$$
  Let $m''\in B^\infty(Y)$ be as in Lemma~\ref{lem:KM smoothing}, so that $\one_{H(s)}\le m''$ and $\|m''\|_{l,2}\lesssim \mu(H(s))$.  By Theorem~\ref{thm:KM matrix coeffs} and Lemma~\ref{lem:small horoballs},
  \begin{align}
\notag    \Prob[d_\Gamma(y)>s]
\notag    &\le (\phi a, m'')\\ 
\notag    &\le \mu(\phi)\mu(m'')+\sansB e^{-\sansC \|a\|}\|\phi\|_{l,2} \|m''\|_{l,2}\\ 
\notag    &\lesssim \mu(H(s))+e^{-\sansC \|a\|}\mu(H(s)) \|\phi\|_{l,2} \\
    \label{eq:equiFlats intermediate}
          &\lesssim e^{-\sansA s}\bigl(1+e^{-\sansC \|a\|} \|\phi\|_{l,2}\bigr).
  \end{align}
  It remains to estimate $\|\phi\|_{l,2}$.  

  We have $\phi=p_*(h\phi_g)$, so $\phi$ is supported in $p(B_{h}(1))$.  For any $w\in G$,
  $$\phi(\Gamma w)=\sum_{x\in \Gamma w} (h\phi_{g})(x)=\sum_{\gamma\in \Gamma} (h\phi_g)(\gamma w).$$
  That is, 
  $$\phi\circ p=\sum_{\gamma\in \Gamma} \gamma h \phi_g.$$

  Let $\Upsilon\in \Env(G)$ be a left-invariant differential operator of order at most $l$ and suppose that $\|\Upsilon\|\le 1$.  Then
  \begin{align*}
    \|\Upsilon \phi\|_{L_2(Y)}
    &=\sqrt{\int_{p(B_h(1))} \Upsilon \phi(y)^2\;d\mu(y)} \\ 
    &\le \sqrt{\int_{B_h(1)} \Upsilon \phi(p(x))^2\;d\mu(x)}\\ 
    &=\|\Upsilon \phi\circ p\|_{L_2(B_h(1))}\\
    &\le \sum_{\gamma\in \Gamma}\|\Upsilon \gamma h \phi_g\|_{L_2(B_h(1))}.
  \end{align*}
  The function $\gamma h \phi_{g}$ is supported in the ball $B_{\gamma h}(1)$, so $\|\Upsilon(\gamma h\phi_g)\|_{L_2(B_h(1))}=0$ unless $d(h,\gamma h)<2$.  Let $S=\{\gamma\in \Gamma\mid d(h,\gamma h)< 2\}$.  By Lemma~\ref{lem:injrad}, there is a $\sansD >1$ such that
  $\#S \lesssim e^{\sansD d_\Gamma(h)}$, so, since $\Upsilon$ is left-invariant,
  $$\|\Upsilon \phi\|_{L_2(Y)} \le \sum_{\gamma\in S} \|\Upsilon \gamma h\phi_g\|_{L_2(B_h(1))} \le \#S \|\Upsilon \phi_g\|_{L_2(G)}\lesssim e^{\sansD d_\Gamma(h)} \|\phi_g\|_{l,2}.$$

  This holds for all $\Upsilon$, so 
  $$\|\phi\|_{l,2}\lesssim e^{\sansD d_\Gamma(h)} \|\phi_{g}\|_{l,2}.$$
  By \eqref{eq:equiFlats intermediate},
  $$\Prob[d_\Gamma(y)>s]\lesssim e^{-\sansA s} \bigl(1+e^{-\sansC \|a\|} e^{\sansD d_\Gamma(h)} \|\phi_g\|_{l,2}\bigr).$$
\end{proof}

The following corollary summarizes the information provided by the lemma in terms of the exponential moments of $d_\Gamma(hga)$.  Recall that for all $t>0$, we let $\blog t:=\max\{1,\log t\}$.
\begin{cor}\label{cor:expMomentsFlats}
  There are constants $b, b'>0$ with the following property.  Let $h\in G$ be a point and let $g\in B_e(1)\subset G$ be a smooth random variable.  Let $R:=d_\Gamma(h)+\blog \|\phi_g\|_{l,2}$.  For all $a\in A$ such that $\|a\|\ge b'R$, we have
  $$\EE[\exp (b d_\Gamma(hga))]\lesssim 1.$$
\end{cor}

\begin{proof}
  Let $\sansA, \sansC, \sansD$ be as in Lemma~\ref{lem:equiFlats}.  Set $b:=\frac{\sansA}{2}$ and $b':=\frac{\sansD}{\sansC}$.  If $\|a\|\ge b'R$, then
  \begin{align*}
    \log (e^{-\sansC\|a\|} e^{\sansD d(h,\Gamma)} \|\phi_g\|_{l,2})
\le -\sansC b' (d_\Gamma(h)+\blog \|\phi_g\|_{l,2})+\sansD d_\Gamma(h) + \log \|\phi_g\|_{l,2} 
    &\le 0.
  \end{align*}

  By Lemma~\ref{lem:equiFlats}, this implies
  $$\Prob[d_\Gamma(hga)>s]\lesssim e^{-\sansA s}.$$
  Consequently, substituting $u=e^{b s}$, we find
  \begin{align*}
    \EE[\exp(b d_\Gamma(hga))]
    &=\int_0^\infty \Prob[\exp(b d_\Gamma(hga))>u]\; du\\
    &=\int_{-\infty}^\infty b e^{b s}\Prob[d_\Gamma(hga)>s]\; ds \\
    &\lesssim \int_{-\infty}^{0}b e^{b s}\;ds + \int_{0}^{\infty}b e^{b s} e^{-2b s}\; ds\lesssim 1.
  \end{align*}
\end{proof}

A similar result holds for points in smooth random Weyl chambers.  In order to state this result, we will need to introduce a version of the exponential map for the symmetric space $X=G/K$.  Let $CX_\infty$ be the Euclidean cone over $X_\infty$.  This is the metric space
$$CX_\infty=(X_\infty \times [0,\infty))/(X_\infty \times 0)$$
equipped with the metric
$$d_{CX_\infty}((\sigma_1,t_1),(\sigma_2,t_2))^2=t_1^2+t_2^2-2t_1t_2\cos \angle(\sigma_1,\sigma_2).$$
Under this metric, the cone over an apartment in $X_\infty$ is isometric to $\R^k$.

For any $x\in X$ and $v=(\sigma,t)\in CX_\infty$, let $\gamma_{x,\sigma}\from [0,\infty)\to X$ be the geodesic ray based at $x$ and asymptotic to $\sigma$.  We define
$$e_x(v)=e_x(\sigma,t):=\gamma_{x,\sigma}(t).$$  
This gives rise to a Lipschitz map $X\times CX_\infty\to X$.
\begin{lemma}\label{lem:ex locally Lipschitz}
  For any $x,x'\in X$ and $v,v'\in CX_\infty$, 
  $$d_X(e_x(v),e_{x'}(v'))\le d_X(x,x')+d_{CX_\infty}(v,v').$$
\end{lemma}
\begin{proof}
  Let $\sigma,\sigma'\in X_\infty$, $t,t'\ge 0$ be such that $v=(\sigma,t)$, $v'=(\sigma',t')$.  We write
  $$d_X(e_x(v),e_{x'}(v'))\le d_X(e_x(v),e_{x'}(v))+d_X(e_{x'}(v),e_{x'}(v')).$$
  The geodesics $\gamma_{x,\sigma}$ and $\gamma_{x',\sigma}$ are both asymptotic to $\sigma$, so convexity implies that
  $$d_X(e_x(v),e_{x'}(v))=d_X(\gamma_{x,\sigma}(t), \gamma_{x',\sigma}(t))\le d_X(x,x').$$
  The fact that $X$ is CAT(0) implies that $e_{x'}$ is a distance-decreasing map from $CX_\infty$ to $X$ and thus
  $$d_X(e_{x'}(v),e_{x'}(v'))\le d_{CX_\infty}(v,v').$$
\end{proof}

This map lets us parametrize Weyl chambers in $X$. Let  $x\in X $, $\bd\in \chx$ be a chamber of $X_\infty$ and let $C\bd\subset CX_\infty$ be the cone over $\bd$.  There is a unique flat $E_\bd$ such that $x\in E_\bd$ and $\bd\subset (E_\bd)_\infty$; this flat contains $\gamma_{x,\sigma}$ for all $\sigma \in \bd$, so $e_x$ sends $C\bd$ isometrically to a Weyl chamber in $E_\bd$ based at $x$.  

As in Section~\ref{sec:standing}, let $E=[A]\subset X$ be the model flat and let $\bz,\bz^*\in \cF^{k-1}(E_\infty)$ be opposite chambers.  If $\br$ is a random chamber, then $e_x(C\br)$ is a \emph{random Weyl chamber based at $x$}.  Let $\pi_\br\from X_\infty \to \br$ be the map that sends each chamber of $X_\infty$ to $\bb$ by a marking-preserving isomorphism, so that we can write points of $\br$ in the form $\pi_\br(z)$ for $z\in \bz^*$.
\begin{lemma}\label{lem:equiChamber}
  Suppose that $\sansA, \sansC, \sansD$, and $l$ are as in Lemma~\ref{lem:equiFlats}.  Let $x\in X$, $\rho \ge 0$, $z\in \bz^*$, and $t\ge 0$.  Let $\br\in \cSc_{x}(\rho)$ be a smooth random chamber and let $y=e_{x}(\pi_\br(z),t)\in X$ be a random variable.  Then for all $s>0$,
  \begin{equation}\label{eq:equiChamber}
    \Prob[\dGamma(y)>s]\lesssim_\rho e^{-\sansA s} \bigl(1+e^{-\sansC t} e^{\sansD \dGamma(x)} \|\phi_\br\|_{x}\bigr).
  \end{equation}
\end{lemma}
\begin{proof}
  We will first prove \eqref{eq:equiChamber} when $\rho$ is sufficiently small, then prove the general case by a scaling argument.  

  Let $\Stab(\bz)=NAM$ and $\Stab(\bz^*)=MAN^*$ be the Levi decompositions of $\Stab(\bz)$ and $\Stab(\bz^*)$, so that $M$ is a subgroup of the maximal compact subgroup that normalizes $A$, $N^*\subset G$ is conjugate to $N$ and 
  $$\mathfrak{g}=\mathfrak{n}\oplus \mathfrak{a}\oplus \mathfrak{m}\oplus \mathfrak{n}^*.$$
  Let $Z=N\times A\times M\times N^*$ a left invariant metric  and let $f\from Z\to G$ be the map $f(n,a,m,n^*)=namn^*$.  Let $0<\epsilon<\frac{1}{4}$ be such that $f$ takes the ball $B^Z_e(5\epsilon)$ to its image $f(B^Z_e(5\epsilon))$ diffeomorphically.  Let $m\in B^M_e(\epsilon)$, $a\in B^A_e(\epsilon)$, and $n^*\in B^{N^*}_e(\epsilon)$ be independent smooth random variables that are independent of $\br$.

  Let $g\in G$ be such that $[g]=x$.  Suppose that $\rho\le \epsilon$ and thus $\br\in \cSc_x(\rho)\subset \cSc_x(\epsilon)$.  Let $n=\iota(g^{-1}\br)\in N$; this is a smooth random variable such that $\br=g n \bz^*$ and $\pi_\br(z)=gnz$.  We have $n\in B^N_e(\epsilon)$, and $\|\phi_\br\|_x=\|\phi_{n}\|_{l,2}$.  Let $w=f(n,a,m,n^*)=namn^*\in B_e(4\epsilon).$  Lemma~\ref{lem:SobolevProps} and the fact that $f$ is a diffeomorphism on $B^Z_e(5\epsilon)$ imply that $w$ is a smooth random variable and that
  \begin{equation}\label{eq:equiChamber phi g}
    \|\phi_w\|_{l,2}\lesssim \|\phi_n\|_{l,2}\|\phi_a\|_{l,2}\|\phi_m\|_{l,2}\|\phi_{n^*}\|_{l,2}\lesssim \|\phi_\br\|_x.
  \end{equation}

  Let $\gamma\from \R\to A$ be the unit-speed geodesic that is based at $e$ and asymptotic to $z$.  Since the stabilizer of $\bz^*$ contains $N^*, A$, and $M$,
  $$\lim_{t\to \infty}[gw\gamma(t)]=gwz=gnamn^*z=gnz=\pi_\br(z);$$
  that is, $[gw\gamma]$ is a unit-speed geodesic asymptotic to $\pi_\br(z)$.   Since $\Omega(u)=e_{x}(\pi_\br(z),u)$ is another such geodesic, convexity implies
  that for $y=e_{x}(\pi_\br(z),t)=\Omega(t)$,
  $$d([gw\gamma(t)],y)=d([gw\gamma(t)],\Omega(t))\le d([gw\gamma(0)], \Omega(0))=d([gw],x)\le 4\epsilon.$$

  Consequently, $d_\Gamma(gw\gamma(t))>\dGamma(y)-4\epsilon-2\diam K$.  By Lemma~\ref{lem:equiFlats} and \eqref{eq:equiChamber phi g},
  \begin{align}\label{eq:prob}
    \Prob\bigl[\dGamma(y)>s\bigr]
    &\le \Prob[d(gw\gamma(t),\Gamma)>s-4\epsilon-2\diam K]\\ 
    &\lesssim e^{-\sansA s} \bigl(1+e^{-\sansC t} e^{\sansD d_\Gamma(g)} \|\phi_\br\|_{x}\bigr).
  \end{align}
  This concludes the case $\rho<\epsilon$.  

  If $\rho\ge \epsilon$, then by Lemma~\ref{lem:dil}, there is a $x'=[g']\in X$ such that $\br\in \cSc_{x'}(\epsilon)$ and $d(x,x')\lesssim 1+\log \frac{\rho}{\epsilon}$.  If $y'=e_{x'}(\pi_\br(z),t))$, then by convexity, $d(y,y')\le d(x,x')$.  By \eqref{eq:prob} and Lemma~\ref{lem:expDecay}, 
  \begin{align*}
    \Prob[\dGamma(y)>s]
    &\le \Prob[\dGamma(y')>s-d(x,x')]\\
    &\lesssim e^{-\sansA (s-d(x,x'))} \bigl(1+e^{-\sansC t} e^{\sansD d_\Gamma(g)} \|\phi_\br\|_{x'}\bigr)\\
    &\lesssim_\rho e^{-\sansA s} \bigl(1+e^{-\sansC t} e^{\sansD d_\Gamma(g)} \|\phi_\br\|_{x}\bigr).
  \end{align*}
\end{proof}

The following corollary is analogous to Corollary~\ref{cor:expMomentsFlats}, and its proof is essentially the same.  
\begin{cor}\label{cor:expMomentsChambers}
  For any $\rho>0$, there are $b, b'>0$ with the following property.  Let $x\in X$ be a point and let $\br\in \cSc_x(\rho)$ be a smooth random chamber.  Let $z\in \bz^*$.  Let $R=\dGamma(x)+\blog \|\phi_\br\|_x$.  If $t\ge b'R$ and $y=e_{x}(\pi_\br(z),t)\in X$, then 
  $$\EE[\exp (b\dGamma(y))]\lesssim_\rho 1.$$
\end{cor}
\begin{proof}
  As in the proof of Corollary~\ref{cor:expMomentsFlats}, there are $b,b'>0$ so that if $t\ge b'R$, then 
  $$\Prob[\dGamma(y)>s]\lesssim e^{-2 b s}.$$
  We conclude as above that
  $$\EE[\exp(b \dGamma(y))]\lesssim 1.$$
\end{proof}

\section{$k$-dimensional fillings: a super-polynomial lower bound}\label{sec:lower bound}

As a first application of the results of Section~\ref{sec:equidistribution}, we prove a superpolynomial lower bound on the $k$--dimensional filling volume function of $\Gamma$ thus establishing that $k=\Rrank G$ is a ``critical dimension'' of the isoperimetric behavior.  This bound is not sharp, but it is the first step in obtaining the sharp bound (see Proposition~\ref{prop:sharpLowerBounds}).

We start by constructing a sphere that lies logarithmically close to $[\Gamma]$.  Recall that for $r>0$, we defined $X(r)=\dGamma^{-1}([0,r])$ and chose $r_0$ such that $\FD\subset X(r_0)$.
\begin{lemma}\label{lem:polyLower}
  There are constants $b, \eta>0$ such that for any sufficiently large $L>2$, there is a Lipschitz map $\alpha\from S^{k-1}\to X(\eta \log L)$ (indeed, an isometric embedding of a round sphere) such that $\Lip \alpha\approx L$, $\vol^{k-1} \alpha\approx L^{k-1}$, and
  \begin{equation}\label{eq:polyLower moment bound}
    \int_{S^{k-1}} \exp(b \dGamma(\alpha(x)))\; dx\lesssim 1,
  \end{equation}
  and there is a $\omega>0$ such that 
  \begin{equation}\label{eq:polyLower filling bound}
    \FV^k_{X(L/4)}(\fclass{\alpha})\gtrsim e^{\omega L},
  \end{equation}
  where $\fclass{\alpha}$ is the fundamental class of $\alpha$.  
\end{lemma}

\begin{proof}
  As in \cite{LeYoRank1}, our bound is based on the estimate of the
  divergence of $X$ in \cite{LeuCorank}.  It is shown there that there is an $\omega>0$ such that if $E\subset X$ is a flat, $x\in E$, $r>0$, and $S(x,r)$ is the $(k-1)$--sphere in $E$ with center $x$ and radius $r$, then
  \begin{equation}\label{eq:leuzinger divergence bound}
    \FV^k_{X\setminus B_x(\frac{r}{2})}(S(x,r)) \gtrsim e^{\omega r}.
  \end{equation}

  We first construct $\alpha$.    Let $g\in B_e(1)\subset G$ be a smooth random variable and let $L_0=\max\{2,2\diam K+\blog \|\phi_g\|_{l,2}\}$.  For any $r>0$, let $S(r)\subset A$ be the sphere of radius $r$ centered at the identity.  We claim that there are $c, \eta>0$ such that for all $L>L_0$ and $h_0\in G$ such that $\dGamma([h_0])=L$, 
  \begin{equation}\label{eq:random spheres eta blog}
    \Prob\bigl[ [h_0gS(cL)] \subset X(\eta \log L)\bigr]>0.
  \end{equation}
  That is, a random sphere of radius $cL$ centered near $[h_0]$ typically lies in $X(\eta \log L)$.  
  
  Let $c=2b'$, where $b'$ is as in Corollary~\ref{cor:expMomentsFlats}.    Since $L>L_0$, we have 
  $$b'(d_\Gamma(h_0)+\blog \|\phi_g\|_{l,2})\le b'(L+2\diam K+\blog \|\phi_g\|_{l,2})<cL.$$
  By Corollary~\ref{cor:expMomentsFlats}, this implies that there is a $c_0$ depending only on $G$ and $\Gamma$ such that for any $a\in S(cL)$, 
  $$\EE[\exp b\dGamma([h_0ga])]\le c_0.$$

  By Fubini's theorem,
  $$\EE\left[\int_{S(cL)}\exp b\dGamma([h_0ga]) \;da\right]
  =\int_{S(cL)}\EE[\exp b\dGamma([h_0ga])]\;da\le v_{k-1} c_0 L^{k-1}c^{k-1},$$
  where $v_{k-1}=\vol^{k-1}(S(1))$ is the volume of the unit sphere.  Let $c_1=v_{k-1}c_0 c^{k-1}$.  By Chebyshev's inequality, there is a $g_0\in B_e(1)$ such that 
  \begin{equation}\label{eq:choice of g0}
    \int_{S(cL)} \exp b\dGamma([h_0g_0a])\;da\le 2 c_1 L^{k-1}.
  \end{equation}

  Let $D=\max_{a\in S(cL)} \dGamma([h_0g_0a])$ and let $v\in S(cL)$ be such that $\dGamma([h_0g_0v])=D$.  For all $w\in B_v(1)$, we have $\dGamma([h_0g_0w])\ge D-1$, so
  $$\int_{S(cL)} \exp b\dGamma([h_0g_0a]) \;da \ge \vol^{k-1}(B_v(1)\cap S(cL)) \exp (b(D-1)) \gtrsim \exp(bD),$$
  and by \eqref{eq:choice of g0}, we have $\exp(bD)\lesssim L^{k-1}$.  That is, there is an $\eta>0$ depending only on $G$ and $\Gamma$ such that $D \le \eta \log L$ and thus $[h_0g_0S(cL)]\subset X(\eta \log L)$.  Let $\alpha$ parametrize $[h_0g_0S(cL)]$ by a scaling, so that $\Lip \alpha\approx L$ and $\vol^{k-1} \alpha\approx L^{k-1}$.  The bound \eqref{eq:polyLower moment bound} then follows from \eqref{eq:choice of g0}.

  Thus, for any $L>L_0$, there is a sphere $\alpha$ of radius $cL$ such that $\alpha\subset X(\eta \log L)$.  We claim that these spheres satisfy  \eqref{eq:polyLower filling bound} when $L>4+4\eta \log L$.  

  Let $v$ be the center of $\alpha$ (i.e., $v=[h_0g_0]$). Then  $\dGamma(v)\ge L-1$, so 
  $$d(v,X(L/4)) \ge L-1-\frac{L}{4} > \frac{L}{2}.$$
  Consequently, $X(\eta \log L)\subset X(L/4)\subset X\setminus B_{v}(L/2)$.  We have $\alpha\subset X(\eta \log L)\subset X(L/4)$, so by \eqref{eq:leuzinger divergence bound},
  $$\FV_{X(L/4)}(\fclass{\alpha})\ge \FV_{X\setminus B_{v}(L/2)}(\fclass{\alpha}) \gtrsim e^{\omega L},$$
  as desired.
\end{proof}

By combining this with the retraction from Theorem~\ref{thm:fundDomain}, we can construct a sphere in $\FD$.
\begin{prop}\label{prop:polyLower}
  There is an $\epsilon>0$ and an $\omega'>0$ such that when $V$ is
  sufficiently large,
  $\FV^k_{X_0}(V)\gtrsim e^{\omega' V^{\epsilon}}$.
\end{prop}
\begin{proof}
  Let $L>0$ and $\alpha\from S^{k-1}\to X(\eta \log L)$ be as in Lemma~\ref{lem:polyLower} and let $\rho \from X\to \FD$ and $c$ be as in Theorem~\ref{thm:fundDomain}.  Let $\tilde{\alpha}=\rho\circ\alpha\from S^{k-1}\to \FD$.  Since $\alpha(S^{k-1})\subset X(\eta\log L)$, 
  $$\Lip \tilde{\alpha} \lesssim L e^{c \eta\log L} \lesssim L^{c\eta+1},$$
  and thus there is a $C>0$ such that $\vol^{k-1} \tilde{\alpha}\le CL^{(k-1)(c\eta+1)}$.

  We claim that when $L$ is sufficiently large, $\FV^k_{X(L/4)}(\tilde{\alpha})\approx \FV^k_{X(L/4)}(\alpha)$.  It suffices to bound the volume of a homotopy between $\alpha$ and $\tilde{\alpha}$.  Let $H\from S^n(L)\times [0,1]\to X$ be the straight-line homotopy from $\alpha$ to $\tilde{\alpha}=\rho\circ \alpha$.  Every point is moved at most distance $\eta \log L$ by this homotopy, so
  $$\vol^k H\lesssim \eta \log L (\max\{ \Lip \alpha, \Lip \tilde{\alpha}\})^{k-1}\lesssim L^{(k-1)(c\eta+1)+1}.$$
  If $L$ is sufficiently large, then the image of $H$ lies in $X(L/4)$, so
  $$|\FV^k_{X(L/4)}(\tilde{\alpha})-\FV^k_{X(L/4)}(\alpha)|\le \vol^k H\lesssim L^{(k-1)(c\eta+1)+1}$$
  and $H(S^{k-1}\times [0,1])\subset X(2\eta \log L)$.

  When $2\eta \log L<\frac{L}{4}$, we have 
  $$\FV^k_{X(r_0)}(\tilde{\alpha})\ge \FV^k_{X(L/4)}(\alpha)-\vol^k(H)\gtrsim e^{\omega L}.$$
  By Lemma~\ref{lem:polyLower}, 
  $$\FV^k_{X(r_0)}(CL^{(k-1)(c\eta+1)})\gtrsim e^{\omega L}$$
  and thus, if $\epsilon=(k-1)^{-1}(c\eta+1)^{-1}$ and $V$ is
  sufficiently large, then
  $$FV^k_{X_0}(V)\ge\FV^k_{X(r_0)}(V)\gtrsim e^{\omega C^{-\epsilon}V^\epsilon}.$$
\end{proof}

This estimate is superpolynomial, but not sharp; we will find a better estimate in Section~\ref{sec:DiscsSpheres}.  That estimate is based on the same construction as the estimate above; the main difference is that, instead of using the retraction $\rho$ to construct a sphere in $\FD$, we use a Lipschitz extension result.  

\section{Parametrized cones in $X_\infty$}\label{sec:parametrizedCones}
Let $\Delta_{\FD}$ be the infinite-dimensional simplex with vertex set $\FD$ and let $S:=\Delta_{\FD}^{(k-1)}$ be its $(k-1)$--skeleton.  Let $\bz,\bz^*$ be opposite chambers in $E_\infty$, and for any chamber $\bb\in \chx$, let $\pi_\bb\from X_\infty \to \bb$ be the map that sends each chamber of $X_\infty$ to $\bb$ by a marking-preserving isomorphism.  For each chamber $\bb\in \cF^{k-1}(X_\infty)$, let $c_\bb$ be the barycenter of $\bb$; likewise, if $\delta\in \cF(S)$, let $c_\delta$ be the barycenter of $\delta$.

In this section, we will construct a family of maps $\sansP_{\bD}\from S\to X_\infty$ parametrized by an element $\bD=(\bd_{\delta})_{\delta\in \cF(S)}\in (\chxop)^{\cF(S)}$.  The full properties of this map are complicated to state and will appear in Proposition~\ref{prop:generic sansP} and Lemma~\ref{lem:bb}, but we will use it  to construct a family of maps $\Omega_\bD\from S\to X$,
\begin{equation}
  \Omega_\bD(x)=e_{f(x)}(\sansP_{\bD}(x),r(x)),
\end{equation}
where $r\from S\to [0,\infty)$ and $f\from S\to X$.  These are ``based polar coordinates'' in the sense that $\Omega_\bD(x)$ is the endpoint of a ray with origin $f(x)$, direction $\sansP_{\bD}(x)$, and length $r(x)$.  

Our goal is to construct $\sansP_\bD$ and a random variable $\bR$ so that for all $x\in S$, $\Omega_\bR(x)$ is a random variable of the form considered in Corollary~\ref{cor:expMomentsChambers}.  This requires $f$, $r$, and $\pi_\bz\circ \sansP_\bD$ to be independent of $\bD$.  

We will construct $\sansP_\bD$ so that $\pi_\bz\circ \sansP_\bD$ is independent of $\bD$ as long as $\bD$ satisfies the following general-position condition.  For every $\bb,\bc\in \chx$, we choose an apartment $E_{\bb,\bc}$ containing $\bb$ and $\bc$; when $\bb$ and $\bc$ are opposite, this choice is unique.  For every $\bD=(\bd_{\delta})_{\delta\in \cF(S)}$ and every vertex $v\in \cF^0(S)$, let $M_v(\bD)=\bd_v$.  Proceeding inductively, for every simplex $\delta\in \cF^i(S)$ with $i>0$, let
$$M_{\partial \delta}(\bD)=\bigcup_{\delta'\in \cF(\partial\delta)} M_{\delta'}(\bD)\subset X_\infty$$
and let
$$M_{\delta}(\bD)=\bigcup_{\bb \in \cF^{k-1}(M_{\partial \delta}(\bD))} (E_{\bb,\bd_{\delta}})_\infty$$
This is a union of finitely many chambers, and if $\delta'$ is a face of $\delta$, then $M_{\delta'}(\bD)\subset M_{\delta}(\bD)$.  

If $\bD\in (\chxop)^{\cF(S)}$, we say that $\bD$ is \emph{well-opposed} if for all $\delta\in \cF(S)$, $\bd_{\delta}$ is opposite to every chamber of $M_{\partial \delta}(\bD)$.  One can construct well-opposed $\bD$'s inductively, by choosing $\bd_{\delta}$ first for vertices, then edges, triangles, etc.; since $M_{\partial \delta}(\bD)$ has finitely many chambers, the set of choices of $\bd_{\delta}$ that satisfy the condition always has measure zero.

We will prove the following proposition.  
\begin{prop}\label{prop:generic sansP}
  There is a family of maps $\sansP_{\bD}\from S\to X_\infty$ such that for every well-opposed $\bD=(\bd_{\delta})_{\delta\in \cF(S)}\in (\chxop)^{\cF(S)}$, we have:
  \begin{enumerate}
  \item $\sansP_{\bD}(c_\delta)=c_{\bd_\delta}$ for every $\delta \in \cF(S)$.
  \item $\sansP_{\bD}(\delta)\subset M_\delta(\bD)$ for every $\delta\in \cF(S)$.
  \item $\Lip \sansP_{\bD}\lesssim 1$.
  \item If $\bE\in (\chxop)^{\cF(S)}$ is well-opposed, then $\pi_\bz\circ \sansP_{\bD}=\pi_\bz\circ \sansP_{\bE}.$
  \end{enumerate}
  Consequently, the map $z(x)=\pi_\bz(\sansP_{\bD}(x))$ is well-defined.  
\end{prop}



The construction of $\sansP_{\bD}$ is based on geodesic conings.  We start by letting $\sansP_{\bD}(v)=c_{\bd_v}$ for every $v\in \cF^0(S)$, then extend it to $S$ by induction.  Suppose that $\sansP_{\bD}$ is defined on $\partial \delta$.  After a perturbation, we may suppose that $\sansP_{\bD}(\partial \delta)$ avoids the barycenters of the chambers of $X_\infty$.  We can thus define $\sansP_{\bD}$ on $\delta$ by geodesic coning:  we send $c_\delta$ to $c_{\bd_\delta}$, and for every $x\in \partial \delta$, we send the ray from $c_\delta$ to $x$ to the unique minimal geodesic from $c_{\bd_\delta}$ to $\sansP_{\bD}(x)$.  When $\bD$ is well-opposed, the endpoints of this ray are in opposite chambers, so its projection is independent of $\bD$.

To formalize this construction, we need two lemmas.  The first lemma approximates Lipschitz maps by simplicial maps.
\begin{lemma}\label{lem:approx Lipschitz} 
  Let $Y, Z$ be simplicial complexes of dimension at most $n$.  Suppose that each simplex of $Z$ is isometric to a unit Euclidean simplex and that each simplex of $Y$ is $c$--bilipschitz equivalent to an equilateral Euclidean simplex of diameter $r$.  Let $\beta\from Z\to Z^{(0)}$ be a map such that for each $x\in Z$, $\beta(x)$ is the nearest vertex to $x$.

  There is an $\epsilon>0$ depending only on $n$ such that for any map $\alpha\from Y\to Z$ with $\Lip(\alpha)\le c^{-1}r^{-1}\epsilon$, there is a simplicial map $\kappa\from Y\to Z$ such that $\kappa(v)=\beta(\alpha(v))$ for all $v\in Y^{(0)}$.  Furthermore, for all $y\in Y$, the image $\kappa(y)$ is contained in the minimal simplex containing $\alpha(y)$, so there is a straight-line homotopy $h\from Y\times [0,1]\to Z$ such that $h_0=\alpha$, $h_1=\kappa$, and $\Lip h\lesssim c r^{-1}$, where $h_t(y)=h(y,t)$. 
\end{lemma}
\begin{proof}
  Let $\iota\from Z\to [0,1]^{\cF^0(Z)}\subset \ell_1(\cF^0(Z))$ be the embedding that sends each vertex $v$ to the characteristic function $\one_v$ and sends each simplex $\delta$ to the set of functions
  $$\{f\in [0,1]^{\cF^0(\delta)}\mid \sum_{v\in \cF^0(\delta)}f(v)=1\},$$
  This embedding is Lipschitz on each simplex, with constant $L$ depending on $n$.  The map $\beta$ sends each $x\in Z$ to a maximum of $\iota_x$; since $\iota_x$ is nonzero at all but at most $n+1$ points and $\|\iota_x\|_1=1$, we have 
  \begin{equation}\label{eq:iota maximizes}
    \iota_x(\beta(x))\ge \frac{1}{n+1}.
  \end{equation}

  For all $a,b\in Z$, we have
  $$\iota_b(\beta(a))\ge \iota_a(\beta(a))-Ld(a,b),$$
  so if $d(a,b)<\frac{1}{L(n+1)}$, then $\iota_b(\beta(a))>0$.  That is, $\beta(a)$ and $\beta(b)$ both lie in the support of $\iota_b$, so $\beta(a)$ and $\beta(b)$ are adjacent.

  Let $\epsilon=\frac{1}{2L(n+1)}$, let $\alpha$ be $c^{-1}r^{-1}\epsilon$--Lipschitz, and for each vertex $v\in \cF^0(Y)$, let $\kappa(v)=\beta(\alpha(v))$.  If $v$ and $w$ are adjacent vertices of $Y$, then 
  $$d(\alpha(v),\alpha(w))\le cr\Lip(\alpha)\le \frac{1}{2(n+1)},$$
  so $\kappa(v)=\beta(\alpha(v))$ and $\kappa(w)$ are also adjacent.  We extend $\kappa$ linearly on each simplex to obtain a simplicial map $Y\to Z$.  Since this map sends simplices to simplices, $\Lip(\kappa)\le cr^{-1}$.  

  Let $y\in Y$ and let $\delta$ be a simplex containing $y$.  Let $\Delta$ be the minimal simplex of $Z$ that contains $\alpha(y)$; that is, $\Delta=\langle \supp \iota_{\alpha(y)}\rangle$.  If $v$ is a vertex of $\delta$, then $d(\alpha(y),\alpha(v))<\frac{1}{2L(n+1)}$, so $$\iota_{\alpha(y)}(\kappa(v))\ge \frac{1}{L(n+1)}-L d(\alpha(y),\alpha(v))>0.$$
  Thus $\kappa(v)$ is a vertex of $\Delta$.  Since this holds for every vertex of $\delta$, we have $\kappa(\delta)\subset \Delta$, as desired, and we can define $h$ to be the straight-line homotopy from $\alpha$ to $\kappa$. This satisfies
  $$\Lip h\lesssim \max \{1,\Lip \alpha, \Lip \kappa\}\lesssim 1.$$
\end{proof}

We can combine this lemma with the next lemma to approximate maps with larger Lipschitz constants.
\begin{figure}
  \begin{tikzpicture}[scale=2.5]
    \draw ({sqrt(3)/2},-1/2) -- ({-sqrt(3)/2},-1/2) -- (0,1) -- cycle;
    \foreach \n in {0,.25,...,1.01} {
      \draw (0, {-1/2*\n}) -- ({sqrt(3)/4+sqrt(3)/4*\n}, {1/4 - 3/4*\n}); 
      \draw ({sqrt(3)/4*\n}, 1/4*\n) -- ({sqrt(3)/2*\n}, {-1/2}); 
      \draw (0, {-1/2*\n}) -- ({-sqrt(3)/4-sqrt(3)/4*\n}, {1/4 - 3/4*\n}); 
      \draw ({-sqrt(3)/4*\n}, 1/4*\n) -- ({-sqrt(3)/2*\n}, {-1/2}); 
      \draw ({-sqrt(3)/4*\n}, 1/4*\n) -- ({sqrt(3)/4-sqrt(3)/4*\n}, {1/4 + 3/4*\n}); 
      \draw ({sqrt(3)/4*\n}, 1/4*\n) -- ({-sqrt(3)/4+sqrt(3)/4*\n}, {1/4 + 3/4*\n}); 
    }
  \end{tikzpicture}
  \caption{\label{fig:cube subdiv} Partitioning a triangle into cells that are bilipschitz equivalent to cubes.}
\end{figure}
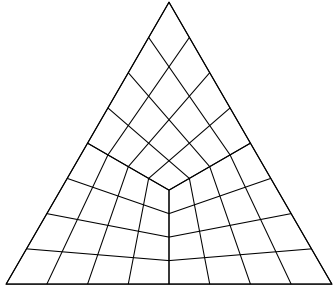
\begin{lemma}\label{lem:cube subdiv}
  For every $d>0$, there is a $c_d$ such that for every $0<r<1$, one can subdivide the standard unit $d$--simplex $\Delta$ into roughly $r^{-d}$ simplices, each of which is $c_d$--bilipschitz equivalent to the equilateral Euclidean simplex of diameter $r$.
\end{lemma}
\begin{proof}
  Let $n=\lfloor r^{-1}\rfloor$.  Consider the partition of $\Delta$ into $d+1$ Voronoi cells centered at the vertices of $\Delta$.  Each cell is bilipschitz equivalent to the unit cube.  We cellulate $\Delta$ by subdividing each of these cubes into $n^{-d}$ subcubes that are each bilipschitz equivalent to the cube of side length $\frac{1}{n}$ (see Figure~\ref{fig:cube subdiv}).  The barycentric subdivision of the resulting cubes is the desired triangulation.
\end{proof}

\subsection{Constructing $\sansP_{\bD}$}
Let $\bD=(\bd_{\delta})_{\delta\in \cF(S)}$.  For every vertex $v\in \cF^0(S)$, let $\sansP_{\bD}(v)=c_{\bd_\delta}$, in accordance with Proposition~\ref{prop:generic sansP}.(1).  Let $L_0=2\pi$, so that $\Lip \sansP_{\bD}|_{S^{(0)}}\le L_0$.  

Let $i\ge 0$.  Suppose by induction that $\sansP_{\bD}$ is defined on $S^{(i)}$, that $\Lip \sansP_{\bD}|_{S^{(i)}}\le L_i$, and that $\sansP_{\bD}(\lambda)\subset M_\lambda(\bD)$ for every $\lambda\in \cF^i(S)$.  Let $\delta\in \cF^{i+1}(S)$.  We extend $\sansP_{\bD}$ to $\delta$ as follows.  For every $\psi\in \partial \delta$ and $r\in [0,1]$, let $$(r,\psi)_\delta=r \psi+(1-r)c_\delta;$$
these are essentially polar coordinates on $\delta$.  Let 
$C:=\{(r,\psi)_\delta\mid r\ge\frac{1}{2}\}$
be a collar of $\partial \delta$ and let
$D:=\{(r,\psi)_\delta\mid r<\frac{1}{2}\}$
be its complement.

The restriction of $\sansP_{\bD}$ to $C$ will be a homotopy from $\sansP_{\bD} |_{\partial \delta}$ to a simplicial map.  Let $\epsilon>0$ be as in Lemma~\ref{lem:approx Lipschitz} and let $c_{i+1}$ be as in Lemma~\ref{lem:cube subdiv}.  Let $r_{i+1}=\epsilon (L_ic_{i+1})^{-1}$.  Let $\tau=\tau_{r_{i+1}}$ be a subdivision of $\partial \delta$ as in Lemma~\ref{lem:cube subdiv} so that each simplex of $\tau$ is $c_{i+1}$--bilipschitz equivalent to the equilateral Euclidean simplex of diameter $r_{i+1}$.

Choose a total order on the set of vertex types of $X_\infty$ and let $\beta\from X_\infty\to X_\infty^{(0)}$ be the map that sends each $x\in X_\infty$ to the closest vertex, breaking ties according to vertex type.  Since vertex type is preserved by the action of $G$, $\beta$ is $G$--equivariant and $\beta\circ \pi_\bz=\pi_\bz\circ \beta$.  By Lemma~\ref{lem:approx Lipschitz}, there is a simplicial map $\kappa_\bD\from \tau \to X_\infty$ such that $\kappa_\bD(v)=\beta(\sansP_{\bD}(v))$ for all $v\in \tau^{(0)}$ and a homotopy $h_\bD\from \partial \delta\times [0,1]\to X_\infty$ such that for all $\psi\in \partial \delta$ we have $h_\bD(\psi,1)=\sansP_{\bD}(\psi)$ and $h_\bD(\psi,0)=\kappa_\bD(\psi)$.  For all $t\ge \frac{1}{2}$, let 
$$\sansP_{\bD}((t,\psi)_\delta)=h_\bD(\psi,2(1-t)).$$
This extends $\sansP_{\bD}$ to $C$ and satisfies $\Lip \sansP_{\bD}|_C\lesssim_i 1$.  For all $\psi \in \partial \delta$, $h_\bD(\psi\times [0,1])$ is contained in the minimal simplex containing $\sansP_{\bD}(\psi)$.  Since $\sansP_{\bD}(\partial\delta)\subset M_{\partial \delta}(\bD)$, we have $\sansP_{\bD}(C)\subset M_{\partial \delta}(\bD)$.

Since $\dim \delta\le k-1$, $\sansP_{\bD}(\partial D)=\kappa_\bD(\partial\delta)$ lies in the $(k-2)$--skeleton of $X_\infty$ and thus does not contain the barycenter of any chamber.  It follows that for any $\psi\in \partial \delta$, there is a unique minimal geodesic $\gamma_\psi\from [0,1]\to X_\infty$ from $c_{\bd_\delta}$ to $\kappa_\bD(\psi)$.  For $t\le \frac{1}{2}$, let $$\sansP_{\bD}((t,\psi)_\delta)=\gamma_\psi(2t).$$
The length of $\gamma_\psi$ is at most $\pi$, so there is a $L_{i+1}$ independent of $\delta$ so that $\Lip(\sansP_{\bD}|_{\delta})\le L_{i+1}$.  Finally, $\kappa_\bD(\psi)\in M_{\partial \delta}(\bD)$, so there is a chamber $\bb\in \cF^{k-1}(M_{\partial\delta})$ that contains $\kappa_\bD(\psi)$.  Then $\gamma_\psi\subset (E_{\bb,\bd_\delta})_\infty \subset M_\delta$, so $\sansP_{\bD}(\delta)\subset M_\delta$.  This proves Proposition~\ref{prop:generic sansP}.(2)--(3).  

\subsection{Proving Proposition~\ref{prop:generic sansP}.(4)}
Let $\bD=(\bd_\delta)$ and $\bE=(\be_\delta)\in (\chxop)^{\cF(S)}$ be well-opposed.  We claim that for any well-opposed $\bD=(\bd_\delta)\in (\chxop)^{\cF(S)}$, we have 
\begin{equation}\label{eq:sansP independence}
  \pi_{\bz}\circ \sansP_{\bD}=\pi_{\bz}\circ \sansP_{\bE}.
\end{equation}
We proceed by induction.  For every vertex $v\in \cF^0(S)$, we defined $\sansP_\bD(v)=c_{\bd_v}$, so $\pi_{\bz}(\sansP_{\bD}(v))=\pi_{\bz}(\sansP_{\bE}(v))=c_\bz$, and \eqref{eq:sansP independence} holds on $S^{(0)}$.  

Let $i\ge 0$ and suppose by induction that \eqref{eq:sansP independence} holds on $S^{(i)}$.  Let $\delta\in \cF^{i+1}(S)$ and let $C$, $D$, $\tau$, $\beta$, and $\kappa_\bD$ be as in the construction of $\sansP_\bD|_\delta$.  We first consider $C$.  Since $\beta$ is $G$--equivariant, for any vertex $v\in \tau^{(0)}$, 
$$\pi_{\bz}(\kappa_\bD(v)) = \pi_{\bz}(\beta(\sansP_{\bD}(v))) = \beta(\pi_{\bz}(\sansP_{\bD}(v)))  = \beta(\pi_{\bz}(\sansP_{\bE}(v))) = \pi_{\bz}(\kappa_\bE(v)).$$
Then $\pi_{\bz}\circ \kappa_\bD$ and $\pi_{\bz}\circ \kappa_\bE$ are simplicial maps that agree on the vertices of $\tau$, so they are equal.  

Let $\psi\in \partial \delta$ and let $\lambda\in \cF(X_\infty)$ be a simplex containing $\sansP_\bD(\psi)$.  Then $\sansP_\bD(\psi)$ and $\kappa_\bD(\psi)$ both lie in $\lambda$, and the curve $r_\psi:=\sansP_{\bD}(([\frac{1}{2},1],\psi)_\delta)$ is the line segment in $\lambda$ that connects them.  The projection $\pi_\bz(r_\psi)$ is therefore the line segment from $\pi_{\bz}(\sansP_{\bD}(\psi))$ to $\pi_{\bz}(\kappa_\bD(\psi))$.  Both endpoints are independent of $\bD$, so \eqref{eq:sansP independence} holds on $r_\psi$ and thus on $C$. 

Now we consider $D$.  Let $\psi \in \partial \delta$, and let $\gamma\from [0,1]\to X_\infty$ be the minimal geodesic from $c_{\bd_\delta}$ to $\kappa_\bD(\psi)$, so that $\sansP_\bD((t,\psi)_\delta)=\gamma(2t)$ for all $t\le \frac{1}{2}$.  Since $\kappa_\bD(\psi)\in M_{\partial\delta}(\bD)$, there is a chamber $\bb$ of $M_{\partial\delta}(\bD)$ such that $\kappa_\bD(\psi)\in \bb$, and because $\bD$ is well-opposed, $\bd_\delta$ is opposite to $\bb$.  

Let $g\in G$ be an element such that $g\bb=\bz$, $g\bd_\delta=\bz^*$.  Then $g c_{\bd_\delta}=c_{\bz^*}$ and $g \kappa_\bD(\psi)=\pi_\bz(\kappa_\bD(\psi))$.  Let $\lambda \from [0,1]\to X_\infty$ be the minimal geodesic from  $c_{\bz^*}$ to $\pi_\bz(\kappa_\bD(\psi))$.  This is independent of $\bD$, and we have $g\gamma_\bD=\lambda$, so
$$\pi_{\bz}(\sansP_{\bD}((t,\psi)_\delta))=\pi_\bz(\lambda(t))$$
for all $t\le \frac{1}{2}$.  The right-hand side is independent of $\bD$, so equation \eqref{eq:sansP independence} holds on $\gamma$ for every $\psi$, and thus holds on $D$.  This proves Proposition~\ref{prop:generic sansP}.  

\subsection{Perturbations of $\bD$}
As in Proposition~\ref{prop:generic sansP}, let $z\from S\to \bz$, $z(x)=\pi_\bz(\sansP_{\bD}(x))$.  If $\bb_\bD\from S\to \chx$ is a function such that
\begin{equation}
  \label{eq:bb prop}
  \sansP_\bD(x)\subset \bb_\bD(x)\text{\qquad for all $x$ and $\bD$},
\end{equation}
then $\sansP_\bD(x)=\pi_{\bb_\bD(x)}(z(x))$.  The only part of this expression that depends on $\bD$ is $\bb_\bD$, and in this section, we will construct a map $\bb_\bD$ that satisfies \eqref{eq:bb prop} and varies smoothly with $\bD$.

The main idea of the construction is that the flat $E_{\bb,\bc}$ depends smoothly on $\bb$ and $\bc$ as long as $\bb$ and $\bc$ are opposite.  Similarly, for any $\delta\in \cF(S)$, $M_\delta(\bD)$ varies smoothly with $\bD$ as long as $\bD$ is well-opposed.  For all $x\in \delta$, $\sansP_\bD(x)\in M_\delta(\bD)$, and we use the fact that $M_\delta(\bD)$ varies smoothly with $\bD$ to show that $x$ varies smoothly with $\bD$.

We start by formalizing the statement that $E_{\bb,\bc}$ depends smoothly on $\bb$ and $\bc$.  Let $W$ be the Weyl group acting on $E_\infty$.  Let
$$\Theta=\{(\bb,\bc)\in \chx\times \chx\mid \bb\text{ is opposite to }\bc\}.$$
The set of chambers $\chx$ is the Furstenberg boundary of $X$, so it is equipped with the structure of a smooth manifold.  Then $\Theta$ is an open subset of $\chx\times \chx$ whose complement has codimension at least $1$.  (See \cite{Warner}, Prop. 1.2.4.9 or \cite{Helg}, Ch. IX, Cor. 1.8.)  

For $(\bb,\bc)\in \Theta$, let $\pi_{\bb,\bc}\from E_\infty \to (E_{\bb,\bc})_\infty$ be the marking-preserving isomorphism that sends $\bz$ to $\bb$ and $\bz^*$ to $\bc$.  For every $w\in W$, we define $m_w\from \Theta\to \chx$ by  
\begin{equation}\label{eq:defMw}
  m_w(\bb,\bc)=\pi_{\bb,\bc}(w\bz).
\end{equation}
Note that for all $(\bb,\bc)\in \Theta$ and $g\in G$, we have $m_w(g\bb,g\bc)=gm_w(\bb,\bc)$, so $m_w$ is $G$--equivariant with respect to the diagonal action on $\Theta$.  Furthermore, if $w^*$ is the maximal-length element of $W$, then $m_e(\bb,\bc)=\bb$, $m_{w^*}(\bb,\bc)=\bc$, and 
$$E_{\bb,\bc}=\bigcup_{w\in W} m_w(\bb,\bc).$$

\begin{lemma}\label{lem:mw submersion}
  For any $w\in W$, the map $m_w$ is a smooth submersion. 
\end{lemma}
\begin{proof}
  First, we claim that $m_w$ is smooth on a neighborhood of $(\bz,\bz^*)$.  Let $w^*\in W$ be the element of the Weyl group of maximal length, so that $w^* \bz= \bz^*$ and $w^*\bz^*=\bz$.  Let $f\from G\to \Theta$ be the map $f(g)=(g\bz,g\bz^*)$.  We claim that the derivative $D_ef\from T_eG\to T_{(\bz,\bz^*)}\Theta=T_{\bz}\chx\times T_{\bz^*}\chx$ is surjective.  The group $N$ acts transitively on $\chxop$ and stabilizes $\bz$, so $D_ef(\mathfrak{n})=0 \times T_{\bz^*}\chx$; likewise, $N^*=w^* N (w^*)^{-1}$ acts transitively on the set of chambers opposite to $\bz^*$ and stabilizes $\bz^*$, so $D_ef(\mathfrak{n}^*)=T_{\bz}\chx\times 0$, and $D_ef$ is surjective.

  By the implicit function theorem, there is locally a smooth section of $f$; i.e., a neighborhood $U\subset \Theta$ containing $(\bz,\bz^*)$ and a $g\from U\to G$ such that $g(\ba,\bb)(\bz, \bz^*)=(\ba,\bb)$ for all $(\ba,\bb)\in U$.  By the equivariance of $m_w$, we have
  $$m_w(\ba,\bb)=g(\ba,\bb)m_w(\bz,\bz*)=g(\ba,\bb)w\bz,$$
  so $m_w$ is smooth on $U$.  Since $G$ acts transitively on $\Theta$, this implies that $m_w$ is smooth on $\Theta$.

  Since $m_w$ is $G$--equivariant, the set of critical values of $m_w$ is also $G$--equivariant; i.e., it is either empty or all of $\chx$.  By Sard's Theorem, it must be empty, so $m_w$ is a submersion.

\end{proof}

We can use the $m_w$ to describe $M_\delta(\bD)$ and define $\bb_\bD$.

\begin{lemma}\label{lem:bb}
  There is a family of maps $w_\delta\from \inter \delta\to W$ with the following property.  For every well-opposed $\bD$, let $\bb_\bD\from S\to \chxop$ be the map defined inductively by $\bb_\bD(c_v)=\bd_v$ for all $v\in \cF^0(S)$ and, for every $i>0$, every $\delta\in \cF^i(S)$, and every $x=(\psi,t)_\delta\in \inter \delta$, 
  \begin{equation}\label{eq:bb inductive}
    \bb_\bD(x)=m_{w_\delta(x)}(\bb_\bD(\psi), \bd_\delta).
  \end{equation}
  Then:
  \begin{enumerate}
  \item\label{eq:bb contain} 
    For every $x\in S$, we have $\sansP_{\bD}(x)\in \bb_\bD(x)$ and thus $\sansP_{\bD}(x)=\pi_{\bb_\bD(x)}(z(x)).$
  \item \label{it:bb in mw}
    For every $\delta\in \cF(S)$ and every $x\in \delta$, $\bb_\bD(x)\in \cF(M_\delta(\bD))$.
  \item \label{it:bb collar} 
    For every $\delta\in \cF(S)$, $\psi\in \partial \delta$, and $t\in [\frac{1}{2},1]$, we have $\bb_\bD((t,\psi)_\delta)=\bb_\bD(\psi)$.
  \end{enumerate}
\end{lemma}
\begin{proof}
  Let $\delta\in \cF(S)$.  Let $\Pi\from X_\infty\to E_\infty$ be the canonical projection of $X_\infty$ centered at $\bd_\delta$, so that $\Pi(\bd_\delta)=\bz^*$ and the restriction of $\Pi$ to any apartment containing $\bd_\delta$ is an isomorphism.  Since every chamber of $M_{\partial\delta}(\bD)$ is opposite to $\bd_\delta$, we have $\Pi(m)=\pi_{\bz}(m)$ for all $m\in M_{\partial\delta}(\bD)$ and thus
  $\Pi(\sansP_\bD(y))=z(y)$ for all $y\in \partial \delta$.  

  Let $\psi\in \partial \delta$.  Let $\gamma_\psi\from [0,1]\to X_\infty$ be the curve $\gamma_\psi(t)=\sansP_{\bD}((t,\psi)_\delta)$.  This is a broken geodesic from $\gamma_\psi(0)=c_{\bd_\delta}$ through $\gamma_\psi(\frac{1}{2})=\kappa_\bD(\psi)$ to $\gamma_\psi(1)=\sansP_{\bD}(\psi)$.  Since $\kappa_\bD(\psi)$ and $\sansP_{\bD}(\psi)$ lie in $M_{\partial\delta}(\bD)$, the composition $\Pi\circ \gamma_\psi$ consists of the broken geodesic from $c_{\bz^*}$ through $\pi_{\bz}(\kappa_\bD(\psi))=z((\frac{1}{2},\psi)_\delta)$ to $z(\psi)$.  In particular, $\Pi\circ \gamma_\psi$ is independent of $\bD$.  Since this holds for every $\psi$, the composition $\Pi\circ \sansP_\bD|_{\delta}$ is independent of $\bD$.

  For all $x\in \delta$, let $w_\delta(x)$ be the shortest element of $W$ such that $\Pi(\sansP_\bD(x)) \in w_\delta(x)\bz$.  This is independent of $\bD$.  Furthermore, if $\sansP_\bD(x)\in M_{\partial\delta}(\bD)$, then $\Pi(\sansP_\bD(x))\in \bz$, so $w_\delta(x)=e$.  In particular, for all $t\in [\frac{1}{2},1)$ and all $\psi\in \partial \delta$, we have $w_\delta((t,\psi)_\delta)=e$ and thus $\bb_\bD((t,\psi)_\delta)=\bb_\bD(\psi)$; this proves property \ref{it:bb collar}.

  Since $\bD$ is well-opposed, property \ref{it:bb in mw} follows by induction from the definitions.  That is, if $\delta\in \cF(S)$, $(\psi,t)_\delta\in \delta$, and $\bb_\bD(\psi)\in \cF(M_{\partial\delta}(\bD))$, then $\bb_\bD((\psi,t)_\delta)$ is a face of $E_{\bb_\bD(\psi),\bd_\delta}$ and thus is a face of $M_{\delta}(\bD).$

  We use a similar induction to prove \eqref{eq:bb contain}.  If $v\in S^{(0)}$, then $\sansP_{\bD}(v)=c_{\bd_v}\in \bd_v=\bb_\bD(v)$.  Let $i>0$ and $\delta\in \cF^i(S)$ and suppose by induction that $\sansP_{\bD}(\psi)\in \bb_\bD(\psi)$ for all $\psi\in \partial\delta$.  Let $F_\psi=(E_{\bb_\bD(\psi), \bd_\delta})_\infty$.  

  Let $t\in [0,1)$ and let $x=(t,\psi)_\delta$.  By \eqref{eq:bb inductive}, we have $\bb_\bD(x)\subset F_\psi$; in fact, $\Pi|_{F_\psi}$ is a marking-preserving isomorphism, so 
  $$\bb_\bD(x)=(\Pi|_{F_\psi})^{-1}(w_\delta(x) \bz).$$
  By induction and Lemma~\ref{lem:approx Lipschitz}, we have $\sansP_{\bD}(\psi),\kappa_\bD(\psi)\in \bb_\bD(\psi)$, so the broken geodesic $\gamma_\psi$ defined above lies in $F_\psi$.  Thus $\sansP_\bD(x)=\gamma_\psi(t) \in F_\psi$, and
  $$\sansP_\bD(x)=(\Pi|_{F_\psi})^{-1}(\Pi(\sansP_\bD(x)))\in (\Pi|_{F_\psi})^{-1}(w_\delta(x) \bz)=\bb_\bD(x),$$
  as desired.
\end{proof}

\section{Random cones in $X_\infty$}\label{sec:random cones}
In this section, we will construct a random variable $\bR\in (\chxop)^{\cF(S)}$ and consider the random map $\sansP_{\bR}$.  Our first step is to construct the subset of $(\chxop)^{\cF(S)}$ on which $\bR$ is supported.  We will prove the following proposition.  For each $\delta\in \cF(S)$, define the \emph{size} of $\delta$ as
\begin{equation}\label{eq:defBarS}
  \sigma(\delta):=\diam_X \cF^0(\delta)+1.
\end{equation}
\begin{prop}\label{prop:basepointChoices}
  There is an $\epsilon>0$ and a collection of points $g_\delta\in NA$, $x_{\delta}:=[g_\delta]\in X$ indexed by $\delta\in \cF(S)$ such that
  \begin{enumerate}
  \item  \label{it:bc well-opposed}
    If $\bd_\delta\in \cSc_{x_\delta}(\epsilon)$ for all $\delta\in \cF(S)$, then $\bD=(\bd_\delta)$ is well-opposed and $\cF^{k-1}(M_{\delta}(\bD))\subset \cSc_{x_\delta}$ for every $\delta\in \cF(S)$.  Thus, for all $x\in \delta$, we have $\bb_\bD(x)\in \cSc_{x_\delta}$
  \item \label{it:bc opposition}
    Let $\delta', \delta\in \cF(S)$ be such that $\delta'\subsetneq \delta$.  If $\bb\in \cSc_{x_{\delta'}}$ and $\bc\in \cSc_{x_\delta}(\epsilon)$, then $\bb$ is opposite to $\bc$ and 
    $$\cF^{k-1}((E_{\bb,\bc})_\infty)\subset \cSc_{x_\delta}.$$
  \item \label{it:bc vertices}
    For each vertex $v\in \cF^0(S)=\FD$, $x_v=v$.
  \item \label{it:bc diameters} 
    For all $\delta\in \cF(S)$ and all $v\in \cF^0(\delta)$, we have $d(x_\delta, v)\lesssim \sigma(\delta)$.  Consequently, if $\delta'\subset \delta$, then
    $$d(x_{\delta'},x_{\delta}) \lesssim \sigma(\delta),$$
    and 
    \begin{equation}\label{eq:height of xDelta}
      d_{[\Gamma]}(x_\delta)\lesssim d_{[\Gamma]}(v)+ \sigma(\delta)\lesssim \sigma(\delta).
    \end{equation}
  \item \label{it:bc faces}
    If $\delta'\subset \delta$, then $\cSc_{x_{\delta'}}\subset \cSc_{x_\delta}.$
  \end{enumerate}
\end{prop}

Let $\Theta$ and $m_w\from \Theta\to \chx$ be as in the previous section.  In order to prove Proposition~\ref{prop:basepointChoices}, we will need the following lemma, which describes the behavior of $m_w$ on $\cSc_x(\epsilon)\times \cSc_y(\epsilon)$ for a particular pair of points $x$ and $y$.
\begin{lemma}\label{lem:pairOfShadows}
  There is an $n_0\in N$ and an $\epsilon>0$ such that $\cSc_{[n_0]}(\epsilon)\subset \cSc_{[e]}$, $\cSc_{[e]}(\epsilon)\times\cSc_{[n_0]}(\epsilon)\subset \Theta$, and for all $w\in W$,
  \begin{equation}\label{eq:mwDefined}
    m_w(\cSc_{[n_0]}(\epsilon)\times \cSc_{[e]}(\epsilon)) \subset \cSc_{[e]}.
  \end{equation}
  Furthermore, there is a $D>0$ such that if $\phi\in C^\infty(\cSc_{[e]}(\epsilon))$ and $\psi \in C^\infty(\cSc_{[n_0]}(\epsilon))$, then
  \begin{equation}\label{eq:mwSobolev}
    \|(m_w)_*(\phi\times \psi)\|_{e}\le D  \|\phi\|_{e} \|\psi\|_{e}.
  \end{equation}
\end{lemma}

\begin{proof}
  By Lemma~\ref{lem:mw submersion} and the fact that a generic chamber is opposite to $\bz$, there are two opposite chambers $\bb_1$ and $\bb_2$ such that every chamber of $(E_{\bb_1,\bb_2})_\infty$ is opposite to $\bz$.  For every $n\in N$, let $\bc_n=n\bz^*$, so that $\cSc_{[n]}(\epsilon)$ is a neighborhood of $\bc_n$.  Let $n_1,n_2$ be such that $\bb_i=\bc_{n_i}$, let $n=n_1^{-1}n_2$, and let $F=E_{\bc_e,\bc_n}$.  Then $F_\infty=n_1^{-1}(E_{\bb_1,\bb_2})_\infty$, so every chamber of $F$ is opposite to $\bz$.  

  By Lemma~\ref{lem:dil}, there is an element $\alpha\in A$ (indeed, a vector in the direction of the barycenter of $\bz$) such that $\cF^{k-1}(F_\infty)\subset \cSc_{[\alpha]}(\frac{1}{2})$ and thus $\cF^{k-1}(\alpha^{-1}F_\infty)\subset \cSc_{e}(\frac{1}{2})$.  We have
  $$\alpha^{-1}F=E_{\alpha^{-1}\bc_e, \alpha^{-1}\bc_n}=E_{\bc_e,\bc_{\alpha^{-1}n\alpha}},$$
  and we let $n_0=\alpha^{-1} n\alpha$.  

  For any $w\in W$, 
  $$m_w(\bc_{n_0},\bc_{e})\in \cF^{k-1}(\alpha^{-1}F_\infty)\subset \cSc_{[e]}(\frac{1}{2}).$$
  In particular, $\bc_{n_0}\in \cSc_{[e]}(\frac{1}{2}),$ so $\cSc_{[n_0]}(\frac{1}{2})\subset \cSc_{[e]}$.  By Lemma~\ref{lem:pairOfShadows}, $\cSc_{[e]}(\epsilon)\times\cSc_{[n_0]}(\epsilon)\subset \Theta$ and \eqref{eq:mwDefined} holds when $\epsilon$ is sufficiently small.  Since $m_w$ is a submersion, Lemma~\ref{lem:SobolevProps} implies that \eqref{eq:mwSobolev} holds for sufficiently small $\epsilon$.
\end{proof}

Now we prove Proposition~\ref{prop:basepointChoices}.
\begin{proof}[{Proof of Proposition~\ref{prop:basepointChoices}}]
  Let $n_0\in N$ and $\epsilon>0$ be as in Lemma~\ref{lem:pairOfShadows}.  Let $\alpha\in A$ be the unit vector in the direction of the barycenter of $\bz$ and let $\alpha^t\subset A$ be the one-parameter subgroup generated by $\alpha$.  By Lemma~\ref{lem:largeShadows}, there is a $C>1$ such that for all $g, h\in NA$ and all $t\ge C(d(g,h)+1)$,
  \begin{equation}\label{eq:alphaShrink}
    \cSc_{[h]}\subset \cSc_{[g\alpha^t]}(\epsilon).
  \end{equation}

  Let $C_0=C$, and for $i\ge 1$, let $C_i=C(C_{i-1}+3)$.  For each vertex $v\in \cF^0(\delta)$, let $g_v\in NA$ be an element such that $[g_v]=v$.  For each $\delta\in \cF(S)$ with $\dim \delta\ge 1$, choose a vertex $v(\delta)\in \cF^0(\delta)$ and, for $d:=\dim\delta$, set
  $$g_\delta:=g_{v(\delta)}\alpha^{C_d\sigma(\delta)}n_0^{-1}.$$
  We claim that this choice of $g_\delta$ satisfies the desired properties.  By definition, it satisfies property \ref{it:bc vertices}.

  We have $d(g_{v(\delta)},g_\delta)\le C_d\sigma(\delta)+1\lesssim \sigma(\delta)$, so for an arbitrary vertex $v$ of $\delta$
  $$d(g_v,g_{\delta})\lesssim d(g_v,g_{v(\delta)})+ d(g_{v(\delta)},g_{\delta})\lesssim \sigma(\delta).$$
  Thus property \ref{it:bc diameters} holds.  

  To prove the remaining properties, we will show that for all $\delta'\subsetneq \delta$, 
  \begin{equation}\label{eq:bc faces intermediate}
    \cSc_{x_{\delta'}}\subset \cSc_{[g_{\delta}n_0]}(\epsilon)\subset \cSc_{x_{\delta}};
  \end{equation}
  Let $d=\dim \delta$, $d'=\dim \delta'$.  Then
  $$d(g_{v(\delta)},g_{\delta'})\le d(g_{v(\delta)},g_{v(\delta')})+C_{d'} \sigma(\delta')+1\le (C_{d'}+2)\sigma(\delta).$$
  Since $d>d'$, we have
  $$C(d(g_{v(\delta)},g_{\delta'})+1)\le C (C_{d'}+3)\sigma(\delta)\le C_d \sigma(\delta);$$
  and, as $g_{v(\delta)}\alpha^{C_d\sigma(\delta)}=g_{\delta}n_0$, \eqref{eq:alphaShrink} implies $\cSc_{x_{\delta'}}\subset \cSc_{[g_{\delta}n_0]}(\epsilon)$.  The second inclusion in \eqref{eq:bc faces intermediate} follows from the fact that $\cSc_{[n_0]}(\epsilon)\subset \cSc_{[e]}$.  This proves property \ref{it:bc faces}.

  We prove properties \ref{it:bc well-opposed} and \ref{it:bc opposition} using \eqref{eq:bc faces intermediate}.  Let $\delta', \delta\in \cF(S)$ be such that $\delta'\subsetneq \delta$ and let $\bb\in \cSc_{x_{\delta'}}$ and $\bc\in \cSc_{x_\delta}(\epsilon)=\cSc_{[g_\delta]}(\epsilon)$.  By \eqref{eq:bc faces intermediate}, we have $\bb\in \cSc_{[g_{\delta}n_0]}(\epsilon)$, so by Lemma~\ref{lem:pairOfShadows}, $m_w(\bb,\bc)\in \cSc_{x_\delta}$ for all $w\in W$; that is, $\cF^{k-1}((E_{\bb,\bc})_\infty)\subset \cSc_{x_\delta}$.  This proves property \ref{it:bc opposition}.

  Property \ref{it:bc well-opposed} follows by applying property \ref{it:bc opposition} inductively.  Suppose that $\bd_\delta\in \cSc_{x_\delta}(\epsilon)$ for all $\delta\in \cF(S)$.  We claim that $\cF^{k-1}(M_{\delta}(\bD))\subset \cSc_{x_\delta}$ for every $\delta\in \cF(S)$.  Suppose that the property holds for every simplex of dimension at most $i$; this is true by hypothesis when $i=0$.  

  Let $\delta\in \cF^{i+1}(S)$ and let $\delta'$ be a proper face of $\delta$.  Let $\bb\in \cF^{k-1}(M_{\delta'}(\bD))$.  By induction, $\bb\in \cSc_{x_{\delta'}}$; by property \ref{it:bc opposition}, $\bb$ is opposite to $\bd_\delta$, and
  $$\cF^{k-1}((E_{\bb,\bd_\delta})_\infty)\subset \cSc_{x_\delta}.$$
  Therefore, $\bD$ is well-opposed and
  $$\cF^{k-1}(M_\delta(\bD))=\bigcup_{\bb\in \cF^{k-1}(M_{\partial\delta}(\bD))} \cF^{k-1}((E_{\bb,\bd_{\delta}})_\infty)\subset \cSc_{x_\delta},$$
  as desired.
\end{proof}

Finally, we construct $\bR$.  
\begin{defn}\label{def:bR}
  Let $x_\delta$, $g_\delta$, and $\epsilon$ be as in Proposition~\ref{prop:basepointChoices}.  (In particular, suppose that  $\epsilon$ satisfies Lemma~\ref{lem:pairOfShadows}.)  Let
  $$\cU:=\{(\bd_\delta)_{\delta\in\cF(S)} \in (\chxop)^{\cF(S)}\mid \bd_\delta\in\cSc_{x_\delta}(\epsilon)\text{ for all }\delta\in \cF(S)\}.$$

  Choose a distribution function $\phi_0\in C^\infty(\cSc_{[e]}(\epsilon))$ and let $\bR=(\br_\delta)_{\delta\in \cF(S)}\in \cU$ be a random variable whose coordinates $\br_\delta\in \cSc_{x_\delta}(\epsilon)$ are independent random variables with distribution function $g_\delta\phi_0$.  By Proposition~\ref{prop:basepointChoices}.\ref{it:bc well-opposed}, $\bR$ is well-opposed, and for every $\delta\in \cF(S)$ and every $x\in \delta$, we have $\bb_\bR(x)\in \cSc_{x_\delta}$.
\end{defn}

\begin{lemma}\label{lem:qSmoothRandom}
  There is a $c>0$ such that for every $\delta\in \cF(S)$ and every $x\in \delta$,
  \begin{equation}\label{eq:qSmoothRandom}
    \|\phi_{\bb_\bR(x)}\|_{x_{\delta}}\lesssim \exp(c\sigma(\delta)).
  \end{equation}
\end{lemma}
\begin{proof}
  We proceed by induction.  If $v\in S$ is a vertex, then $x_v=v$ and $\bb_\bR(v)=\br_v$.  This is a smooth random chamber satisfying $\|\phi_{\bb_\bR(v)}\|_v=\|g_v^{-1} \phi_{\br_v}\|_e=\|\phi_0\|_{e}$, so $\|\phi_{\bb_\bR(v)}\|_v\lesssim 1$.  

  Let $\epsilon$, $n_0$, and $D$ be as in Lemma~\ref{lem:pairOfShadows} so that for all $w\in W$, $\phi\in C^\infty(\cSc_{[e]}(\epsilon))$, and $\psi \in C^\infty(\cSc_{[n_0]}(\epsilon))$, the push-forward $(m_w)_*(\phi\times \psi)$ is smooth and satisfies     $\|(m_w)_*(\phi\times \psi)\|_{e}\le D  \|\phi\|_{e} \|\psi\|_{e}$.  Let $\delta\in \cF(S)$ and let $i=\dim \delta$.  Suppose by induction that there is a $C_{i-1}>0$ such that for any proper face $\lambda\subsetneq \delta$ and any $y\in \lambda$,
  \begin{equation}\label{eq:qSmooth induction}
    \|\phi_{\bb_\bR(y)}\|_{x_{\lambda}}\le C_{i-1} \exp(C_{i-1} \sigma(\lambda)).
  \end{equation}
  Let $w_\delta\from \inter \delta\to W$ be as in Lemma~\ref{lem:bb}, and let $t\in [0,1)$, $\psi\in \partial \delta$, and $x=(t,\psi)_\delta \in \inter\delta$.  Consider
  \begin{align*}
    g_\delta^{-1}\phi_{\bb_\bR(x)} 
    &=g_\delta^{-1} (m_{w_\delta(x)})_*(\phi_{\bb_\bR(\psi)}\times \phi_{\br_\delta})\\
    &=(m_{w_\delta(x)})_*(g_\delta^{-1}\phi_{\bb_\bR(\psi)}\times \phi_0).
  \end{align*}

  Let $\lambda$ be a proper face of $\delta$ that contains $\psi$.  By \eqref{eq:bc faces intermediate}, 
  $$g_\delta^{-1}\phi_{\bb_\bR(\psi)}\in C^{\infty}(g_\delta^{-1} \cSc_{x_\lambda})\subset C^{\infty}(g_\delta^{-1} \cSc_{[g_\delta n_0]}(\epsilon)) = C^{\infty}(\cSc_{[n_0]}(\epsilon)).$$
  Let $c$ be as in Lemma~\ref{lem:expDecay}.  By \eqref{eq:mwSobolev} and \eqref{eq:qSmooth induction}, there is a $C_i>C_{i-1}$ such that
  \begin{align*}
    \|\phi_{\bb_\bR(x)}\|_{x_\delta} & =\|g_\delta^{-1} \phi_{\bb_\bR(x)}\|_{e}\\
    &\le D \|g_\delta^{-1} \phi_{\bb_\bR(\psi)}\|_e \cdot\|\phi_0\|_e \\ 
    &= D \|\phi_{\bb_\bR(\psi)}\|_{x_\delta}\cdot \|\phi_0\|_e \\
    &\le D \exp(c d(x_\delta,x_\lambda)) \|\phi_{\bb_\bR(\psi)}\|_{x_{\lambda}} \cdot\|\phi_0\|_e \\
    &\le D C_{i-1} \exp(C_{i-1} \sigma(\lambda) + c d(x_\delta,x_\lambda)) \|\phi_0\|_e \\
    &\le C_i \exp(C_{i} \sigma(\delta)),
  \end{align*}
  as desired.
\end{proof}

\section{Random maps to $X$}\label{sec:randomMapsX}

In the previous sections, we constructed a random map $\sansP_{\bR}\from S\to X_\infty$, a family of smooth random chambers $\bb_\bR(x)$, $x\in S$, and a (deterministic) function $z\from S\to \bz$ such that $\sansP_\bR(x)=\pi_{\bb_\bR(x)}(z(x))$.  Now we will use $\sansP_\bR$ to construct a random map $\Omega_\bR\from S\to X$ such that with high probability, most of the image of $\Omega_\bR$ lies close to $[\Gamma]$.  

\begin{prop}\label{prop:OmegaBounds}
  There are maps $f\from S\to X$ and $r\from S\to \R$ such that if 
  \begin{equation}\label{eq:OmegaDef}
    \Omega_{\bD}(x):=e_{f(x)}(\sansP_\bD(x),r(x)),
  \end{equation}
  then for all $v\in \cF^0(S)=\FD$, all well-opposed $\bD$, and all $\delta\in \cF(S)$, we have $\Omega_{\bD}(v)=v$ and 
  \begin{equation}\label{eq:OmegaLocallyLip}
    \Lip \Omega_{\bD}|_{\delta}\lesssim \sigma(\delta).
  \end{equation}
  Furthermore, there is a $b>0$ such that for any $x\in S$, 
  \begin{equation}\label{eq:OmegaExpMoments}
    \EE[\exp(b \dGamma(\Omega_\bR(x)))]\lesssim 1.
  \end{equation}
\end{prop}
\begin{proof}
  Let the points $x_\delta$ be as in Proposition~\ref{prop:basepointChoices}.  For all $v\in \cF^0(S)$, let $f(v)=v$.  We define $f$ inductively; for $\delta\in \cF^{i+1}(S)$ and $\psi\in \partial \delta$, let $\gamma\from [0,1]\to X$ be the geodesic from $x_\delta$ to $f(\psi)$.  Let
  $$f((t,\psi)_\delta)=\begin{cases}
    x_\delta & t\le \frac{1}{2} \\
    \gamma_{\psi}(2t-1) & t\ge \frac{1}{2}.
  \end{cases}$$
  Since $X$ is a CAT(0) space, this map is continuous, and
  $$\Lip f|_{\delta} \lesssim \max\{\Lip f|_{\partial \delta}, \max_{\delta'\subset \delta} d(x_\delta, x_{\delta'})\}.$$
  By Proposition~\ref{prop:basepointChoices} and induction on $i$, this implies that there is a $L>0$ such that 
  \begin{equation}\label{eq:Lip OmegaX}
    \Lip f|_{\delta} \le L\sigma(\delta).
  \end{equation}

  Let $R\from S\to \R$ be the function
  $$R(s)=\dGamma(f(s))+ \blog \|\phi_{\bb_\bR(s)}\|_{f(s)}.$$
  Let $\delta\in \cF(S)$ and let $s\in \delta$.  By \eqref{eq:Lip OmegaX}, we have $\diam f(\delta)\lesssim \sigma(\delta)$.  Since $f(v)=v$ for all $v\in \cF^0(S)$, this implies $\dGamma(f(s))\lesssim\sigma(\delta)$.  Combining this with Lemma~\ref{lem:expDecay} and Lemma~\ref{lem:qSmoothRandom}, we get
  $$R(s)\lesssim \sigma(\delta) + d(x_\delta,f(s)) + \blog \|\phi_{\bb_\bR(s)}\|_{x_\delta}\lesssim \sigma(\delta).$$
  Let $C_0>0$ be such that $R(s)\le C_0 \sigma(\delta)$ for all $s\in \delta$.

  Let $a>0$ be as in Lemma~\ref{lem:expDecay} and let $b, b'>0$ be as in Corollary~\ref{cor:expMomentsChambers}.  Let $C_1=2b' C_0+(a+1) b' L \pi.$
  We define $r$ by induction.  For $v\in \cF^0(S)$, let $r(v)=0$.  Suppose by induction that $i\ge 1$ and we have defined $r$ on $S^{(i-1)}$.  Let $\delta\in \cF^i(S)$.  For $r\in[0,1]$, $\psi\in \partial\delta$, define
  $$r((t,\psi)_\delta)=\min\{b' C_0\sigma(\delta), r(\psi) + (1-t) C_1 \sigma(\delta)\}.$$
  One can check that $r((t,\psi)_\delta)=b' C_0\sigma(\delta)$ for all $t\le \frac{1}{2}$ and that $\Lip r|_\delta\lesssim \sigma(\delta)$.  

  Let $\Omega_\bD$ be as in \eqref{eq:OmegaDef}.  This satisfies $\Omega_{\bD}(v)=v$ for all $v\in \cF^0(S)$ by definition.  We claim that $\Lip \Omega_{\bD}|_{\delta}\lesssim \sigma(\delta)$ for all $\delta\in \cF(S)$.  In order to show this, we first show that there is a $\rho>0$ such that for all $\bD\in \cU$, all $\delta\in \cF(S)$ and all $s\in \delta$, we have $\bb_\bD(s)\in \cSc_{f(s)}(\rho)$.  

  We proceed by induction on dimension.  Let $\rho_0=1$; when $v\in S^{(0)}$, we have $\bb_\bD(v)\in \cSc_v(\epsilon)\subset \cSc_v$.  Suppose that $i\ge 0$ and that there is a $\rho_i\ge 1$ such that $\bb_\bD(s)\in \cSc_{f(s)}(\rho_i)$ for all $s\in S^{(i)}$.  Let $\delta\in \cF^{i+1}(S)$ and let $s=(t,\psi)_\delta$.  If $t\le \frac{1}{2}$, then $f(s)=x_\delta$, and by Lemma~\ref{lem:qSmoothRandom}, $\bb_\bD(s)\in \cSc_{x_\delta}$.  

  Otherwise, $t\ge\frac{1}{2}$, and $f(s)$ lies on the geodesic from $f(\psi)$ to $x_\delta$.  Lemma~\ref{lem:convex hull shadows} implies that it suffices to show that $\bb_\bD(s)$ lies in the shadows of $f(\psi)$ and $x_\delta$.  By Lemma~\ref{lem:bb}.\ref{it:bb collar} and induction, we have
  $$\bb_\bD(s)=\bb_\bD(\psi)\in \cSc_{f(\psi)}(\rho_i),$$
  and by Proposition~\ref{prop:basepointChoices}.\ref{it:bc well-opposed}, we have $\bb_\bD(s)\in \cSc_{x_\delta}$.  It follows that there is a $\rho_{i+1}$ such that $\bb_\bD(s)\in \cSc_{f(s)}(\rho_{i+1})$.

  Thus, for all $s\in S$, we have $\bb_\bD(s)\in \cSc_{f(s)}(\rho_{k-1})$.  By Lemma~\ref{lem:ex locally Lipschitz}, for each simplex $\delta\in \cF(S)$, 
  \begin{equation}\label{eq:Omega local Lip formula}
    \Lip \Omega_\bD|_{\delta}\lesssim \Lip(z|_\delta)\max(r(\delta)) + \Lip(r|_\delta) + \Lip(f|_\delta)\lesssim \sigma(\delta).
  \end{equation}

  It remains to prove \eqref{eq:OmegaExpMoments}.  Let $C_2=b' \max(\dGamma(\FD)) + \blog \|\phi_0\|_{l,2}$.  We will show that 
  \begin{equation}\label{eq:main criterion r}
    r(s)\ge b' R(s)-C_2
  \end{equation}
  for all $s\in S$, then apply Corollary~\ref{cor:expMomentsChambers}.

  By construction, \eqref{eq:main criterion r} holds when $s$ is a vertex of $S$.  Suppose that \eqref{eq:main criterion r} holds on $\partial \delta$ and let $s=(t,\psi)_\delta$.  If $r(s)=b' C_0\sigma(\delta)$ (in particular, if $t\le \frac{1}{2}$), then $r(s)=b' C_0\sigma(\delta)\ge b' R(s)$, so \eqref{eq:main criterion r} holds.  

  We thus suppose that $t\ge \frac{1}{2}$ and $r(s)= r(\psi) + (1-t) C_1 \sigma(\delta).$  Then $\bb_\bR(s)=\bb_\bR(\psi)$, and by Lemma~\ref{lem:expDecay}, 
  \begin{multline*}
    |R(s)-R(\psi)|\le d(f(s),f(\psi))+a d(f(s), f(\psi))\\ \le (a+1) L \sigma(\delta) d(s,\psi)\le (a+1) L \sigma(\delta) (1-t) \pi\le \frac{C_1}{b'} \sigma(\delta) (1-t).
  \end{multline*}
  By the inductive hypothesis, $r(\psi)\ge b' R(\psi)-C_2$, so 
  $$r(s)=r(\psi) + C_1 \sigma(\delta)(1-t)  \ge b' R(\psi) + C_1 \sigma(\delta) (1-t) - C_2\ge b' R(s)-C_2.$$
  Thus \eqref{eq:main criterion r} holds for all $s\in S$.  

  Let $s\in S$ and let
  $$y=e_{f(s)}(\sansP_\bR(s),r(s)+b' C_2),$$
  so that $d(y,\Omega_\bR(s))=b' C_2$.  By Corollary~\ref{cor:expMomentsChambers}, we have $\EE[\exp (b\dGamma(y))]\lesssim 1$, so
  $$\EE[\exp(b \dGamma(\Omega_\bR(s)))]\le e^{b b' C_2}\EE[\exp (b\dGamma(y))]\lesssim 1.$$
\end{proof}

In the rest of this paper, we will use this construction to bound the filling functions of $\Gamma$.

\section{Bootstrapping: Proof of Theorems~\ref{thm:mainThmDehn} and \ref{thm:mainThmLower}}\label{sec:DiscsSpheres}

In this section, we prove two of our main theorems, the exponential lower bound on $\FV^k_\Gamma$, and the quadratic bound on the Dehn function of lattices in higher-rank semisimple Lie groups.  

Both of these theorems use a Lipschitz extension theorem based on $\Omega_\bR$.  We can use $\Omega_\bR$ to extend an $L$--Lipschitz map $\alpha\from S^{n-1}\to \FD$ to a $CL$--Lipschitz map $\beta_0\from D^{n}\to X$.  This extension is made up of random flats, and by \eqref{eq:OmegaExpMoments}, it typically lies logarithmically close to $\FD$; by composing it with the retraction from Theorem~\ref{thm:fundDomain}, we obtain a map $\beta\from D^n\to X$.  The Lipschitz constant of the retraction grows exponentially, so the volume of $\beta$ is at most polynomial in $L$. 

When $n=2$, this gives a polynomial bound on the Dehn function of $\Gamma$, but the degree depends on the constants in \eqref{eq:OmegaExpMoments} and in Theorem~\ref{thm:fundDomain}.  Nonetheless, we can use this polynomial bound to prove a sharp bound.  Instead of composing the exponential retraction with $\beta_0$ to get $\beta$, we cut out the parts of $\beta_0$ that leave $\FD$, then use the polynomial filling inequality to fill in the resulting holes.  The area of the resulting filling is bounded in terms of an integral of a power of $\dGamma\circ \beta_0$, but \eqref{eq:OmegaExpMoments} implies that any such integral is bounded.  Since this uses the polynomial bound to prove a sharp bound, we call this \emph{the bootstrapping argument}.

The bootstrapping argument also lets us improve the super-polynomial bound in Section~\ref{sec:lower bound}.  In Lemma~\ref{lem:polyLower}, we constructed a sphere with large filling volume using the exponential retraction.  As above, we can use the polynomial Lipschitz extension theorem instead of the exponential retraction to obtain a sharp bound.

First, we prove the Lipschitz extension theorem.  

Recall that for all $t>0$, we defined $X(t)=\dGamma^{-1}([0,t])$ and $\blog t=\max\{1,\log t\}$.  We will prove the following proposition:
\begin{prop}\label{prop:lipExtend}
  There is an $\eta>0$ such that if $n < k = \rank\, X$, $L\ge 0$, and
  $\alpha\from S^{n-1}\to \FD$ is an $L$--Lipschitz map, then
  there is an extension $\beta_0\from D=D^n\to X(\eta\blog L)$
  with $\Lip \beta_0\lesssim L$.  Furthermore, if $b$ is as in
  Corollary~\ref{cor:expMomentsChambers}, then
  \begin{equation}\label{eq:LipExpMoment}
    \int_{D^n} e^{b\dGamma(\beta_0(x))}\; dx\lesssim 1.
  \end{equation}
\end{prop}

\begin{proof}
  We apply a technique of Gromov to construct Lipschitz extensions out of simplices \cite[3.5.D]{GroCC}.  

  When $L<1$, the proposition follows from the fact that $X$ is CAT(0).  We thus suppose that $L\ge 1$; in fact, we may suppose that $L=2^i-1$ for some $i\ge 1$.  By applying a scaling and a homeomorphism, it suffices to show that if $I^{n}(L)=[-L,L]^{n}$ and $\alpha\from \partial I^{n}(L)\to \FD$ is a $1$--Lipschitz map defined on the surface of the cube, then there is an extension $\beta_0\from I^n(L)\to X(\eta\blog L)$ with $\Lip \beta_0\lesssim 1$ such that
  \begin{equation}\label{eq:LipExpMomentRescaled}
    \int_{I^n(L)} e^{b\dGamma(\beta_0(x))}\; dx\lesssim L^n.
  \end{equation}

  We first subdivide $I^{n}(L)$ into dyadic cubes as in Figure~\ref{fig:dyadicSub}; for each cube $C$, the side length $s(C)$ is a power of 2 such that for all $x\in C$, 
  $$1\le s(C)\approx d(x,\partial I^n(L))+1.$$
  We barycentrically subdivide this complex into a simplicial complex $\tau$ and let $\iota\from \cF^0(\tau)\to \partial I^n(L)$ be closest-point projection.  If $C$ is a cube and $v\in C$, we have $d(v,\iota(v))\lesssim s(C)$, and if $v,w\in C$ are connected by an edge, then
  \begin{align*}
    d(\iota(v),\iota(w)) &\lesssim s(C)+d(v,w)+s(C)\\
                         &\lesssim s(C)\approx d(v,w),
  \end{align*}
  so $\iota$ is Lipschitz.

  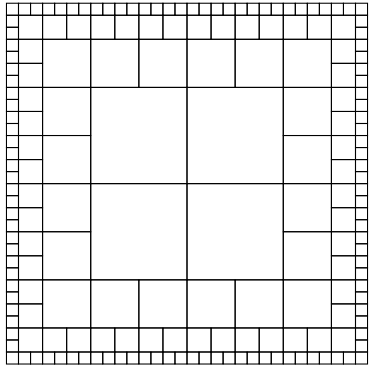
\begin{figure}
    \begin{tikzpicture}[scale=.16]
      \foreach \tk in {1,2,4,8} {
        \pgfmathsetmacro\endn{32/\tk-3}
        \pgfmathsetmacro\ty{16-2*\tk}
        \foreach \tn in {1,...,\endn} {
          \pgfmathsetmacro\tx{-16+\tn*\tk}
          \draw (\tx,\ty) rectangle (\tx+\tk,\ty+\tk); 
          \draw (-\tx,-\ty) rectangle (-\tx-\tk,-\ty-\tk); 
          \draw (-\ty,\tx) rectangle (-\ty-\tk,\tx+\tk); 
          \draw (\ty,-\tx) rectangle (\ty+\tk,-\tx-\tk); 
        }
      }
    \end{tikzpicture}
    \caption{\label{fig:dyadicSub} A subdivision of $[-2^i+1,2^i-1]$ into dyadic squares}
  \end{figure}
  
  This lets us define a map on the $(k-1)$--skeleton of $\tau$.  Let $f\from \tau^{(k-1)}\to S=\Delta_{\FD}^{(k-1)}$ be the simplicial map such that $f(v)=\alpha(\iota(v))$ for every $v\in \tau^{(0)}$.  This map sends each simplex $\delta\in \cF(\tau)$ of scale $s(\delta)$ to a  simplex $f(\delta)$ of scale $1$, so $\Lip f|_\delta\lesssim s(\delta)^{-1}$.  Furthermore, the bound on $\Lip\iota$ implies that $\sigma(f(\delta))\lesssim s(\delta)$ for all $\delta\in \cF(\tau)$.

  By Proposition~\ref{prop:OmegaBounds}, for any $\bD\in \mathcal U$ and any $\delta\in \cF(\tau)$ with $\dim \delta>0$, we have 
  $$\Lip(\Omega(\bD)\circ f|_\delta)\le \Lip \Omega(\bD)|_{f(\delta)}  \Lip f|_\delta\lesssim \sigma(f(\delta))s(\delta)^{-1}\lesssim 1.$$
  Furthermore, if $\bR\in \mathcal U$ is the random variable constructed in Section~\ref{sec:random cones}, then 
  $$\EE[\int_{I^n(L)}\exp(b \dGamma(\Omega(\bR)(f(x))))\;dx]=\int_{I^n(L)} \EE[\exp(b \dGamma(\Omega(\bR)(f(x))))]\;dx\lesssim L^n,$$
  so there is a $\bD_0\in \mathcal U$ such that if $g=\Omega(\bD_0)\circ f$, 
  $$\int_{I^n(L)}\exp(b \dGamma(g(x)))\;dx \lesssim L^n.$$

  The map $g$ need not agree with $\alpha$ on $\partial I^n(L)$, but for any vertex $v\in \partial I^n(L)$, we have $g(v)=(\Omega(\bD)\circ \alpha \circ \iota)(v)=\alpha(v)$.  Every point in $\partial I^n(L)$ is bounded distance from a vertex, so since $\Lip g\lesssim 1$, we have $d(\alpha(x),g(x))\lesssim 1$.  Let $\gamma\from \partial I^n(L)\times [0,1]\to X$ be the straight-line homotopy from $\alpha$ to $g$.  Since $\alpha$ and $g$ are Lipschitz and are a bounded distance apart, $\Lip \gamma\lesssim 1$, and $\dGamma$ is bounded on the image of $\gamma$.  Let $\beta_0$ be the concatenation of $g$ and $\gamma$.  Then $\Lip \beta_0\lesssim 1$ and $\beta_0$ satisfies \eqref{eq:LipExpMomentRescaled}.

  Finally, we check that there is an $\eta>0$ such that $\beta_0(I^n(L))\subset X(\eta \blog L)$.  Let $M=\max_x \dGamma(\beta_0(x))$ and let $x_0\in I^n(L)$ be such that $M=\dGamma(\beta_0(x_0))$.  Let $B\subset I^n(L)$ be the unit ball around $x_0$; we have $\vol B\approx 1$ and $\dGamma(\beta_0(x))\ge M-\Lip \beta_0$ for all $x\in B$.  Then
  $$\int_{I^n(L)} \exp(b \dGamma(\beta_0(x)))\; dx \ge \int_{B} \exp(b \dGamma(\beta_0(x)))\; dx \gtrsim e^{bM}.$$
  Combined with \eqref{eq:LipExpMomentRescaled}, this yields $e^{bM}\lesssim L^n$ and thus
  $M\lesssim \blog L$.
\end{proof}

By composing this with a retraction, we get an extension with image in $\FD$.  
\begin{cor}\label{cor:polyLipExtend}
  There is an $m>0$ such that if $n < k$, $L\ge1$, and
  $\alpha\from S^{n-1}\to \FD$ is a $L$--Lipschitz map, then there
  is an extension $\beta\from D^n\to \FD$ with
  $\Lip \beta\lesssim L^m$.
\end{cor}
\begin{proof}
  By Proposition~\ref{prop:lipExtend}, $\alpha$ has an extension $\beta_0\from D\to X(\eta\blog L)$ such that $\Lip\beta_0\lesssim L$.  Let $\rho\from X\to \FD$ be as in Theorem~\ref{thm:fundDomain}.  Then $\rho\circ \beta_0$ is an extension of $\alpha$ with image in $\FD$ and
  $$\Lip \rho\circ \beta_0 \lesssim e^{c \eta \blog L}L\lesssim L^{c\eta+1}.$$
\end{proof}

Using this bound, we can retract surfaces in $X$ to surfaces in $\FD$ with polynomial bounds.  

\begin{lemma}\label{lem:cuspReplacement}
  Let $n<k$ and let $m$ be as in Corollary~\ref{cor:polyLipExtend}.  Let $t, L\ge 1$, $Y=S^n(L)$ or $D^n(L)$ and let $\beta\from Y\to X(t)$ be a $1$--Lipschitz map.  There is a map $\kappa \from Y\to \FD$ such that
  \begin{enumerate}
  \item $\Lip \kappa\lesssim t^{m^n}$,
  \item for all $y\in \beta^{-1}(\FD)$, $\kappa(y)=\beta(y)$,
  \item \label{it:cuspRep bdd dist}
    for all $y\in Y$, $d(\kappa(y),\beta(y))\lesssim t^{m^n}$, and
  \item we have
    \begin{equation}\label{eq:integralVolume}
      \vol \kappa \lesssim \int_Y (\dGamma(\beta(y))+1)^{nm^n}\;dy.
    \end{equation}
  \end{enumerate}
\end{lemma}  
\begin{proof}
  We may suppose without loss of generality that $L$ is a positive
  integer.  Let $\tau$ be a triangulation of $Y$ such that the number of $n$-simplices is
  $\#\cF^n(\tau)\approx L^n$ and each simplex is bilipschitz equivalent
  to the unit simplex.  For instance, we can identify $D^n$ with
  the cube of side length $L$ by a bilipschitz map, then take $\tau$
  to be a subdivision of the unit grid. 

  Let $\tau_0$ be the union of all of the closed cells of $\tau$ that intersect $\beta^{-1}(\FD)$.  Let $\rho\from X\to X_0$ be the retraction from Theorem~\ref{thm:fundDomain} and define $\kappa=\rho\circ \beta$ on $\tau_0$.  In particular, $\beta(y)=\kappa(y)$ for all $y\in \beta^{-1}(\FD)$.  For each vertex $v\in \tau^{(0)}$, let $\kappa(v)$ be the closest point in $\FD$ to $\beta(v)$.  For each edge $e=\langle v,w\rangle\in \cF^1(\tau)$, either $e\subset \tau_0$, in which case $d(\kappa(v),\kappa(w))\lesssim 1$, or
  $$d(\kappa(v),\kappa(w))\lesssim \dGamma(\beta(v))+d(v,w)+\dGamma(\beta(w))\lesssim \dGamma(\beta(v))+1\le t+1.$$

  By Corollary~\ref{cor:polyLipExtend}, we can extend $\kappa$ to $\tau^{(1)}$ so that for each $e\in \cF^1(\tau)$ and all $x\in e$,
  $$\Lip \kappa|_{e}\lesssim (\dGamma(\beta(x))+1)^m.$$
  Applying Corollary~\ref{cor:polyLipExtend} inductively, we can extend $\kappa$ to $\tau^{(i)}$ so that for all $i$, all $\delta \in \cF^i(\tau)$, and all $x\in \delta$,
  \begin{equation}\label{eq:indPoly}
    \Lip \kappa|_{\delta}\lesssim (\dGamma(\beta(x))+1)^{m^i}.
  \end{equation}
  
  This extension satisfies the desired conditions.  By \eqref{eq:indPoly}, we have $\Lip \kappa\lesssim t^{m^n}$.  We constructed $\kappa$ so that $\kappa(y)=\beta(y)$ for all $y\in \beta^{-1}(\FD)$.  For any $y\in Y$, let $v\in \tau^{(0)}$ is a vertex such that $d(v,y)\lesssim 1$.  Then
  \begin{align*}
    d(\beta(y),\kappa(y))
    &\le d(\beta(y),\beta(v))+d(\beta(v),\kappa(v))+d(\kappa(v),\kappa(y))\\
    &\lesssim 1+\dGamma(\beta(v))+t^{m^n}\lesssim t^{m^n}.
  \end{align*}
  
  Finally, by \eqref{eq:indPoly}, we have
  $$\vol \kappa \lesssim \sum_{\delta\in \cF^n(\tau)} (d([\Gamma],\beta(\delta))+1)^{nm^n}\approx \int_Y (\dGamma(\beta(y))+1)^{nm^n}\;dy,$$
  as desired.
\end{proof}

The proof of Lemma~\ref{lem:cuspReplacement} uses Proposition~\ref{prop:lipExtend}, but we can also use Lemma~\ref{lem:cuspReplacement}, to improve Proposition~\ref{prop:lipExtend} and prove sharp bounds on the Dehn functions of higher-rank lattices.

\begin{proof}[{Proof of Theorem~\ref{thm:mainThmDehn}}]
  Since $\Gamma$ is not hyperbolic, we have $\delta_\Gamma(L)\gtrsim L^2$ for all $L\ge 1$.  It thus suffices to show that $\delta_{\FD}(L)\lesssim L^2$ for all $L\ge 1$.  

  Let $\alpha\from S^1\to \FD$ be a closed curve of length $L$, parametrized with constant speed and let $\beta_0$ be an extension as in Proposition~\ref{prop:lipExtend}.  By rescaling, we may view $\beta_0$ as a $1$--Lipschitz map $\beta_0 \from D=D^2(L)\to X(\eta\blog L)$ such that $\int_{D}e^{b\dGamma(\beta_0(y))}\;dy\lesssim L^2$ (see \eqref{eq:LipExpMomentRescaled}).  By Lemma~\ref{lem:cuspReplacement}, there is a map $\kappa\from D\to \FD$ such that $\kappa$ extends $\alpha$ and
  $$\area \kappa \lesssim \int_{D}(1+\dGamma(\beta_0(y)))^{2m^n}\;dy\lesssim \int_{D}e^{b\dGamma(\beta_0(y))}\;dy\lesssim L^2.$$
\end{proof}
\begin{remark}
  Note that, though the area of $\kappa$ is quadratic in $L$, its Lipschitz constant could be much larger than $L$.  It is an open question whether $\FD$ is Lipschitz $1$--connected.  Lipschitz $1$--connectedness implies a quadratic Dehn function.  It is not known whether the converse is always true, but Lipschitz $1$--connectedness for some solvable groups with quadratic Dehn functions is proved in \cite{CohenLipOne}
\end{remark}

\subsection{Proof of Thm.~\ref{thm:mainThmLower}}

Combining Lemma~\ref{lem:cuspReplacement} with Lemma~\ref{lem:polyLower} leads to exponential lower bounds on the $k$-dimensional filling functions of  irreducible lattices in semisimple groups of $\mathbb R$-rank $k$.  As it is known that filling functions of lattices are at most exponential, these bounds are sharp \cite[5.A$_7$]{GroAII}\cite{LeuPoly}.

\begin{thm}\label{prop:sharpLowerBounds}
  There is an $c>0$ such that for every sufficiently large
  $L\ge 2$, there is a sphere $\kappa\from S^{k-1}\to \FD$ such
  that $\vol^{k-1} \kappa\approx L^{k-1}$ and
  $$\delta^{k-1}_{\FD}(\kappa)\ge \FV^k_{\FD}(\fclass{\kappa})\gtrsim e^{c L}.$$
  Consequently, 
  $$\FV^k_{\Gamma}(V)\gtrsim e^{c V^{\frac{1}{k-1}}}$$
  and 
  $$\delta^{k-1}_{\Gamma}(V)\gtrsim e^{c V^{\frac{1}{k-1}}}.$$
\end{thm}
\begin{proof}
  The sphere $\alpha\from S^{k-1}\to X(\eta \log L)$ constructed in
  Lemma~\ref{lem:polyLower} is a round sphere of radius $cL$
  lying in a flat, and there is an $\omega>0$ such that 
  $$\FV^k_{X(L/4)}(\fclass{\alpha})\gtrsim e^{\omega L}.$$
  Let $\alpha'\from S^{k-1}(L)\to X(\eta \log L)$ be a rescaling of $\alpha$ and let $\kappa\from Y=S^{k-1}(L)\to \FD$ be the result of applying
  Lemma~\ref{lem:cuspReplacement} to $\alpha'$.  Then
  \begin{align*}
    \vol \kappa&\lesssim \int_{Y} (1+\dGamma(\alpha'(y)))^{(k-1)m^{k-1}}\;dy\\ 
               &\lesssim \int_{Y} e^{b\dGamma(\alpha'(y))}\;dy \\
               &=L^{k-1}\int_{S^{k-1}} e^{b\dGamma(\alpha(y))}\;dy \lesssim L^{k-1}.
  \end{align*}
  We claim that $\FV^k_{\FD}(\fclass{\kappa}) \gtrsim e^{\omega L}.$

  By Lemma~\ref{lem:cuspReplacement}.\ref{it:cuspRep bdd dist}, we have
  $$w:=\sup_{y\in Y} d(\alpha'(y),\kappa(y))\lesssim (\eta \log L)^{m^k}.$$  
  Let $r_0=\sup_{y\in \FD} \dGamma(y)$ and let $H\from S^{k-1}(L)\times [0,1]\to X(r_0+w)$ be the straight-line
  homotopy from $\alpha'$ to $\kappa$.  Then
  $$\Lip H \lesssim \max \{\Lip \alpha',\Lip \kappa, w\}\lesssim (\eta \log L)^{m^k},$$
  and $\vol^k H\lesssim L^{k-1}(\eta \log L)^{k m^k}.$ This is asymptotically smaller than $e^{\omega L}$, so if $L$ is sufficiently large, we may assume that
  $$\vol^k H\le \frac{\FV^k_{X(L/4)}(\fclass{\alpha})}{2}$$
  and $X(r_0+w) \subset X(L/4)$.  If so, then 
  $$\FV^k_{\FD}(\fclass{\kappa})\ge \FV^k_{X(L/4)}(\fclass{\alpha})-\vol^k H\gtrsim e^{\omega L}.$$
\end{proof}

\section{Additional tools from geometric measure theory}\label{sec:GMT tools}

The remainder of this paper is devoted to proving sharp polynomial bounds for the higher-dimensional filling volume functions of lattices.  This will require some additional tools from geometric measure theory, because the geometry of fillings is more complicated in higher dimensions.  In low dimensions, one can extend a closed curve to a disc using a dyadic partition of the square as in Figure~\ref{fig:dyadicSub} and Proposition~\ref{prop:lipExtend}.  In higher dimensions, however, the geometry and topology of an efficient filling of a cycle depend on the cycle.  In this section, we introduce some concepts and theorems from geometric measure theory that will help us to construct the desired fillings.

\subsection{Quasiconformal complexes and simplicial approximations}\label{sec:QCcomplexes}

We define a \emph{Riemannian simplicial complex} to be a simplicial complex in which each simplex has the metric of a Riemannian manifold with corners and each gluing map is a Riemannian isometry.  The \emph{standard metric} on a simplex is the metric of the unit Euclidean simplex, and the standard metric on a simplicial complex is the path metric such that each simplex is isometric to a standard simplex.  A Riemannian simplicial complex $\Sigma$ is a quasiconformal complex or \emph{QC complex} if there is a $c$ (called the \emph{QC constant} of $\Sigma$) and a \emph{scale function} $s\from \cF(\Sigma)\to \R$ on
 the faces of $\Sigma$ such that the Riemannian metric on each $d$--simplex $\delta$ is $c$--bilipschitz equivalent to the standard metric scaled by a factor of $s(\delta)$.  In particular, $\diam \delta\approx s(\delta)$.  We further require that if $\delta, \delta'\in \cF(\Sigma)$ and $\delta\cap \delta'\ne \emptyset$, then $s(\delta)\approx s(\delta')$.  (This is a slightly stronger condition than the definition in \cite{YoungHigherSol}, which did not impose a condition on adjacent simplices.)

The Federer-Fleming deformation theorem \cite{FedFlemNormInt} approximates Lipschitz cycles and chains by cellular cycles and chains.  For any $Z\subset Y$, let $\nbhd(Z)$ be the union of all of the closed cells of $Y$ that intersect $Z$.  We first recall a version of the approximation theorem proved in \cite{ECHLPT}.  We use the definitions of $\mass$ etc.\ given in Section~\ref{sec:fillingDefs}.
\begin{thm}[{\cite[10.3]{ECHLPT}}]\label{thm:FedFlemSimp}
  Let $\Sigma$ be a finite-dimensional QC complex with scale function $s\from \cF(\Sigma)\to \R$.  There is a constant $c$ depending on $\dim \Sigma$ and the QC constant of $\Sigma$ such that if $a\in \CLip_k(\Sigma)$ is a Lipschitz $k$--chain and $\partial a\in \Ccell_{k-1}(\Sigma)$, then there are $P(a)\in \Ccell_k(Y)$ and $Q(a)\in \CLip_{k+1}(Y)$ such that:
  \begin{enumerate}
  \item $\partial a=\partial P(a)$,
  \item $\partial Q(a) = a - P(a)$,
  \item $\mass P(a)\le c \cdot\mass(a)$,
  \item if $s(\delta)\le s_0$ for all $\delta\in \cF(\nbhd(\supp a))$, then $\mass Q(a)\le c s_0 \cdot\mass(a)$,
  \item $\supp P(a)\cup \supp Q(a)\subset \nbhd(\supp a)$,
  \item if $Z\subset Y$, then
  \begin{equation}\label{eq:FedFlemMassZBound}
    \mass_Z P(a)=\int_{P(a)} 1_Z\;dz\le C \mass_{\nbhd(Z)} a.
  \end{equation}
  \end{enumerate}  
\end{thm} 
The proof of this theorem is essentially the same as the argument in
Chapter~10.3 of \cite{ECHLPT}, so we provide only a sketch of the
argument.
\begin{proof}[Sketch of proof] 
  Let $d=\dim \Sigma$.  As in Chapter~10.3 of \cite{ECHLPT}, we construct $P(a)$ by constructing a sequence of locally Lipschitz maps $P_i\from \Sigma^{(d-i)}\to \Sigma^{(d-i)}$, $i=0,\dots, d-n-1$.  On each simplex $\delta$, the map $P_i$ sends a small ball $B_\delta\subset \delta$ homeomorphically to $\inter \delta$ and sends the rest of $\delta$ to $\partial \delta$.  Epstein, et al.\ show that by choosing the $B_\delta$ carefully (i.e., to avoid $\supp a$), we can ensure that $P_i$ fixes $\Sigma^{(d-i-1)}$ pointwise, that $(P_i\circ\dots \circ P_1)(\supp a)\subset \Sigma^{(d-i-1)}$, and that
  \begin{equation}\label{eq:PiP0 mass a}
    \mass (P_{i}\circ\dots \circ P_0)_\sharp(a)\lesssim \mass a
  \end{equation}
  for all $i$.  The argument in \cite{ECHLPT} is stated for simplicial complexes formed from unit simplices, but by a scaling argument, it remains true for QC complexes.

  The image $p=(P_{d-n-1}\circ\dots \circ P_0)_\sharp(a)$ is then a Lipschitz $n$--chain in $\Sigma^{(n)}$ with cellular boundary.  The degree of this chain on each $n$--cell is well-defined, and we let $P(a)$ be the cellular $n$--chain whose coefficient on $\delta$ is the degree of $p$ on $\delta$.  Then $\partial P(a)=\partial p=\partial a$ and
  $$\mass P(a)\le \mass p\lesssim \mass a.$$

  The bound \eqref{eq:FedFlemMassZBound} follows from a local version of \eqref{eq:PiP0 mass a}.  Roughly speaking, for any $\delta\in \cF^{\delta}$, the map $P_{i}\circ\dots \circ P_0$ expands the part of $a$ that lies in $\inter \delta$ by at most a constant factor.  To state this rigorously, let
  $$D_i(x)=\limsup_{\epsilon \to 0} \Lip(P_{i}\circ\dots \circ P_0|_{B_{x}(\epsilon)}),$$
  so that
  $$\mass (P_{i}\circ\dots \circ P_0)_\sharp(a)\lesssim \int_{a} D_i(x)^n\;dx.$$
  Then
  $$\int_{a} D_i(x)^n \one_{\inter \delta}(x)\;dx\lesssim \mass_{\inter \delta} a.$$
  
  If $Z\subset Y$, then
  $(P_{d-n-1}\circ\dots \circ P_0)^{-1}(Z)\subset \nbhd(Z)$, so
  \begin{align*}
    \mass_Z P(a)&\le \mass_Z p\\ 
                &\lesssim \sum_{\delta\in \cF(\nbhd(Z))} \int_{a} D_i(x)^n \one_{\inter \delta}(x)\;dx\\ 
                &\lesssim \mass_{\nbhd(Z)} a.
  \end{align*}
\end{proof}

We can also approximate chains with non-cellular boundaries, but the mass of the approximation depends on the boundary of the chain.  
\begin{cor}\label{cor:FFnoncell}
  For $a\in \CLip_n(\Sigma)$, let $\bar{P}(a)=P(a-Q(\partial a))\in \Ccell_n(\Sigma)$ and $\bar{R}(a)=Q(\partial a)\in \CLip_n(\Sigma)$.  If $s_0=\sup \{s(\delta)\mid \delta \subset \nbhd(\supp \partial a)\}$, then 
  $$\mass \bar{P}(a)\lesssim \mass a+s_0 \mass \partial a,$$
  $$\mass \bar{R}(a)\lesssim s_0 \mass \partial a.$$
  Furthermore, 
  $$\partial \bar{P}(a)=\partial(a-Q(\partial a))=P(\partial a)$$
  and thus
  $$\partial a=\partial \bar{P}(a)+\partial \bar{R}(a).$$
\end{cor}
We think of $\bar{P}(a)$ as a cellular approximation of $a$ and $\bar{R}(a)$ as an annulus connecting the boundaries of $a$ and $\bar{P}(a)$.  In fact, by making more careful estimates, one can construct approximations of $a$ with mass bounds independent of the mass of $\partial a$; see \cite{WhiteDeform}.  For our purposes, the estimates in Corollary~\ref{cor:FFnoncell} will suffice.

\subsection{Approximating spaces with finite Assouad-Nagata dimension}

The above theorems are particularly useful when working with metric spaces with finite Assouad--Nagata dimension, because such spaces can be approximated by QC complexes.
\begin{defn}\label{def:ANdim}
  Let $Y$ be a metric space and suppose that $D,s>0$.  A $D$--bounded covering of $Y$ is a collection $\mathcal{B}=(B_i)_{i\in I}$ of subsets of $Y$ such that $\diam B_i\le D$ for all $i\in I$.  The $s$--multiplicity of $\mathcal{B}$ is the smallest integer $n$ such that each subset $U\subset Y$ with $\diam U\le s$ meets at most $n$ elements of $\mathcal{B}$.  

  The \emph{Assouad--Nagata dimension} $\dimAN Y$ of $Y$ is the smallest integer $n$ such that there exists $c>0$ such that for all $s>0$, $Y$ has a $cs$--bounded covering with $s$--multiplicity at most $n+1$.  
\end{defn}
Note that if $Z\subset Y$, then $\dimAN Z\le \dimAN Y$, since any cover of $Y$ can be intersected with $Z$ to obtain a cover of $Z$ with the same diameter and multiplicity bounds.

Importantly, symmetric spaces of noncompact type have finite Assouad--Nagata dimension.
\begin{thm}[\cite{LangSch}]
  Any symmetric space $X$ of noncompact type satisfies $\dimAN X<\infty$.
\end{thm}

The following lemmas, which are based on the constructions used by Lang and Schlichenmaier to prove Theorem~1.5 of \cite{LangSch}, approximate a metric space with finite Assouad--Nagata dimension by a QC complex.

\begin{lemma}\label{lem:AssouadNagataCovers}
  If $Y$ is a space with Assouad--Nagata dimension $n$, then there are $a>0$ and $0<b<1$ (depending on $n$ and the constant $c$ in Definition~\ref{def:ANdim}) with the following property.  Suppose that $w \from Y\to [0,\infty)$ is a 1--Lipschitz function and let $Y_{>0}=w^{-1}((0,\infty))$.  There is a cover $\mathcal{B}=(B_i)_{i\in I}$ of $Y_{>0}$ such that 
  \begin{enumerate}  
  \item For all $i\in I$ we have $\diam B_i\le a \inf w(B_i)$.
  \item If $U\subset Y$ is a set such that $\diam U\le b \inf w(U)$, then $U$ meets at most $2(n+1)$ members of $\mathcal{B}$.
  \end{enumerate}
\end{lemma}

\begin{proof}
  Let $c$ be as in Definition~\ref{def:ANdim}.  For each $j\in \Z$, let $R_j=w^{-1}([2^j,2^{j+1}))$ and let $\mathcal{A}^j=(A^j_i)_{i\in I_j}$ be a $c2^j$--bounded covering of $R_j$ with $2^j$--multiplicity at most $n+1$.  Let $\mathcal{B}=\bigcup_j \mathcal{A}^j$ and let $I=\bigsqcup_j I_j$.  For all $i\in I_j$, we have $\diam A^j_i\le c2^j$ and $\inf w(A^j_i)\ge 2^j$, so we have $\diam B_i \le c \inf w(B_i)$ for all $i\in I$.

  Let $b=\frac{1}{2}$.  Let $U\subset Y$ be a set such that $\diam U\le b \inf w(U)$ and let $j\in \Z$ be such that $\inf w(U)\in [2^j,2^{j+1})$.  Then  $w(U)\subset [2^j, 2^{j+2})$, so $U\subset R_j\cup R_{j+1}$.  Since $\diam U< 2^{j}$, at most $2(n+1)$ elements of $\mathcal{A}^j\cup \mathcal{A}^{j+1}$ meet $U$.
\end{proof}

Given such a cover, we can construct a simplicial complex $\Sigma$ equipped with a Lipschitz map $g\from Y_{>0}\to \Sigma$.  The following lemma restates part of the proof of Theorem~5.2 of \cite{LangSch}.  There are no new ideas in the proof; the main changes are in terminology (i.e., QC complexes) and in the hypothesis that $Y$ is a path metric space, which is necessary to prove that $g$ is Lipschitz.  For $U\subset Y$, let $N_r(U)=\{y\in Y\mid d(y,U)<r)$ be the $r$--neighborhood of $U$.
\begin{lemma}\label{lem:AssouadNagataComplexes}
  Let $Y$ be a path metric space, let $n=\dimAN(Y)$, and suppose that $\mathcal{B}=(B_i)_{i\in I}$ is a cover satisfying Lemma~\ref{lem:AssouadNagataCovers}.  For each $i\in I$, let $s_i= \inf w(B_i)$.  There is an $0<\epsilon<\frac{1}{2}$ depending on $a$ and $b$ such that if $C_i=N_{\epsilon s_i}(B_i)$ for all $i$ and $\Sigma$ is the nerve of $\mathcal{C}=(C_i)_{i\in I}$ with vertex set $I$, then $\dim \Sigma\le 2n+1$ and there is a QC-complex structure on $\Sigma$ with scale function $s_\Sigma\from \cF(\Sigma)\to \R$,
  \begin{equation}\label{eq:sSigma def}
    s_\Sigma(\delta)=\min_{i\in \cF^0(\delta)} s_i
  \end{equation}
  and a map $g\from Y_{>0}\to \Sigma$ such that $\Lip g\lesssim_n 1$ and $g$ satisfies the following properties:
  \begin{enumerate}
  \item 
    For all $x\in Y_{>0}$, let $I_x=\{i\in I\mid x\in C_i\}$. Then, for all $x\in Y_{>0}$,
    \begin{equation}\label{eq:ANcomplexes nerve}
      g(x)\in \langle I_x \rangle,
    \end{equation}
    where $\langle I_x\rangle$ denotes the simplex with vertex set $I_x$.  
  \item For all $i\in I$ 
    \begin{equation}\label{eq:wx sim sSigma}
      \diam C_i\lesssim \inf w(C_i) \approx \sup w(C_i) \approx s_i.
    \end{equation}
  \end{enumerate}
\end{lemma}
\begin{proof}
  Implicit constants in this proof will all depend on $n$, $a$, and $b$.  Let $\epsilon=\frac{b}{2(b+1)}<\frac{1}{2}$.  Let $\Sigma$ be the nerve of $\mathcal{C}$ as above and let $r\from \Sigma\to (0,\infty)$ be the map that is linear on each simplex of $\Sigma$ and such that $r(i)=s_i$ for all $i\in I$.  We give $\Sigma$ the Riemannian metric obtained by rescaling the standard metric by the conformal factor $r$.  

  We first show that \eqref{eq:wx sim sSigma} holds.  Let $i\in I$ and let $x\in C_i$.  Then on one hand, $d(x, B_i)<\frac{s_i}{2}$, so $w(x)\ge s_i - d(x, B_i)\ge \frac{s_i}{2}$.  On the other hand, $\diam C_i\le (2\epsilon+a) s_i$, and $w(x)\le s_i+\diam C_i\le (2\epsilon+a+1)s_i$; i.e., $s_i\approx w(x)$.  Since this holds for all $x\in C_i$, this implies \eqref{eq:wx sim sSigma}.  

  Let $\delta\in \cF(\Sigma)$.  Since $\Sigma$ is the nerve of $\{C_i\}$, there is an $x\in Y_{>0}$ such that $\cF^0(\delta)\subset I_x$.  Then for all $y\in \delta$, 
  $$w(x)\approx \min_{i\in I_x} s_i\le r(y)\le \max_{i\in I_x} s_i\approx w(x).$$
  Since $s_\Sigma(\delta)=\min_{i\in I_x} s_i$, this implies that $\delta$ is quasiconformally equivalent to a standard simplex of diameter $s_\Sigma(\delta)$.

  Next, we bound the dimension of $\Sigma$.  Suppose that $x\in Y_{>0}$.  For each $i\in I_x$, let $x_i\in B_i$ be a point such that $d(x,x_i)<\epsilon s_i$, and let $Z_x=\{x_i\mid i\in I_x\}$; then
  $$\diam Z_x\le 2\epsilon \sup_{i\in I_x} s_i \le 2\epsilon( \diam Z_x+\inf w(Z_x))$$
  and thus $\diam Z_x\le b\inf w(Z_x)$.  By Lemma~\ref{lem:AssouadNagataCovers}, $Z_x$ meets at most $2n+2$ members of $\mathcal{B}$, so $|I_x|\le 2n+2$.  This holds for all $x$, so $\dim \Sigma\le 2n+1$.

  We construct $g\from Y_{>0}\to \Sigma$ by identifying $\Sigma$ with a subset of the Hilbert space $\ell_2(I)$.  For each simplex $\delta\in \cF(\Sigma)$, let $L_\delta\subset \ell_2(\cF^0(\delta))\subset \ell_2(I)$ be the set of non-negative functions $f$ such that $\sum_i f(i)=1$.  Each of the $L_\delta$ is isometric to a standard simplex rescaled by a factor of $\sqrt{2}$, and the path metric on the set $L_\Sigma=\bigcup_{\delta\in \cF(\Sigma)}$ is isometric (up to a constant scaling) to the standard metric on $\Sigma$.  We thus parametrize $\Sigma$ by a map $J\from L_\Sigma\to \Sigma$ such that $J(L_i)=i$ for all $i$.  For all $\delta\in \cF(\Sigma)$, we have $\Lip J|_{L_\delta}\lesssim s_\Sigma(\delta)$.

  For each $i$, define
  $$\sigma_i(x)=\max\{0,\epsilon s_i-d(x,B_i)\}.$$
  Let $\bar{\sigma}=\sum_i\sigma_i$.  Note that $\supp \sigma_i \subset C_i$ for all $i$.  Define a map $g_0\from Y_{>0}\to \ell_2(I)$ by letting
  $$g_0(x)=\left(\frac{\sigma_i(x)}{\bar{\sigma}(x)}\right)_{i\in I}$$
  for all $x\in Y_{>0}$.  We have $\supp g_0(x)\subset I_x$, so the image of this map lies in $L_\Sigma$, and if $g=J\circ g_0$, then $g$ satisfies \eqref{eq:ANcomplexes nerve}.

  Finally, we claim that $g$ is Lipschitz.  Since $Y$ is a path metric space, it suffices to prove this locally, i.e., that $\Lip g|_{C_i}\lesssim 1$ for all $i\in I$.  Let $x,y\in C_i$ and let $S,T\in \cF(\Sigma)$ be the minimal simplices containing $g(x)$ and $g(y)$ respectively.  Since $S$ and $T$ intersect, there is a $v\in L_{S\cap T}$ such that 
  $$\|g_0(x)- v\|_2+\|g_0(y)- v\|_2 \lesssim \|g_0(x)-g_0(y)\|_2.$$
  Since $\Lip J|_{L_S}\approx \Lip J|_{L_T} \approx s_i$, this implies
  $$d(g(x),g(y))\lesssim \|g_0(x)-g_0(y)\|_2 s_i.$$

  We thus proceed as in \cite[(5.13)]{LangSch}.  The points $g_0(x)$ and $g_0(y)$ differ in at most $4n+4$ coordinates, and for each such coordinate $k$, we have
  \begin{align*}
    \left|\frac{\sigma_k(x)}{\bar{\sigma}(x)} - \frac{\sigma_k(y)}{\bar{\sigma}(y)}\right| 
    &\le 
      \left|\frac{\sigma_k(x)}{\bar{\sigma}(x)} - \frac{\sigma_k(y)}{\bar{\sigma}(x)}\right| +
      \left|\frac{\sigma_k(y)}{\bar{\sigma}(x)} - \frac{\sigma_k(y)}{\bar{\sigma}(y)}\right| \\ 
    &\le
      \frac{1}{\bar{\sigma}(x)}(|\sigma_k(x)-\sigma_k(y)|+|\bar{\sigma}(x)-\bar{\sigma}(y)| )\\
    &\le
      \frac{(4n+5)d(x,y)}{\bar{\sigma}(x)}.
  \end{align*}
  Let $j\in I$ be such that $x\in B_j$.  Then 
  $$\bar{\sigma}(x)\ge \sigma_j(x)=\epsilon s_j\approx s_i,$$
  so 
  $$d(g(x),g(y))\lesssim \|g_0(x)-g_0(y)\|_2 s_i \le \frac{(4n+5)d(x,y)}{\epsilon s_j} s_i\approx d(x,y),$$
  as desired.
\end{proof}

\section{Euclidean filling functions below the rank}\label{sec:higherDims} 

\subsection{Filling with random flats}
In this section, we prove that the filling volume functions of $\Gamma$ satisfy polynomial bounds below the rank.  The proof is similar to the proof of Theorem~\ref{thm:mainThmDehn}:  First, we construct a filling from pieces of random apartments using $\Omega_\bR$.  This filling leaves the thick part of $\Gamma\backslash X$, so we use the polynomial extension theorem, Corollary~\ref{cor:polyLipExtend}, to replace the parts of the random filling that leave the thick part.  Since the bulk of the random filling lies in the thick part, we obtain a sharp estimate of the filling volume.

Let $w\from X\to \R$ be the function $w(x)=\dGamma(x)+1$.  This is 1--Lipschitz, and we use Lemmas~\ref{lem:AssouadNagataCovers} and \ref{lem:AssouadNagataComplexes} to construct a QC complex $\Sigma$ approximating $X$ and a map $g\from X\to \Sigma$.  Let $\cU\subset (\chxop)^{\cF(S)}$ be as in Proposition~\ref{prop:basepointChoices}.  We start by constructing a family of maps $f_\bD\from \Sigma^{(k-1)}\to X$, parametrized by an element $\bD\in \cU$, such that $f_\bD$ is a coarse inverse of $g$ for all $\bD$.  
\begin{lemma}\label{lem:random Sigma map}
  For any $\bD\in \cU$, there is a map $f_\bD\from \Sigma\to X$ such that:
  \begin{enumerate}
  \item $\Lip f_\bD\lesssim 1$, 
  \item for all $x\in \FD$, $d(x,f_\bD(g(x)))\lesssim 1$,
  \item if $\bR\in \cU$ is the random variable constructed in Section~\ref{sec:random cones} and if $b>0$ is as in Corollary~\ref{cor:expMomentsChambers}, then for all $y\in \Sigma^{(k-1)}$, 
    \begin{equation}\label{eq:expMomentsSigma}
      \EE[\exp(b \dGamma(f_\bR(y)))]\lesssim 1.
    \end{equation}
  \end{enumerate}
\end{lemma}
\begin{proof}




  As in Lemma~\ref{lem:AssouadNagataComplexes}, for each $i\in I$, let $s_i= \inf w(B_i)$, and let $s_\Sigma(\delta)=\min_{i\in \cF^0(\delta)} s_i$ be the scale function of $\Sigma$.  By the lemma, we have 
  $$\diam C_i \lesssim s_i \approx \inf w(C_i) \approx \sup w(C_i),$$
  and if $i$ and $i'$ are adjacent, then $s_i\approx s_{i'}$.

  For each $i\in I$, let $c_i$ be the corresponding vertex of $\Sigma$, and for each $x\in \FD$, let $c_x$ be the corresponding vertex of $S$.  For each $i\in I$, we have $d(\FD,C_i)\le \inf w(C_i)$.  We choose a point $r(i)\in \FD$ such that $d(r(i),C_i) \le \inf w(C_i)\approx s_i$.  If $i,i'\in \cF^0(\Sigma)$ are adjacent, then $C_i$ and $C_{i'}$ intersect, so
  \begin{equation}\label{eq:iotaSbound}
    d(r(i),r(i'))\lesssim s_i+\diam C_i+\diam C_{i'}+s_{i'}\approx s_i.
  \end{equation}
  Let $\iota\from \Sigma^{(k-1)}\to S$ be the simplicial map such that $\iota(c_i)=c_{r(i)}$ for all $i\in I$, and for all $\bD\in \cU$, let $f_\bD|_{\Sigma^{(k-1)}}=\Omega_\bD\circ \iota$.  For every $i\in I$, we have $f_\bD(c_i)=\Omega_\bD(\iota(c_i))=r(i)$.  

  For every $\delta\in \cF(\Sigma)$, the map $\iota$ sends $\delta$ to a simplex of $S$ with $\sigma(\iota(\delta))\lesssim s_\Sigma(\delta)$, so by Proposition~\ref{prop:OmegaBounds}, we have $\Lip f_\bD|_{\Sigma^{(k-1)}}\lesssim 1$.  Since $X$ is CAT(0) and $\dim \Sigma<\infty$, we can extend $f_\bD$ to all of $\Sigma$ so that $\Lip f_\bD\lesssim 1$.  

  This satisfies \eqref{eq:expMomentsSigma}, and it remains to prove that $d(x,f_\bD(g(x)))\lesssim 1$ for all $x\in \FD$.  Let $x\in \FD$ and let $i\in I$ be such that $x\in C_i$.  By Lemma~\ref{lem:AssouadNagataComplexes}, $d(g(x),c_i)\lesssim s_i \lesssim 1$, and
  $$d(x,f_\bD(c_i))=d(x,r(i)) \le \diam C_i+d(r(i),C_i) \lesssim s_i\lesssim 1.$$
  Consequently,
  \begin{multline*}
    d(x,f_\bD(g(x)))\le d(x,f_\bD(c_i))+d(f_\bD(c_i),f_\bD(g(x))) \\
    \le d(x,f_\bD(c_i))+\Lip(f_\bD) d(g(x),c_i) \lesssim 1.
  \end{multline*}
\end{proof}

We use this lemma and the results in Section~\ref{sec:GMT tools} to prove the following proposition.
\begin{prop}\label{prop:expMomentDistortion}
  If $n\le k-1$, $\beta$ is a Lipschitz $n$--chain in $X$ such that $\supp \partial \beta\subset \FD$, and $M=\mass \beta+\mass \partial \beta$, then there is a Lipschitz $n$--chain $\gamma\in \CLip_n(X)$ such that $\partial\beta=\partial \gamma$, $\mass \gamma\lesssim M$, and
  \begin{equation}\label{eq:expMomentDistortion}
    \int_{\gamma} e^{b \dGamma(x)}\; dx\lesssim M,
  \end{equation}
  where $b$ is as in Corollary~\ref{cor:expMomentsChambers}.
\end{prop}

\begin{proof}
  Let $\Sigma_0=\nbhd(g(\FD))$.  Since $w$ is bounded on $\FD$, $\Sigma_0$ is a subcomplex consisting of simplices with bounded diameter; let $s_0$ be such that $s(\delta)<s_0$ for every simplex $\delta\in \cF(\Sigma_0)$.  Note that for any $\bD\in \cU$ and any $x\in \Sigma_0$, we have $\dGamma(f_\bD(x))\lesssim s_0+1\lesssim 1$.

  Let $\bar{P}$ and $\bar{R}$ be as in Corollary~\ref{cor:FFnoncell}.  For $\bD\in \cU$, let $h\from \FD\times [0,1]\to X$ be the straight-line homotopy from $f_{\bD}\circ g|_{\FD}$ to $\id_{\FD}$, and let
  \begin{align*}
    \gamma_1&=\gamma_1(\bD)=(f_\bD)_\sharp(\bar{P}(g_\sharp(\beta)))\\
    \gamma_2&=\gamma_2(\bD)=(f_\bD)_\sharp(\bar{R}(g_\sharp(\beta)))\\
    \gamma_3&=\gamma_3(\bD)=h_\sharp(\partial \beta \times [0,1]).
  \end{align*}
  Let $\gamma(\bD)=\gamma_1(\bD)+\gamma_2(\bD)+\gamma_3(\bD).$  Then
  \begin{align*}
    \partial \gamma(\bD)&=(f_\bD)_\sharp(\partial \bar{P}(g_\sharp(\beta))+ \partial \bar{R}(g_\sharp(\beta)))+(\partial \beta- (f_{\bD}\circ g)_\sharp(\partial \beta))\\ 
                   &=(f_\bD)_\sharp(\partial g_\sharp(\beta))+\partial \beta- (f_{\bD}\circ g)_\sharp(\partial \beta)\\
                   &=\partial \beta.
  \end{align*}
  By Lemma~\ref{lem:random Sigma map} and Lemma~\ref{lem:AssouadNagataComplexes}, $f_\bD$, $g$, and $h$ are Lipschitz.  Corollary~\ref{cor:FFnoncell} implies that $\mass \gamma_i\lesssim \mass \beta+s_0 \mass\partial \beta\lesssim M$ for $i=1,2,3$.  We claim that there is a $\bD_0\in \cU$ such that
  \begin{equation}\label{eq:expMomentDistortionPieces}
    \int_{\gamma_i(\bD_0)} e^{b \dGamma(x)}\; dx\lesssim M.
  \end{equation}

  Let $\bR\in \cU$ be the random variable constructed in Section~\ref{sec:random cones}.  By Lemma~\ref{lem:random Sigma map}, 
  \begin{align*}
    \EE\biggl[\int_{\gamma_1} e^{b \dGamma(x)}\; dx\biggr]
    &\le \EE\biggl[(\Lip f_{\bR})^n \int_{\bar{P}(g_\sharp(\beta))}e^{b \dGamma(f_\bR(y))}\; dy\biggr]\\ 
    &\lesssim \int_{\bar{P}(g_\sharp(\beta))} \EE[e^{b \dGamma(f_\bR(y))}]\; dy]
    \lesssim M.
  \end{align*}
  Consequently, there is a $\bD_0\in \cU$ that satisfies \eqref{eq:expMomentDistortionPieces} for $i=1$.  

  For any $\bD\in \cU$, the two chains $\gamma_2$ and $\gamma_3$ are both supported on bounded neighborhoods of $\FD$.  For $\gamma_2$, this follows from the fact that $g(\supp \partial \beta)\subset g(\FD)$ and thus $\supp \bar{R}(g_\sharp(\beta))\subset \Sigma_0$.  For $\gamma_3$, this follows from the fact that the image of $h$ lies in a bounded neighborhood of $\FD$.  Thus, \eqref{eq:expMomentDistortionPieces} is satisfied for $i=2,3$ as well.  Consequently, $\gamma=\gamma(\bD_0)$ satisfies the lemma.
\end{proof}

Proposition~\ref{prop:expMomentDistortion} produces a filling $\gamma$ such that only an exponentially small fraction of $\gamma$ lies outside $X(t)$.  We use techniques similar to those of Lemma~\ref{lem:cuspReplacement} to retract those parts of $\gamma$ to $\FD$.

\begin{prop}\label{prop:bootstrapping}
  Let $m>0$ be as in Corollary~\ref{cor:polyLipExtend}.  Let $n\le k-1$ and let $\gamma$ be a Lipschitz $n$--chain in $X$ such that $\supp \partial \gamma\subset \FD$.  There is a Lipschitz $n$--chain $\psi\in \CLip_n(\FD)$ such that $\partial\psi=\partial \gamma$ and
  \begin{equation}\label{eq:bootstrapping}
    \mass \psi\lesssim \mass \partial \gamma + \int_{\gamma} (\dGamma(x)+1)^{nm^n}\; dx.
  \end{equation}
\end{prop}
\begin{proof}
  Let $\tau$ be a finite-dimensional simplicial complex with the standard metric that is quasi-isometric to $X$ by Lipschitz maps $a\from X\to \tau$, $z\from \tau \to X$.  We construct such a $\tau$ by taking $w\equiv 1$ in Lemma~\ref{lem:AssouadNagataComplexes}.  (One can take $\tau$ to be a triangulation of $X$ and $a$ and $z$ to be the identity map, but constructing a triangulation with the necessary metric properties requires some technical sophistication; see \cite{BoissonnatDyerGhosh}.)

  We construct a map $q \from \tau \to X$ inductively.  First, we define $q$ on $\tau^{(k-1)}$.  Let $\rho\from X\to \FD$ be the closest-point projection, as in Theorem~\ref{thm:fundDomain} and define $q(v)=\rho(z(v))$ for every vertex $v\in \tau^{(0)}$.  If $0<i\le k-1$ and we have defined $q$ on $\tau^{(i-1)}$, we extend $q$ to $\tau^{(i)}$ by using Corollary~\ref{cor:polyLipExtend} to extend $q$ to each $i$--simplex.  This ensures that $q(\tau^{(k-1)})\subset \FD$.  

  The Lipschitz constant of $q$ grows polynomially with distance to $\FD$.    For any two adjacent vertices $v$ and $v'\in \tau^{(0)}$, we have 
  $$d(q(v),q(v'))\le \dGamma(z(v))+d(z(v),z(v'))+\dGamma(z(v'))\lesssim \dGamma(z(v))+1,$$
  Each time we apply Corollary~\ref{cor:polyLipExtend}, the Lipschitz constant is raised to the power of $m$, so for each simplex $\delta\in \cF^i(\tau)$ with $i\le k-1$, 
  \begin{equation}\label{eq:lipP Ddelta}
    \Lip q|_{\delta} \lesssim (\min \dGamma(z(\delta))+1)^{m^i}.
  \end{equation}
  Consequently, $\Lip q|_{\nbhd(a(X_0))^{(k-1)}}\lesssim 1$.

  For $i>k-1$, if $q$ is defined on $\tau^{(i-1)}$, we extend $q$ to $\tau^{(i)}$ by geodesic coning.  That is, for each simplex $\delta\in \cF^i(\tau)$ with barycenter $c_\delta$, we choose a vertex $v\in\delta$, then define $q$ on $\delta$ so that $q(c_\delta)=q(v)$ and so that for any $w\in \partial \delta$, $q$ sends the line segment between $c_\delta$ and $w$ to the geodesic from $q(v)$ to $q(w)$.  Note that this geodesic need not be contained in $\FD$.  

  The Lipschitz constant of $q$ is large on parts of $\tau$ that lie far from $\Gamma$, but $\Lip q|_{\nbhd(a(X_0))}\lesssim 1$ and thus $\Lip (q\circ a|_{\FD})\lesssim 1$.  In fact, $q|_{\nbhd(a(\FD))}$ is a Lipschitz quasi-inverse to $a|_{\FD}$, so if $h\from \FD\times [0,1]\to X$ is the straight-line homotopy from $q\circ a|_{\FD}$ to $\id_{\FD}$, then $\Lip h\lesssim 1$.

  As in Proposition~\ref{prop:expMomentDistortion}, let $\psi_1=q_\sharp(\bar{P}(a_\sharp(\gamma)))$, $\psi_2=q_\sharp(\bar{R}(a_\sharp(\gamma)))$, $\psi_3=h_\sharp(\partial \gamma \times [0,1])$.  Let $\psi=\rho_\sharp(\psi_1+\psi_2+\psi_3)$; as before, $\partial \psi=\partial \gamma$.

  We claim that $\psi$ satisfies \eqref{eq:bootstrapping}.  Since $a$, $z$, and $h$ are Lipschitz, we have $\mass \psi_2\lesssim \mass \partial \gamma$ and $\mass \psi_3\lesssim \mass \partial \gamma$.  Furthermore, $\supp \psi_2\cup \supp \psi_3$ is contained in a bounded neighborhood of $\FD$, so $\mass \rho_\sharp(\psi_2+\psi_3)\lesssim \mass \partial \gamma$. 

  It remains to bound $\mass \rho_\sharp(\psi_1)$.  Since $\bar{P}(a_\sharp(\gamma))$ is a cellular $n$--chain, we have $\supp \psi_1\subset q(\tau^{(k-1)})\subset \FD$, so $\rho_\sharp(\psi_1)=\psi_1$.  

  For $i\in \N$, let $Z_i=\{w\in \tau \mid \dGamma(z(w))\in [i-1,i]\}$ and let $Y_i=\nbhd(Z_i)$.  For all $y\in \tau$, let $\chi(y)=\sum_i \one_{Y_i}(y)$ be the number of $Y_i$'s that $y$ is contained in; since $z$ is Lipschitz, only boundedly many of the $Z_i$'s intersect any simplex of $\tau$, so $\chi(y)\approx 1$ for all $y$.  We have $\dGamma(z(y))+1\approx i$ for all $y\in Y_i$, so for any $n$--chain $\lambda\in \CLip_n(\tau)$ and any $C>0$, 
  \begin{multline}
    \int_{\lambda} (\dGamma(z(y))+1)^{C}\;dy \lesssim \sum_{i=1}^\infty i^{C} \mass_{Z_i}(\lambda)\le \sum_{i=1}^\infty i^{C} \mass_{Y_i}(\lambda) \\
    \approx_C \int_{\lambda} (\dGamma(z(y))+1)^{C}\chi(y)\;dy \approx_C \int_{\lambda} (\dGamma(z(y))+1)^{C}\;dy.
  \end{multline}
  That is,
  \begin{equation}\label{eq:integralSum tau}
    \int_{\lambda} (\dGamma(z(y))+1)^{C}\;dy \approx_C \sum_{i=1}^\infty i^{C} \mass_{Z_i}(\lambda)\approx_C \sum_{i=1}^\infty i^{C} \mass_{Y_i}(\lambda).
  \end{equation}
  Letting $C=nm^n$ and letting $F(y)=(\dGamma(z(y))+1)^C$, we find
  \begin{align*}
    \mass \psi_1 & \stackrel{\eqref{eq:lipP Ddelta}}{\lesssim} \int_{\bar{P}(a_\sharp(\gamma))} F(y)\;dy \\
    & \stackrel{\eqref{eq:integralSum tau}}{\approx}  \sum_{i=1}^\infty i^C \mass_{Z_i} \bar{P}(a_\sharp\gamma)\\
    & \stackrel{\eqref{eq:FedFlemMassZBound}}{\lesssim} \sum_{i=1}^\infty i^C \mass_{Y_i}(a_\sharp(\gamma)-Q(a_\sharp(\partial \gamma)))\\
    & \stackrel{\eqref{eq:integralSum tau}}{\lesssim} \int_{a_\sharp(\gamma)}F(y)\;dy+\int_{Q(a_\sharp(\partial \gamma))}F(y)\;dy.
  \end{align*}
  Then
  \begin{align*}
    \int_{a_\sharp(\gamma)} F(y) \;dy\le (\Lip a)^n \int_{\gamma} (1+\dGamma(z(a(x))))^C\;dx\lesssim\int_\gamma (1+\dGamma(x))^C\;dx.
  \end{align*}
  To bound the second term, note that $\supp Q(a_\sharp(\partial \gamma))\subset \nbhd(a(\FD))$ and 
  $$\max \dGamma(z(a(\FD))) \lesssim 1,$$
  so
  $$\int_{Q(a_\sharp(\partial \gamma))} F(y)\;dy\lesssim \mass Q(a_\sharp(\partial \gamma))\lesssim \mass \partial \gamma.$$
  Combining the estimates on $\psi_1, \psi_2$, and $\psi_3$, we obtain \eqref{eq:bootstrapping}. 
\end{proof}

\subsection{Sharp bounds: Proof of Theorem~\ref{thm:mainThmUpper} and Theorem~\ref{thm:disdim}}
Finally, we use the results in the previous section to prove sharp bounds on higher-order filling invariants of irreducible lattices.  These invariants are at least Euclidean in dimensions below the rank.
\begin{thm}\label{thm:euclower} 
  Let $\Gamma$ be an irreducible, non-uniform lattice in a semisimple Lie group $G$ acting on a symmetric space $X=G/K$ of rank $k$. Then the filling invariants of dimension less than $k$ have Euclidean lower bounds:
  $$\FV_{\Gamma}^n(V)\gtrsim V^{\frac{n}{n-1}} \ \ \textup{if}\ \ 2\leq  n < k,$$
  $$\delta_{\Gamma}^{n-1}(V)\gtrsim V^{\frac{n}{n-1}} \ \ \textup{if}\ \ 2\leq n<k.$$
 \end{thm}
 \begin{proof} 
   A theorem of Mostow asserts that there are many closed maximal flats in $X\backslash \Gamma$. (See \cite{MoLOCSYM}, Lemma 8.3, 8.3$'$) Pick one such flat and its universal cover $F$ in $X_0$.  Then the restriction of the orthogonal projection $\pi:X\to F$ to $X_0$ is $1$-Lipschitz, and the claim follows from the solution of the Euclidean isoperimetric problem in $F$.
 \end{proof}

Our results prove that these bounds are sharp.
\begin{thm}
  Let $2\le n<k=\rank X$ and let $\alpha\in \CLip_{n-1}(\FD)$ be an $(n-1)$--cycle.  Then
  $$\FV_{\FD}(\alpha)\lesssim \mass \alpha + \FV_{X}(\alpha) \lesssim (\mass \alpha)^{\frac{n}{n-1}}.$$
  It follows that $\FV^{n}_\Gamma(V)\lesssim V^{\frac{n}{n-1}}$ and that $\FD$ is undistorted up to dimension $k-1$.
\end{thm}
\begin{proof}
  Let $\epsilon>0$ and let $\beta\in \CLip_{n}(X)$ such that $\partial \beta=\alpha$ and $\mass \beta \le \FV_X(\alpha)+\epsilon$.  By Proposition~\ref{prop:expMomentDistortion} and Proposition~\ref{prop:bootstrapping}, there are chains $\gamma\in \CLip_n(X)$ and $\psi\in \CLip_n(\FD)$ such that $\partial \beta=\partial \gamma=\partial \psi$ and
  \begin{align*}
    \mass \psi
    &\lesssim \mass(\partial \gamma) + \int_{\gamma} (\dGamma(x)+1)^{nm^n}\; dx\\
    &\lesssim \mass(\partial \beta) + \int_{\gamma} \exp b \dGamma(x)\; dx\\
    &\lesssim \mass \alpha+ \FV_X(\alpha)+\epsilon.
  \end{align*}
  Then $\FV_{\FD}(\alpha)\le \mass\psi$.  Letting $\epsilon$ go to zero, we have $\FV_{\FD}(\alpha)\le \mass \alpha+\FV_{X}(\alpha)$, and by Theorem~\ref{thm:CAT0 fillings}, $\FV_{X}(\alpha)\lesssim (\mass \alpha)^{\frac{n}{n-1}}$.  
\end{proof}
This implies Theorem~\ref{thm:mainThmUpper}.  The only remaining thing to prove is that $X_0$ is distorted in dimension $k$, which is a direct consequence of Theorem~\ref{thm:mainThmLower}.

\bibliographystyle{amsalpha}
\bibliography{higherArith}
\end{document}